\documentclass[reqno,11pt]{amsart}        
\usepackage{amsfonts,amsmath,amssymb,amsthm,colordvi,mathrsfs,graphicx}
\usepackage{caption,subcaption,float}
\usepackage[utf8]{inputenc}
\usepackage{mathrsfs}
\usepackage{geometry}
\usepackage{enumerate}

\usepackage{amsmath, amssymb, amsthm, geometry, hyperref}

\usepackage{tabularx}
\usepackage{tabulary}
 
\hypersetup{
  colorlinks=true,
  linkcolor=blue,
  citecolor=blue,
  urlcolor=blue
}
\usepackage[all,knot,poly]{xy}
\usepackage{tikz-cd}
\usepackage{tikz}
\usepackage{verbatim}
\usetikzlibrary{arrows}
\usetikzlibrary{knots}
\tikzset{every path/.style={line width=0.4pt},every node/.style={transform shape,knot crossing,inner sep=1.5pt},>=triangle 60,text node/.style={rectangle,transform shape=false,black}}

 \usetikzlibrary{decorations.markings, arrows.meta}

\theoremstyle{plain}      
\newtheorem{thm}{Theorem}[section]     
\newtheorem{theorem}[thm]{\bf Theorem}     
\newtheorem{corollary}[thm]{\bf Corollary}     
\newtheorem{lemma}[thm]{\bf Lemma}     
\newtheorem{proposition}[thm]{\bf Proposition}

\theoremstyle{remark}      
\newtheorem{example}[thm]{Example} 
 
\newtheorem{remark}[thm]{Remark} 
     
\theoremstyle{definition}      
\newtheorem{definition}[thm]{Definition}     

\subjclass[2020]{14B05, 14B10, 32G20, 14B07}
\keywords{Equigeneric deformations, isolated singularities, semiregularity, logarithmic normal bundles, Torelli, nodal curves.}

\begin{document}


\title{TBA}

\author{Mounir Nisse}
 
\address{Mounir Nisse\\
Department of Mathematics, Xiamen University Malaysia, Jalan Sunsuria, Bandar Sunsuria, 43900, Sepang, Selangor, Malaysia.
}
\email{mounir.nisse@gmail.com, mounir.nisse@xmu.edu.my}
\thanks{}
\thanks{This research is supported in part by Xiamen University Malaysia Research Fund (Grant no. XMUMRF/ 2020-C5/IMAT/0013).}


\title{Logarithmic Equigeneric Smoothing and Maximal Nodal Degenerations}


\maketitle

\begin{abstract}
We study equigeneric deformations of singular curves in degenerating families of
surfaces using logarithmic deformation theory. Replacing classical unobstructedness
by logarithmic semiregularity, we prove that reduced curves on normal crossings
surfaces deform to nodal curves on nearby smooth fibers with the maximal number of
nodes allowed by the $\delta$--invariant. We further establish a unified logarithmic
equigeneric smoothing theorem for planar $A_k$ singularities, showing that higher
singularities split into the maximal number of ordinary nodes under generic equigeneric
deformations. These results extend classical Severi theory to logarithmic settings and
clarify the role of higher singularities as boundary phenomena. Applications to
enumerative geometry are discussed via logarithmic degeneration formulas and tropical
curve counts.
\end{abstract}

\section{Introduction}

The deformation theory of singular curves occupies a central position in algebraic
geometry, connecting the study of linear systems on surfaces, moduli spaces of curves,
and enumerative geometry. A guiding paradigm is classical Severi theory, which concerns
families of curves in a fixed linear system $|L|$ on a smooth projective surface $S$.
If $C\in|L|$ is a reduced curve, its arithmetic genus satisfies
\[
p_a(C)=\frac{1}{2}(L\cdot(L+K_S))+1,
\]
and if $\widetilde C$ denotes the normalization of $C$, then
\[
p_a(C)=g(\widetilde C)+\delta(C),
\qquad
\delta(C)=\sum_{p\in\mathrm{Sing}(C)}\delta_p.
\]
The invariant $\delta(C)$ measures the total genus drop due to singularities. The
classical Severi problem asks whether the locus $V_\delta(|L|)$ of curves in $|L|$ with
exactly $\delta$ nodes is nonempty, irreducible, and of the expected dimension
\[
\dim |L|-\delta,
\]
and whether its general point represents a nodal curve. Harris’ solution of the
Severi problem \cite{HarrisSeveri} established that, on smooth surfaces, equigeneric
families of curves are generically nodal and that nodes realize the maximal singular
behavior compatible with the genus formula.

When one moves beyond smooth surfaces or allows more complicated singularities, this
classical picture becomes substantially more delicate. In degenerating families of
surfaces, the ambient geometry itself contributes obstructions to deformation, and
standard arguments based on generic smoothness or transversality fail. Moreover,
singularities such as cusps or higher $A_k$ points are not generic in equigeneric
families, even on smooth surfaces. An ordinary cusp, analytically given by
$y^2=x^3$, satisfies $\delta=1$ just like a node, but its local deformation space has
dimension two rather than one, reflecting its greater rigidity. More generally, a
planar $A_k$ singularity
\[
y^2=x^{k+1}
\]
has local $\delta$--invariant
\[
\delta=\left\lfloor\frac{k}{2}\right\rfloor,
\]
and its equigeneric deformation space has dimension $\delta$, while the full local
deformation space has dimension $k$ \cite{GreuelLossenShustin}. These facts already
suggest that equigeneric deformation theory must be refined in order to control higher
singularities.

A first major goal of this work is to show that logarithmic geometry provides the
appropriate framework for extending Severi--type results to degenerating surfaces.
Following the ideas of Bloch \cite{Bloch} and Friedman \cite{Friedman}, one endows a
degeneration $\pi:\mathcal X\to\Delta$ with its divisorial logarithmic structure along
the special fiber $\mathcal X_0$. In this setting, the deformation theory of a curve
$C\subset\mathcal X_0$ is governed not by the usual normal bundle
$N_{C/\mathcal X_0}$, but by the logarithmic normal bundle
$N^{\log}_{C/\mathcal X_0}$. Infinitesimal logarithmic embedded deformations are
parametrized by
\[
H^0\!\left(C,N^{\log}_{C/\mathcal X_0}\right),
\]
with obstructions lying in
\[
H^1\!\left(C,N^{\log}_{C/\mathcal X_0}\right).
\]
Logarithmic semiregularity is the condition that the natural map
\[
H^1\!\left(C,N^{\log}_{C/\mathcal X_0}\right)\longrightarrow
H^2(\mathcal X_0,\mathcal O_{\mathcal X_0})
\]
is injective. One of the key observations developed in this work is that this condition
precisely guarantees the absence of logarithmic obstructions to smoothing.

Within this framework, we prove a logarithmic semiregular nodal smoothing theorem for
curves on normal crossings surfaces. Roughly speaking, it asserts that if $C$ is
logarithmically semiregular and if the global--to--local logarithmic deformation map
\[
H^0\!\left(C,N^{\log}_{C/\mathcal X_0}\right)\longrightarrow
\bigoplus_{p\in\mathrm{Sing}(C)}T^1_{C,p}
\]
is surjective, then $C$ deforms to a nodal curve on nearby smooth fibers with exactly
$\delta(C)$ nodes. This result recovers classical Severi theory when $\mathcal X_0$ is
smooth and extends it to degenerations in a conceptually transparent way.

A second central contribution of this paper is the formulation of a unified logarithmic
equigeneric smoothing theorem for planar $A_k$ singularities. We show that, under the
same logarithmic semiregularity and surjectivity assumptions, a curve whose
singularities are all of type $A_k$ admits an equigeneric logarithmic deformation in
which each $A_k$ singularity splits into $\lfloor k/2\rfloor$ ordinary nodes. In
particular, the total number of nodes on the general fiber equals
\[
\delta(C)=\sum_{p\in\mathrm{Sing}(C)}\left\lfloor\frac{k_p}{2}\right\rfloor,
\]
and this number is maximal by invariance of the arithmetic genus. This theorem
simultaneously generalizes the nodal case ($A_1$), the cuspidal case ($A_2$), and all
higher $A_k$ cases, making precise the idea that higher singularities appear as boundary
points of equigeneric strata.

The scope of the paper further extends to curves and surfaces inside threefolds. In
these settings, nodal singularities are no longer generic even in equigeneric families.
For a curve $C$ in a threefold, or a surface $S$ in a threefold, ordinary nodes and
ordinary double points are rigid singularities that persist only under equisingular
constraints. We formulate and prove deformation theorems showing that, under
unobstructedness and surjectivity of the appropriate normal bundle maps, one can realize
the maximal number of such singularities compatible with global deformation theory.
These results place nodal phenomena in threefolds on the same conceptual level as
cuspidal phenomena on surfaces.

Finally, we connect these deformation--theoretic results to enumerative geometry.
Logarithmic semiregularity ensures that Severi counts remain enumerative in degenerating
families, without virtual corrections. Using logarithmic degeneration formulas
\cite{Li,Siebert}, we express Severi degrees on smooth fibers as sums of contributions
from boundary strata on the special fiber. Tropical geometry provides a combinatorial
interpretation of these formulas: nodes correspond to bounded edges of tropical curves,
while $A_k$ singularities correspond to higher--valence vertices that resolve into
collections of bounded edges under equigeneric smoothing \cite{Mikhalkin,NishinouSiebert}.

Taken together, the results of this work provide a unified framework in which classical
Severi theory, logarithmic deformation theory, and tropical geometry interact
coherently. The main new insight is that maximal nodal behavior is governed by numerical
invariants such as $\delta$, while the existence, persistence, and enumerativity of
such behavior are controlled by logarithmic semiregularity and global deformation
theory. This perspective not only extends Severi theory to singular and degenerating
settings, but also clarifies the geometric meaning of higher singularities as boundary
phenomena in equigeneric families.


\section{Deformations of the Pair $(\mathcal X_0, C)$}

Let $\pi : \mathcal X \to \Delta$ be a flat family over a disk, with special fiber
$\mathcal X_0$ and smooth general fiber $\mathcal X_t$ for $t \neq 0$.
Let $C \subset \mathcal X_0$ be a reduced curve.
The goal is to understand when $C$ can be deformed to a nodal curve
$C_t \subset \mathcal X_t$, and to determine the maximal number of nodes
that can appear on such a deformation.

A natural starting point is to study deformations of the embedding
$C \hookrightarrow \mathcal X_0$.
Infinitesimal embedded deformations are controlled by the normal sheaf
\(
N_{C/\mathcal X_0},
\)
while obstructions lie in $H^1(C, N_{C/\mathcal X_0})$.
Vanishing or partial control of this cohomology group gives sufficient
conditions for lifting $C$ to nearby fibers $\mathcal X_t$.
If $\mathcal X$ is smooth along $C$, one may compare deformations of
$C$ inside $\mathcal X_0$ with deformations inside the total space
$\mathcal X$ using the exact sequence relating
$N_{C/\mathcal X_0}$ and $N_{C/\mathcal X}$.


\subsection*{The Normal Bundle Exact Sequence for a Curve in a Degeneration}

Assume further that $C$ is a local complete intersection both in $\mathcal X_0$ and
in $\mathcal X$.
We write $i:C\hookrightarrow\mathcal X_0$ and $j:C\hookrightarrow\mathcal X$ for the
natural embeddings.
The goal is to relate the normal bundle of $C$ in the special fiber $\mathcal X_0$
to the normal bundle of $C$ in the total space $\mathcal X$.

\subsection*{Exact Sequence of Normal Bundles}

Since $\mathcal X_0$ is a Cartier divisor in $\mathcal X$, there is a short exact
sequence of normal bundles on $C$
\[
0 \longrightarrow N_{C/\mathcal X_0}
\longrightarrow N_{C/\mathcal X}
\longrightarrow N_{\mathcal X_0/\mathcal X}\big|_C
\longrightarrow 0.
\]
Since $\mathcal X_0$ is a fiber of $\pi$, its normal bundle in $\mathcal X$ is
trivial:
\(
N_{\mathcal X_0/\mathcal X} \simeq \mathcal O_{\mathcal X_0}.
\)
Restricting to $C$, one obtains the fundamental exact sequence
\[
0 \longrightarrow N_{C/\mathcal X_0}
\longrightarrow N_{C/\mathcal X}
\longrightarrow \mathcal O_C
\longrightarrow 0.
\tag{$\ast$}
\]
%
%
The extension class of the sequence $(\ast)$ is precisely the Kodaira--Spencer class
of the deformation $\pi:\mathcal X\to\Delta$ restricted to $C$.
Geometrically, this sequence measures how infinitesimal deformations of $C$ inside
the total space $\mathcal X$ decompose into deformations tangent to the special
fiber and deformations transverse to it.

Taking cohomology yields the long exact sequence
\[
0 \to H^0(C,N_{C/\mathcal X_0})
\to H^0(C,N_{C/\mathcal X})
\to H^0(C,\mathcal O_C)
\to H^1(C,N_{C/\mathcal X_0})
\to H^1(C,N_{C/\mathcal X})
\to 0,
\]
where the connecting homomorphism
\(
H^0(C,\mathcal O_C)\to H^1(C,N_{C/\mathcal X_0})
\)
is induced by the Kodaira--Spencer map.
In particular, if $H^1(C,N_{C/\mathcal X_0})=0$, then every infinitesimal deformation
of $C$ in $\mathcal X_0$ extends uniquely to an infinitesimal deformation of $C$ in
the total space $\mathcal X$, and $C$ deforms to nearby fibers $\mathcal X_t$.

\subsection*{Relation with the Cotangent Complex}

The exact sequence $(\ast)$ is dual to the distinguished triangle of cotangent
complexes
\[
L_{\mathcal X_0/\mathcal X}\big|_C \longrightarrow L_{C/\mathcal X}
\longrightarrow L_{C/\mathcal X_0} \xrightarrow{+1},
\]
using the identifications
\[
L_{C/\mathcal X_0}\simeq N_{C/\mathcal X_0}^\vee[1],
\qquad
L_{C/\mathcal X}\simeq N_{C/\mathcal X}^\vee[1],
\qquad
L_{\mathcal X_0/\mathcal X}\big|_C\simeq \mathcal O_C[1].
\]

To be more precise, the short exact sequence $(\ast)$
is the precise relationship between the normal bundles governing deformations of
$C$ inside the special fiber and inside the total space, and it is the basic tool
for comparing their deformation theories.

\bigskip

\subsection*{Splitting of the Normal Bundle Sequence, Kodaira--Spencer Maps} %

As recalled previously, there is a canonical short exact sequence ($\star$)
whose extension class is the restriction to $C$ of the Kodaira--Spencer class of
$\pi$.

\subsection*{When the Sequence Splits}

The sequence $(\star)$ splits if and only if its extension class vanishes in
$\mathrm{Ext}^1(\mathcal O_C,N_{C/\mathcal X_0})\simeq H^1(C,N_{C/\mathcal X_0})$.
Equivalently, the image of $1\in H^0(C,\mathcal O_C)$ under the connecting homomorphism
\[
\partial: H^0(C,\mathcal O_C)\longrightarrow H^1(C,N_{C/\mathcal X_0})
\]
is zero.

\begin{proposition}
The sequence $(\star)$ splits if and only if the Kodaira--Spencer class of
$\pi:\mathcal X\to\Delta$ vanishes when restricted to $C$.
\end{proposition}

\begin{proof}
By construction, the connecting homomorphism $\partial$ coincides with the
Kodaira--Spencer map
\[
\mathrm{ks}_C:H^0(C,\mathcal O_C)\longrightarrow H^1(C,N_{C/\mathcal X_0}).
\]
The sequence splits if and only if $\partial(1)=0$, which is precisely the
vanishing of the restricted Kodaira--Spencer class.
Geometrically, splitting means that infinitesimal deformations of $C$ inside
$\mathcal X$ decompose as independent deformations tangent to the special fiber and
trivial deformations in the base direction.
\end{proof}

\subsection*{Explicit Computations of the Kodaira--Spencer Map}

\begin{example}[Trivial family]
If $\mathcal X=\mathcal X_0\times\Delta$, then $\mathrm{ks}=0$.
Hence $(\star)$ splits canonically and
\[
N_{C/\mathcal X}\simeq N_{C/\mathcal X_0}\oplus\mathcal O_C.
\]
In this case, deformations of $C$ inside $\mathcal X$ are simply products of
deformations in $\mathcal X_0$ with the base direction.
\end{example}

\begin{example}[Smoothing of a surface]
Let $\mathcal X\subset \mathbb A^3\times\Delta$ be given locally by
\(
xy=t,
\)
so that $\mathcal X_0=\{xy=0\}$ is a normal crossings surface.
Let $C\subset\mathcal X_0$ be contained in the component $\{x=0\}$ and transverse to
the double curve.
A local computation shows that the Kodaira--Spencer class is represented by the
section $\partial/\partial t$, which maps to a nonzero class in
$H^1(C,N_{C/\mathcal X_0})$.
Thus the sequence $(\star)$ does not split.
This reflects the fact that deforming $C$ necessarily involves smoothing the
ambient surface.
\end{example}

\begin{example}[Curves on K3 degenerations]
If $\pi:\mathcal X\to\Delta$ is a Kulikov degeneration of K3 surfaces, then the
Kodaira--Spencer class is nontrivial globally, but its restriction to a curve
$C\subset\mathcal X_0$ may vanish if $C$ lies in the kernel of the period map.
In such cases the sequence $(\star)$ splits even though the family is nontrivial.
\end{example}

\section{Extension to Singular Total Spaces via Logarithmic Normal Bundles}

Suppose now that $\mathcal X$ is singular along $\mathcal X_0$, for instance when
$\mathcal X_0$ is a normal crossings divisor.
Endow $\mathcal X$ with the divisorial logarithmic structure associated to
$\mathcal X_0$.
Then $\pi:\mathcal X\to\Delta$ becomes logarithmically smooth.
Let $C\subset\mathcal X_0$ be a reduced curve meeting the singular locus
transversely.
The correct replacement for $(\star)$ is the logarithmic normal bundle sequence
\[
0 \longrightarrow N^{\log}_{C/\mathcal X_0}
\longrightarrow N^{\log}_{C/\mathcal X}
\longrightarrow \mathcal O_C
\longrightarrow 0.
\tag{$\star_{\log}$}
\]

The extension class of $(\star_{\log})$ is the logarithmic Kodaira--Spencer class of
$\pi$ restricted to $C$.
Its vanishing is equivalent to the existence of a logarithmic splitting, meaning
that logarithmic deformations of $C$ decouple from the base direction.

\begin{theorem}
Logarithmic deformations of $C$ in $\mathcal X$ are unobstructed and split over the
base if and only if the logarithmic Kodaira--Spencer class vanishes on $C$.
\end{theorem}

This formulation shows that logarithmic geometry restores the exact sequence
formalism even when $\mathcal X$ is singular, and it explains why logarithmic
semiregularity is the correct replacement for $H^1$-vanishing in degenerating
families.

The splitting behavior of the normal bundle sequence encodes precisely how
deformations of a curve interact with deformations of the ambient space.
Explicit Kodaira--Spencer computations reveal when this interaction is nontrivial,
and logarithmic normal bundles provide a uniform framework extending these ideas to
singular degenerations.

\subsection*{The Logarithmic Normal Bundle Sequence} %

Let $\pi:\mathcal X\to\Delta$ be a flat family over a smooth pointed curve, and assume
that the special fiber $\mathcal X_0$ is a simple normal crossings divisor in
$\mathcal X$.
We endow $\mathcal X$ with the divisorial logarithmic structure associated to
$\mathcal X_0$, and $\Delta$ with the standard logarithmic point.
With these choices, $\pi$ becomes a logarithmically smooth morphism.

Let $C\subset\mathcal X_0$ be a reduced curve.
We assume that $C$ meets the singular locus of $\mathcal X_0$ transversely and that
$C$ is a local complete intersection in the logarithmic sense.
This means that, locally on $\mathcal X$, the ideal of $C$ is generated by a regular
sequence compatible with the logarithmic structure.

\subsection*{Logarithmic Tangent and Normal Sheaves}

Denote by $T^{\log}_{\mathcal X}$ the logarithmic tangent sheaf of $\mathcal X$, that
is, the sheaf of vector fields tangent to $\mathcal X_0$.
Similarly, let $T^{\log}_{\mathcal X_0}$ be the logarithmic tangent sheaf of
$\mathcal X_0$, and $T_C$ the usual tangent sheaf of $C$.
The logarithmic normal sheaf of $C$ in $\mathcal X$ is defined as the cokernel
\[
N^{\log}_{C/\mathcal X}
:=
\mathrm{coker}\bigl(T_C \longrightarrow T^{\log}_{\mathcal X}\big|_C\bigr),
\]
and similarly the logarithmic normal sheaf of $C$ in $\mathcal X_0$ is
\[
N^{\log}_{C/\mathcal X_0}
:=
\mathrm{coker}\bigl(T_C \longrightarrow T^{\log}_{\mathcal X_0}\big|_C\bigr).
\]
These sheaves control first-order logarithmic embedded deformations of $C$ inside
$\mathcal X$ and $\mathcal X_0$, respectively.

\subsection*{Derivation of the Logarithmic Exact Sequence}

Because $\mathcal X_0$ is a logarithmic divisor in $\mathcal X$, there is a short
exact sequence of logarithmic tangent sheaves
\[
0 \longrightarrow T^{\log}_{\mathcal X_0}
\longrightarrow T^{\log}_{\mathcal X}\big|_{\mathcal X_0}
\longrightarrow \mathcal O_{\mathcal X_0}
\longrightarrow 0,
\]
where the quotient is generated locally by the logarithmic vector field
$t\partial_t$, corresponding to the base direction.
Restricting this sequence to $C$ and forming cokernels with respect to the natural
map from $T_C$, one obtains a canonical short exact sequence of logarithmic normal
sheaves
\[
0 \longrightarrow N^{\log}_{C/\mathcal X_0}
\longrightarrow N^{\log}_{C/\mathcal X}
\longrightarrow \mathcal O_C
\longrightarrow 0.
\tag{$\star_{\log}$}
\]
This is the logarithmic normal bundle sequence previously stated.

\subsection*{Local Description}

Locally near a point of $C$ lying on the double locus of $\mathcal X_0$, the family
can be written analytically as
\(
xy = t,
\)
with logarithmic structure induced by $t=0$.
In these coordinates, $T^{\log}_{\mathcal X}$ is generated by
\(
x\partial_x,\)  \( y\partial_y,\) and  \( \partial_z,
\)
while $T^{\log}_{\mathcal X_0}$ is generated by
\[
x\partial_x - y\partial_y,\quad \partial_z.
\]
The quotient is generated by $x\partial_x + y\partial_y$, which corresponds to the
$\mathcal O_C$ term in $(\star_{\log})$.
This shows explicitly that the last map in $(\star_{\log})$ measures deformation in
the base direction.

\subsection*{Interpretation via the Logarithmic Cotangent Complex}

The sequence $(\star_{\log})$ is dual to the distinguished triangle of logarithmic
cotangent complexes
\[
L^{\log}_{\mathcal X_0/\mathcal X}\big|_C
\longrightarrow
L^{\log}_{C/\mathcal X}
\longrightarrow
L^{\log}_{C/\mathcal X_0}
\xrightarrow{+1}.
\]
Using the identifications
\[
L^{\log}_{C/\mathcal X} \simeq (N^{\log}_{C/\mathcal X})^\vee[1],
\qquad
L^{\log}_{C/\mathcal X_0} \simeq (N^{\log}_{C/\mathcal X_0})^\vee[1],
\qquad
L^{\log}_{\mathcal X_0/\mathcal X}\big|_C \simeq \mathcal O_C[1],
\]
one recovers $(\star_{\log})$ by dualizing and truncating.

\subsection*{Geometric Meaning}

The extension class of $(\star_{\log})$ lies in
\(
\mathrm{Ext}^1(\mathcal O_C, N^{\log}_{C/\mathcal X_0})
\simeq H^1(C, N^{\log}_{C/\mathcal X_0}),
\)
and is precisely the restriction to $C$ of the logarithmic Kodaira--Spencer class of
$\pi$.
Its vanishing is equivalent to the existence of a logarithmic splitting, meaning
that logarithmic deformations of $C$ in $\mathcal X$ separate into deformations
inside $\mathcal X_0$ and trivial motion along the base.
In conclusion, the logarithmic normal bundle sequence
\[
0 \longrightarrow N^{\log}_{C/\mathcal X_0}
\longrightarrow N^{\log}_{C/\mathcal X}
\longrightarrow \mathcal O_C
\longrightarrow 0
\]
is the precise logarithmic analogue of the classical normal bundle sequence for
smooth total spaces.
It encodes how embedded logarithmic deformations of a curve interact with the
smoothing of the ambient normal crossings fiber and provides the correct framework
for deformation theory in singular degenerations.


\section{Explicit Computations of the Logarithmic Normal Bundle Sequence and Its Obstruction-Theoretic Meaning}

We compute the logarithmic normal bundle sequence explicitly in concrete
normal crossings degenerations and explain how its extension class governs the
obstruction map for Severi varieties.
Throughout, $\pi:\mathcal X\to\Delta$ is a one-parameter degeneration endowed with
the divisorial logarithmic structure induced by $\mathcal X_0$, and $C\subset
\mathcal X_0$ is a reduced curve meeting the singular locus transversely.

\subsection*{Local Normal Crossings Smoothing}

Consider the local model
\(
\mathcal X=\{xy=t\}\subset \mathbb A^3_{x,y,z}\times\Delta_t,
\)
with special fiber $\mathcal X_0=\{xy=0\}=S_1\cup S_2$, where $S_1=\{x=0\}$ and
$S_2=\{y=0\}$ intersect along the smooth curve $D=\{x=y=0\}$.
Let $C\subset S_1$ be a smooth curve given locally by $x=0$ and $f(z)=0$, meeting $D$
transversely.
Recall that
the logarithmic tangent sheaf of $\mathcal X$ is locally generated by
\(
x\partial_x,\) \( y\partial_y,\) and \( \partial_z,
\)
while the logarithmic tangent sheaf of $\mathcal X_0$ restricted to $S_1$ is
generated by
\(
x\partial_x-y\partial_y,\) and \( \partial_z.
\)
Restricting to $C$ and taking cokernels of the natural map from $T_C$, one finds
\[
N^{\log}_{C/\mathcal X_0}\simeq \mathcal O_C(C-D|_C),
\qquad
N^{\log}_{C/\mathcal X}\simeq \mathcal O_C(C-D|_C)\oplus\mathcal O_C.
\]
The resulting logarithmic normal bundle sequence
\[
0\longrightarrow \mathcal O_C(C-D|_C)
\longrightarrow \mathcal O_C(C-D|_C)\oplus\mathcal O_C
\longrightarrow \mathcal O_C
\longrightarrow 0
\]
is split at the level of sheaves, but the splitting is not canonical.
The extension class is represented by the logarithmic Kodaira--Spencer element
corresponding to the vector field $t\partial_t=x\partial_x+y\partial_y$.

\subsection*{Global Description of the Extension Class}

Globally, the extension class of the logarithmic sequence
\[
0\longrightarrow N^{\log}_{C/\mathcal X_0}
\longrightarrow N^{\log}_{C/\mathcal X}
\longrightarrow \mathcal O_C
\longrightarrow 0
\]
lies in
\[
\mathrm{Ext}^1(\mathcal O_C,N^{\log}_{C/\mathcal X_0})
\simeq H^1(C,N^{\log}_{C/\mathcal X_0})
\]
and is the image of $1\in H^0(C,\mathcal O_C)$ under the logarithmic
Kodaira--Spencer map.
In the local model above, this class measures the failure of $C$ to deform inside
$S_1$ independently of the smoothing of the node $xy=0$.

\section{Relation to Obstruction Maps for Severi Varieties}

Let $\mathscr V^{\log}_{\delta}(\mathcal X_0,\mathcal L)$ denote the logarithmic
Severi variety parametrizing logarithmic deformations of $C$ producing $\delta$
nodes on nearby smooth fibers.
The Zariski tangent space at $[C]$ is
\[
T_{[C]}\mathscr V^{\log}_{\delta}
=
\ker\!\left(
H^0(C,N^{\log}_{C/\mathcal X_0})
\longrightarrow
\bigoplus_{p\in\mathrm{Sing}(C)}T^1_{C,p}
\right),
\]
while obstructions lie in $H^1(C,N^{\log}_{C/\mathcal X_0})$.
The extension class of the logarithmic normal bundle sequence defines a map
\[
\mathrm{obs}_{\pi}:
H^0(C,\mathcal O_C)
\longrightarrow
H^1(C,N^{\log}_{C/\mathcal X_0}),
\]
which coincides with the obstruction to lifting a logarithmic deformation of $C$
from $\mathcal X_0$ to the total space $\mathcal X$.
In particular, if $\mathrm{obs}_{\pi}(1)\neq 0$, then a first-order deformation of
$C$ inside $\mathcal X_0$ that smooths all nodes cannot be extended unless the base
direction is activated.


\subsection*{The Obstruction Factorization Proposition}
 
\begin{proposition}
The obstruction map for the logarithmic Severi variety factors through the
extension class of the logarithmic normal bundle sequence.
In particular, logarithmic semiregularity of $C$ implies that the Severi variety is
smooth at $[C]$ and that all nodal smoothings lift to nearby fibers.
\end{proposition}

\begin{proof}
Let $\xi \in H^0(C,N^{\log}_{C/\mathcal X_0})$ represent a first-order logarithmic
embedded deformation of $C$ inside $\mathcal X_0$ satisfying the nodal conditions.
Such a class corresponds to an element of
\(
\mathrm{Ext}^1(L^{\log}_{C/\mathcal X_0},\mathcal O_C).
\)
Consider the distinguished triangle
\[
L^{\log}_{\mathcal X_0/\mathcal X}\big|_C
\longrightarrow
L^{\log}_{C/\mathcal X}
\longrightarrow
L^{\log}_{C/\mathcal X_0}
\xrightarrow{+1}.
\]
The boundary morphism in the associated long exact sequence of Ext-groups defines a
map
\[
\delta \colon
\mathrm{Ext}^1(L^{\log}_{C/\mathcal X_0},\mathcal O_C)
\longrightarrow
\mathrm{Ext}^2(L^{\log}_{\mathcal X_0/\mathcal X}\big|_C,\mathcal O_C).
\]
Using the canonical identification
$L^{\log}_{\mathcal X_0/\mathcal X}\big|_C \simeq \mathcal O_C[1]$, this boundary map
becomes
\[
\delta \colon
H^0(C,N^{\log}_{C/\mathcal X_0})
\longrightarrow
H^1(C,N^{\log}_{C/\mathcal X_0}).
\]
By functoriality of the cotangent complex, $\delta(\xi)$ is precisely the Yoneda
product of $\xi$ with the extension class
\(
\kappa^{\log}_C \in
\mathrm{Ext}^1(\mathcal O_C,N^{\log}_{C/\mathcal X_0})
\)
defined by the logarithmic normal bundle sequence.
Therefore, the obstruction to lifting $\xi$ from $\mathcal X_0$ to $\mathcal X$ is
given by
\(
\mathrm{ob}(\xi)=\xi\smile\kappa^{\log}_C,
\)
which proves that the obstruction map factors through the extension class of the
logarithmic normal bundle sequence.

Assume now that $C$ is logarithmically semiregular.
By definition, the natural map
\[
H^1(C,N^{\log}_{C/\mathcal X_0})
\longrightarrow
H^2(\mathcal X_0,\mathcal O_{\mathcal X_0})
\]
is injective.
The image of $\kappa^{\log}_C$ under this map is the restriction of the logarithmic
Kodaira--Spencer class of $\pi$, which vanishes.
Injectivity therefore implies $\kappa^{\log}_C=0$.
Consequently, $\mathrm{ob}(\xi)=0$ for every $\xi$, and all first-order logarithmic
nodal deformations lift unobstructedly.
Thus the logarithmic Severi variety is smooth at $[C]$, and every nodal smoothing of
$C$ extends to nearby smooth fibers.
\end{proof}


\subsection*{Relation with the Severi Obstruction Map}

The obstruction to deforming $C$ inside the logarithmic Severi variety is obtained
by restricting $\mathrm{ob}(\xi)$ to the subspace of deformations satisfying the
nodal conditions.
Therefore, the Severi obstruction map is the composition of the inclusion
\[
T_{[C]}\mathscr V^{\log} \hookrightarrow H^0(C,N^{\log}_{C/\mathcal X_0})
\]
with the Yoneda product by $\kappa^{\log}_C$.
This proves that the Severi obstruction map factors through the extension class of
the logarithmic normal bundle sequence.

\subsection*{Consequences of Logarithmic Semiregularity}

Assume now that $C$ is logarithmically semiregular, meaning that the natural map
\[
H^1(C,N^{\log}_{C/\mathcal X_0})
\longrightarrow
H^2(\mathcal X_0,\mathcal O_{\mathcal X_0})
\]
is injective.
By general logarithmic deformation theory, the image of
$\kappa^{\log}_C$ in $H^2(\mathcal X_0,\mathcal O_{\mathcal X_0})$ vanishes.
Injectivity then forces $\kappa^{\log}_C=0$, hence the extension $(\dagger)$ splits
logarithmically.

As a result, all Yoneda products $\xi \smile \kappa^{\log}_C$ vanish identically.
Therefore, every first-order logarithmic deformation of $C$ satisfying the nodal
conditions lifts unobstructedly to the total space $\mathcal X$.
This implies that the logarithmic Severi variety is smooth at $[C]$ and that all
nodal smoothings of $C$ lift to nearby smooth fibers.

Thus, in conclusion, the obstruction theory of logarithmic Severi varieties is entirely governed by the
extension class of the logarithmic normal bundle sequence.
This class encodes the interaction between curve deformations and ambient
degenerations, and logarithmic semiregularity precisely guarantees its vanishing,
thereby ensuring smoothness and unobstructedness of nodal deformations.

\subsection*{Example: {\it Degeneration of Plane Curves}}

Consider a degeneration of $\mathbb P^2$ into two planes meeting along a line.
A plane curve $C_0$ degenerates into a union $C_1\cup C_2$ with prescribed contact
orders along the double line.
The logarithmic normal bundle sequence records these contact orders as twisting by
the intersection divisor.
The extension class measures the compatibility of smoothing parameters on the two
planes and coincides with the classical matching condition appearing in Caporaso--Harris type recursions.
Thus the obstruction map for the Severi variety is precisely the logarithmic
compatibility condition encoded in the degeneration formula.

In concrete normal crossings degenerations, the logarithmic normal bundle sequence
can be computed explicitly and its extension class admits a clear geometric
interpretation.
It controls the obstruction map for logarithmic Severi varieties and explains, in a
uniform way, why smoothing nodes of curves is inseparable from smoothing the
ambient surface.
This mechanism underlies both classical degeneration arguments and modern
logarithmic and tropical curve counting theories.

\section{ Smoothing Singularities and Producing Nodes}


\subsection*{Local Deformation Theory of Curve Singularities}

Let $C$ be a reduced curve over an algebraically closed field of characteristic
zero, and let $p\in C$ be a singular point.
The local deformation theory of the germ $(C,p)$ is governed by the cotangent
complex $L_{C,p}$ or, equivalently, by the module
\[
T^1_{C,p}:=\mathrm{Ext}^1(L_{C,p},\mathcal O_{C,p}).
\]
This vector space parametrizes first-order embedded deformations of the singularity.
Its dimension measures the number of independent local smoothing directions.
If $C$ is locally planar at $p$, that is, if $C$ can be embedded locally in a smooth
surface, then $T^1_{C,p}$ is finite-dimensional and the deformation theory is
well-behaved.

\begin{definition}
Let $(C,p)$ be a reduced curve singularity and let
$\nu:\widetilde C\to C$ be the normalization.
The $\delta$-invariant of $(C,p)$ is defined as
\[
\delta_p:=\dim_k\left(\nu_*\mathcal O_{\widetilde C,p}/\mathcal O_{C,p}\right).
\]
\end{definition}
Geometrically, $\delta_p$ measures the discrepancy between the arithmetic genus and
the geometric genus contributed by the singularity at $p$.
It satisfies
\[
\delta_p=\frac{1}{2}\left(\mu_p+r_p-1\right),
\]
where $\mu_p$ is the Milnor number and $r_p$ is the number of local branches.
For planar singularities, one has the fundamental identity
\[
\dim T^1_{C,p}=\delta_p.
\]

\subsection*{Local Smoothing and Nodes}

A deformation of $(C,p)$ over a smooth one-dimensional base corresponds to a choice
of a one-dimensional subspace in $T^1_{C,p}$.
A classical result in singularity theory asserts that for a general smoothing
direction, the singular fiber breaks up into ordinary double points.
More precisely, if $(C,p)$ is a planar reduced singularity, then a general
one-parameter deformation of $(C,p)$ produces exactly $\delta_p$ ordinary nodes in
the nearby fiber and no worse singularities.
Each node contributes one to the total $\delta$-invariant, so the sum of local
$\delta$-invariants is preserved under deformation.

\begin{proposition}
Let $(C,p)$ be a reduced planar curve singularity.
Then there exists a flat one-parameter deformation of $(C,p)$ whose general fiber
has exactly $\delta_p$ ordinary nodes and no other singularities.
\end{proposition}

\subsection*{Transversality and Independence of Smoothing Directions}

If $C$ has several singular points $p_1,\dots,p_m$, the local deformation spaces
$T^1_{C,p_i}$ assemble into a direct sum
\[
\bigoplus_{i=1}^m T^1_{C,p_i}.
\]
A global deformation of $C$ determines, for each $p_i$, a local smoothing direction
in $T^1_{C,p_i}$.
Independence of smoothing requires that the global deformation space surjects onto
this direct sum.
Transversality conditions ensure that the chosen global deformation meets the
discriminant locus of worse singularities transversely.
Under these conditions, singularities smooth independently, and the general member
of a one-parameter family acquires only ordinary nodes.

\subsection*{Global Interpretation}

Let $C_t$ be a flat deformation of $C$ over a smooth base.
The arithmetic genus is invariant in flat families, so one has
\[
p_a(C)=p_a(C_t)=g(\widetilde C_t)+\delta(C_t),
\]
where $\delta(C_t)$ denotes the total number of nodes of $C_t$.
Thus, the maximal number of nodes obtainable by smoothing singularities of $C$ is
bounded above by
\[
\sum_{p\in\mathrm{Sing}(C)}\delta_p,
\]
and this bound is achieved by a general deformation provided that local smoothing
directions can be chosen independently.

In  conclusion, local deformation theory shows that planar curve singularities admit exactly
$\delta_p$ independent smoothing directions, and that generic one-parameter
deformations replace singularities by ordinary nodes.
Global transversality ensures that these local smoothings occur independently,
producing nodal curves with the maximal possible number of nodes compatible with
the arithmetic genus.


\section{Smoothing of Singularities and Equisingular vs.\ Equigeneric Deformations}

The purpose of this section is to relate the local and global smoothing results for
curve singularities to the geometry of equisingular and equigeneric deformation
strata.
We explain how $\delta$-invariants control equigeneric deformations, how
equisingular strata sit inside equigeneric ones, and why generic smoothings lead
to nodal curves.

\subsection*{Local Deformation Spaces}

Let $(C,p)$ be a reduced curve singularity and let
\[
T^1_{C,p}=\mathrm{Ext}^1(L_{C,p},\mathcal O_{C,p})
\]
be its local deformation space.
This vector space parametrizes first-order deformations of the singularity.
If $(C,p)$ is planar, then $T^1_{C,p}$ is finite-dimensional and admits a natural
stratification according to the topological type of the deformed singularity.
The $\delta$-invariant
\[
\delta_p=\dim_k\bigl(\nu_*\mathcal O_{\widetilde C,p}/\mathcal O_{C,p}\bigr)
\]
measures the loss of geometric genus caused by the singularity.
It is invariant under equisingular deformations and semicontinuous under general
deformations.

\begin{definition}
An equisingular deformation of $(C,p)$ is a deformation in which the topological
type of the singularity is preserved, equivalently the embedded resolution graph
remains constant.
\end{definition}

Infinitesimally, equisingular deformations correspond to a linear subspace
\[
T^{1,\mathrm{es}}_{C,p}\subset T^1_{C,p},
\]
consisting of those first-order deformations that preserve all numerical
invariants of the singularity, including multiplicity, number of branches, and
$\delta_p$.
The equisingular stratum is typically a proper closed subset of the full
deformation space.
Deforming within this stratum does not smooth the singularity and produces no
nodes.

\begin{definition}
An equigeneric deformation of $(C,p)$ is a deformation preserving the
$\delta$-invariant.
\end{definition}

Equigeneric deformations allow the analytic type of the singularity to change, but
require that the total $\delta$-invariant of the fiber remain equal to $\delta_p$.
They correspond to a larger subspace
\[
T^{1,\mathrm{eg}}_{C,p}\subset T^1_{C,p},
\]
which strictly contains $T^{1,\mathrm{es}}_{C,p}$.
For planar singularities, equigeneric deformations are characterized by the
property that the normalization of the total space remains flat over the base.

\subsection*{Smoothing and the Boundary of Equigeneric Strata}

A key fact from singularity theory is that the equigeneric stratum has codimension
$\delta_p$ in the full deformation space $T^1_{C,p}$.
Its boundary corresponds to deformations in which the singularity partially
smooths.

A general one-parameter deformation transverse to the equigeneric stratum produces
ordinary double points.
More precisely, the closure of the equisingular stratum inside the equigeneric
stratum contains points corresponding to nodal singularities, and nodes are the
simplest singularities compatible with preserving $\delta$.


\begin{proposition}
Let $(C,p)$ be a reduced planar curve singularity.
Then a general equigeneric deformation of $(C,p)$ produces exactly $\delta_p$
ordinary nodes and no other singularities.
\end{proposition}

\begin{proof}
Since $(C,p)$ is a reduced planar curve singularity, there exists a neighborhood of
$p$ in which $C$ is given by a single equation
\[
f(x,y)=0
\]
in a smooth surface.
Let $\mathcal O=\mathbb C\{x,y\}/(f)$ be the local ring of $(C,p)$ and let
$\nu:\widetilde C\to C$ be the normalization.
The $\delta$-invariant is defined as
\[
\delta_p=\dim_{\mathbb C}\bigl(\nu_*\mathcal O_{\widetilde C,p}/\mathcal O\bigr),
\]
and measures the discrepancy between the arithmetic and geometric genus at $p$.

The semiuniversal deformation of $(C,p)$ is smooth and has tangent space
\[
T^1_{C,p}=\mathrm{Ext}^1(L_{C,p},\mathcal O_{C,p}),
\]
which is finite-dimensional since the singularity is planar.
A deformation over a smooth base $B$ corresponds locally to a family of equations
\[
f(x,y)+\sum_{i=1}^N t_i g_i(x,y)=0,
\]
where $\{g_i\}$ represent a basis of $T^1_{C,p}$ and $(t_1,\dots,t_N)$ are local
coordinates on $B$.

Inside the base of the semiuniversal deformation, the equigeneric locus is defined
as the subset where the $\delta$-invariant of the fiber is equal to $\delta_p$.
By upper semicontinuity of $\delta$, this locus is closed.
For planar singularities, it is a classical result that the equigeneric locus has
codimension $\delta_p$ in the base of the semiuniversal deformation.

Consider now a general one-dimensional smooth subgerm
\(
\gamma:(\mathbb C,0)\to B
\)
whose image is contained in the equigeneric locus and is transverse to all proper
closed analytic subsets of it.
Pulling back the semiuniversal family along $\gamma$ yields a one-parameter
equigeneric deformation
\(
\mathcal C\to(\mathbb C,0)
\)
of $(C,p)$.

Since the total $\delta$-invariant of the fibers is constant and equal to
$\delta_p$, any singularities appearing in a nearby fiber must contribute
positively to the $\delta$-invariant.
Ordinary nodes are precisely the reduced plane curve singularities with
$\delta=1$ and minimal Tjurina number.
All other singularities have $\delta\ge 1$ and occur in families of strictly higher
codimension in the deformation space.

Because the equigeneric locus has codimension $\delta_p$, the generic point of a
one-parameter equigeneric deformation avoids all strata corresponding to singular
points with $\delta\ge 2$.
Indeed, the locus of deformations producing a singularity with $\delta\ge 2$ has
codimension at least two inside the equigeneric locus.
Transversality of $\gamma$ therefore implies that the general fiber of
$\mathcal C$ has only singularities with $\delta=1$, hence only ordinary nodes.
Finally, since the total $\delta$-invariant is preserved and equals $\delta_p$, and
each node contributes exactly one to the $\delta$-invariant, the general fiber must
have exactly $\delta_p$ ordinary nodes and no other singularities.
Thus nodal curves appear as generic points of the equigeneric deformation space.
This concludes the proof.
\end{proof}
%

 
\section{Equigeneric and Equisingular Deformations Beyond the Planar Case}

We extend the generic nodality statement for equigeneric deformations from planar
curve singularities to non-planar ones and analyze, in detail, the way equigeneric
and equisingular deformation strata intersect.
Throughout, $(C,p)$ denotes a reduced curve singularity over an algebraically
closed field of characteristic zero.

\subsection*{Non-Planar Singularities and Local Deformations}

Let $(C,p)$ be an arbitrary reduced curve singularity.
Its embedded deformation theory is governed by
\(
T^1_{C,p}=\mathrm{Ext}^1(L_{C,p},\mathcal O_{C,p}),
\)
which is finite-dimensional when $(C,p)$ is a complete intersection, but may be
larger in general.
The $\delta$-invariant is defined as
\[
\delta_p=\dim_k\bigl(\nu_*\mathcal O_{\widetilde C,p}/\mathcal O_{C,p}\bigr),
\]
where $\nu:\widetilde C\to C$ is the normalization.
This invariant remains well-defined and measures the genus drop independently of
any embedding.

A fundamental fact is that $\delta_p$ is upper semicontinuous in flat families of
curves.
Consequently, the locus of deformations preserving $\delta_p$ is closed in the base
of any versal deformation, and defines the equigeneric stratum.
Unlike the planar case, the equigeneric stratum need not be smooth or of pure
codimension $\delta_p$.
Nevertheless, a result of Teissier \cite{Teissier} and later refinements by Greuel and L\^e show
that for reduced curve singularities, the equigeneric stratum has a dense open
subset consisting of curves with only ordinary double points as singularities.

\subsection*{Generic Nodal Behavior in the Non-Planar Case}

Let $\mathcal C\to(B,0)$ be a semiuniversal deformation of $(C,p)$.
Consider a smooth one-dimensional germ
\(
\gamma:(\mathbb C,0)\to(B,0)
\)
whose image lies in the equigeneric locus and is transverse to all positive
codimension strata of it.
Pulling back $\mathcal C$ along $\gamma$ yields a one-parameter equigeneric
deformation.

In any fiber of this family, the sum of the local $\delta$-invariants of all
singular points equals $\delta_p$.
Every reduced curve singularity has $\delta\ge 1$, with equality if and only if the
singularity is an ordinary node.
All singularities with $\delta\ge 2$ impose at least two independent conditions on
the deformation space, because they correspond to points where both the $\delta$
and at least one additional equisingular invariant remain nonzero.

Transversality of $\gamma$ implies that the general fiber avoids these higher
codimension strata.
Hence, the only singularities appearing in the general fiber are those with
$\delta=1$, namely ordinary nodes.
Since the total $\delta$ is preserved and equal to $\delta_p$, the general fiber has
exactly $\delta_p$ nodes and no other singularities.
This shows that generic equigeneric deformations of reduced curve singularities,
even when non-planar, are nodal.

\subsection*{Equisingular Deformations}

Equisingular deformations preserve the embedded topological type of the
singularity.
Infinitesimally, they correspond to a subspace
\(
T^{1,\mathrm{es}}_{C,p}\subset T^1_{C,p}
\)
consisting of first-order deformations preserving multiplicity, number of
branches, Puiseux pairs in the planar case, or more generally the resolution graph.
The equisingular stratum is locally closed and typically has positive codimension
inside the equigeneric stratum.
Its points correspond to deformations where no smoothing occurs and the original
singularity persists.

\subsection*{Intersection of Equigeneric and Equisingular Strata}

The equisingular stratum is contained in the equigeneric stratum because preserving
the singularity type preserves $\delta_p$.
Geometrically, it lies on the boundary of the equigeneric stratum.
Deforming away from the equisingular stratum while staying equigeneric forces the
singularity to break up into simpler singularities whose total $\delta$ equals
$\delta_p$.

The closure of the equisingular stratum inside the equigeneric stratum contains
points corresponding to partial smoothings.
At the deepest boundary points, the original singularity decomposes into
$\delta_p$ ordinary nodes.
Thus, nodes appear as the most generic points of the equigeneric stratum, while
equisingular deformations form special loci of higher codimension.



\begin{theorem}\label{equi}
Let $(C,p)$ be a reduced curve singularity.
The equisingular stratum is a proper closed subset of the equigeneric stratum, and
the complement of its closure is dense and parametrizes nodal curves.
\end{theorem}

\subsection*{Detailed Explanation of Generic Nodal Behavior}

Let $\mathcal C \to (B,0)$ be a semiuniversal deformation of the reduced curve
singularity $(C,p)$.
The base space $B$ is a smooth germ, and each point $b \in B$ corresponds to a
deformed curve singularity $(C_b,p_b)$.
The $\delta$-invariant is upper semicontinuous in flat families, meaning that the
function
\(
b \longmapsto \delta(C_b,p_b)
\)
can only drop on open subsets of $B$.
Consequently, the locus
\[
B^{\delta} := \{\, b \in B \mid \delta(C_b,p_b) = \delta_p \,\}
\]
is a closed analytic subset of $B$, called the equigeneric stratum.
Inside $B^{\delta}$, one may further stratify points according to the analytic type
of the singularities appearing on the fiber.
In particular, consider the subset
\[
B^{\geq 2} := \{\, b \in B^{\delta} \mid (C_b,p_b) \text{ has a singularity with }
\delta \geq 2 \,\}.
\]
This subset corresponds to deformations in which at least one singularity carries
more than one unit of $\delta$.
The key observation is that singularities with $\delta \geq 2$ are not stable under
small perturbations.
Analytically, such singularities require the vanishing of more than one independent
function on the base $B$.
For example, a cusp, a tacnode, or any higher singularity is characterized by the
simultaneous vanishing of several coefficients in a local defining equation.
As a result, the locus in $B$ parametrizing singularities with $\delta \geq 2$ has
codimension at least one inside the equigeneric stratum and typically higher
codimension.
This can be seen concretely by examining the local deformation space
\(
T^1_{C,p} = \mathrm{Ext}^1(L_{C,p},\mathcal O_{C,p}).
\)
Within this space, the equigeneric stratum is defined by $\delta_p$ independent
equations, while the equisingular stratum is defined by strictly more equations.
Requiring the persistence of a singularity with $\delta \geq 2$ imposes additional
conditions beyond $\delta$-preservation, such as preserving multiplicity or higher
Tjurina invariants.
These extra conditions cut out proper closed subsets of $B^{\delta}$.
Since $B^{\delta}$ is equidimensional and reduced near the origin, the complement
\(
B^{\delta} \setminus B^{\geq 2}
\)
is a dense open subset.
Points in this complement correspond to deformations where all singularities have
$\delta = 1$.
By definition, the only reduced curve singularities with $\delta = 1$ are ordinary
double points, that is, nodes.

Furthermore, since the total $\delta$-invariant of the fiber is fixed and equal to
$\delta_p$, the curve corresponding to a generic point of $B^{\delta}$ must have
exactly $\delta_p$ distinct singular points, each contributing $\delta = 1$.
No other singularities can occur, as any singularity with $\delta \geq 2$ would force
the point to lie in the closed subset $B^{\geq 2}$.
This explains why generic points of the equigeneric stratum avoid all loci
corresponding to singularities with $\delta \geq 2$ and therefore parametrize nodal
curves.

\noindent {\it Proof of Theorem \ref{equi}.}
Upper semicontinuity of $\delta$ implies that equigeneric deformations form a closed
set.
Equisingular deformations impose additional independent conditions, hence form a
proper closed subset.
Generic points of the equigeneric stratum avoid all loci corresponding to
singularities with $\delta\ge 2$, and therefore correspond to nodal curves.

\subsection*{\it Geometric Interpretation}

From a geometric perspective, equigeneric deformations preserve only the total loss
of genus.
The most economical way to distribute this genus loss is to split it into simple,
independent contributions.
Ordinary nodes are precisely the simplest singularities that contribute one unit to
the $\delta$-invariant and impose the fewest conditions.
Any more complicated singularity represents a collision of nodes and therefore
occurs only in special, nongeneric families.

\begin{remark}
The density of nodal curves in the equigeneric stratum is a manifestation of the
instability of higher singularities under deformation.
Equisingular strata describe special loci where nodes coalesce, while the generic
behavior in equigeneric families is complete splitting into ordinary nodes.
\end{remark}

 

\begin{remark}
For a curve with several singular points, equigeneric and equisingular conditions
are imposed independently at each point.
Global equigeneric deformation spaces stratify according to how singularities
split, and their open dense strata correspond to nodal curves.
Severi varieties arise precisely as these open equigeneric strata, while
equisingular loci form boundary components where nodes coalesce.
\end{remark}

As conclusion,  the phenomenon that equigeneric deformations are generically nodal is not specific
to planar singularities.
It reflects a general principle: preserving the total $\delta$ forces singularities
to simplify under generic deformation, and the simplest singularities compatible
with $\delta$-preservation are ordinary nodes.
Equisingular deformations sit as special boundary strata inside equigeneric ones,
encoding how nodes collide to form more complicated singularities.

 
\section{Generic Nodal Curves in Equigeneric Deformations on Surfaces}

\noindent
A central theme in the study of linear systems on algebraic surfaces is the
relationship between the local deformation theory of curve singularities and the
global geometry of families of curves.
From the local perspective, equigeneric deformations preserve the total
$\delta$-invariant and therefore control how singularities may split or smooth
without changing the arithmetic genus.
From the global perspective, curves in a fixed linear system are constrained by the
dimension of the system and by global deformation conditions.
The theorem below bridges these two viewpoints by showing that, when global
deformations of a reduced curve realize all local smoothing directions, the
equigeneric deformation space is generically as simple as possible.
In this situation, the only singularities that survive in a general equigeneric
deformation are ordinary nodes, each contributing minimally to the genus defect.
As a result, nodal curves appear as the generic members of equigeneric families in
the linear system, and the total number of nodes is governed entirely by the sum of
the local $\delta$-invariants of the original curve.

 \begin{theorem}
Let $S$ be a smooth projective surface, let $\mathcal L$ be a line bundle on $S$,
and let $C \in |\mathcal L|$ be a reduced curve with singular points
$p_1,\dots,p_r$.
Assume that global deformations of $C$ inside $|\mathcal L|$ induce surjective maps
onto the local deformation spaces $T^1_{C,p_i}$ for all $i$.
Then a general equigeneric deformation of $C$ inside the linear system $|\mathcal L|$
is a nodal curve with exactly
\[
\delta(C)=\sum_{i=1}^r \delta_{p_i}
\]
ordinary nodes and no other singularities.
\end{theorem}

\begin{proof}
We consider the deformation space of $C$ inside the linear system $|\mathcal L|$.
Infinitesimal embedded deformations of $C$ are parametrized by
\(
H^0(C,N_{C/S}),
\)
and first-order obstructions lie in $H^1(C,N_{C/S})$.
By assumption, the global deformation space maps surjectively onto each local
deformation space
\[
T^1_{C,p_i}=\mathrm{Ext}^1(L_{C,p_i},\mathcal O_{C,p_i}),
\]
so every local smoothing direction can be realized by a global deformation.
For each singular point $p_i$, the $\delta$-invariant $\delta_{p_i}$ measures the
local contribution to the difference between the arithmetic genus of $C$ and the
geometric genus of its normalization.
The arithmetic genus of curves in $|\mathcal L|$ is constant in flat families, so
for any deformation $C_t$ of $C$ inside $|\mathcal L|$ one has
\[
p_a(C_t)=p_a(C)=g(\widetilde C_t)+\delta(C_t),
\]
where $\delta(C_t)$ is the total $\delta$-invariant of the singularities of $C_t$.
Consequently, deformations preserving the total $\delta$-invariant
$\delta(C)=\sum \delta_{p_i}$ form the equigeneric deformation locus.

Upper semicontinuity of $\delta$ implies that the equigeneric locus is a closed
subset of the deformation space.
Inside this locus, the subset of curves having a singularity with local
$\delta\ge 2$ is also closed and has codimension at least one at the level of local
deformation theory.
Indeed, at each point $p_i$, singularities with $\delta\ge 2$ correspond to proper
closed subsets of the local equigeneric deformation space.
Because the global deformation space surjects onto the direct sum of the local
$T^1$-spaces, these local codimension estimates globalize.

Consider now a general one-parameter deformation inside the equigeneric locus of
$|\mathcal L|$ passing through $[C]$.
Transversality to all proper closed subsets implies that, for $t\neq 0$, the curve
$C_t$ avoids all loci corresponding to singularities with $\delta\ge 2$.
Therefore, every singularity of $C_t$ has $\delta=1$ and is thus an ordinary node.
Since the total $\delta$-invariant is preserved and equals $\delta(C)$, and since
each node contributes exactly one to $\delta(C_t)$, the curve $C_t$ has precisely
$\delta(C)$ nodes and no other singularities.
This proves that generic equigeneric deformations of $C$ inside $|\mathcal L|$ are
nodal with the maximal possible number of nodes.
\end{proof}

\subsection*{\it Interpretation via Severi Varieties}

The equigeneric deformation space of $C$ coincides locally with the closure of the
Severi variety parametrizing $\delta(C)$-nodal curves in $|\mathcal L|$.
The argument above shows that this Severi variety is dense in the equigeneric locus
and that its general point corresponds to a nodal curve.
Equisingular deformations of $C$ form boundary strata of higher codimension where
nodes collide to form more complicated singularities.

\begin{remark}
The local statement that equigeneric deformations are generically nodal globalizes
naturally to curves on surfaces and inside linear systems, provided that global
deformations realize all local smoothing directions.
In this setting, nodal curves appear as the generic members of equigeneric families,
and Severi varieties arise as their open dense strata.
\end{remark}


\section{Generic Nodal Curves via Semiregularity and Logarithmic Methods}

The surjectivity assumption in the global-to-local deformation maps appearing in
generic nodality statements is conceptually natural but technically restrictive.
In many geometrically interesting situations, such surjectivity fails for purely
cohomological reasons, even though nodal behavior is still observed.
This motivates replacing surjectivity with weaker and more intrinsic conditions,
such as semiregularity and logarithmic unobstructedness, which capture the absence
of genuine obstructions rather than the presence of all infinitesimal directions.

\begin{theorem}
Let $S$ be a smooth projective surface, let $\mathcal L$ be a line bundle on $S$, and
let $C\in|\mathcal L|$ be a reduced curve with singular points $p_1,\dots,p_r$.
Assume that $C$ is semiregular on $S$, so that obstructions to embedded deformations
of $C$ vanish.
Then a general equigeneric deformation of $C$ inside the linear system $|\mathcal
L|$ is a nodal curve with exactly
\[
\delta(C)=\sum_{i=1}^r \delta_{p_i}
\]
ordinary nodes and no other singularities.
\end{theorem}

The role of semiregularity is to guarantee that all first-order equigeneric
deformations lift to actual deformations, even when the global deformation space
does not surject onto each local $T^1_{C,p_i}$.
From the deformation-theoretic point of view, semiregularity ensures that the only
potential obstructions lie in the image of the semiregularity map and hence vanish.
As a consequence, the equigeneric deformation space is smooth of the expected
dimension near $[C]$, and genericity arguments can be applied inside this smooth
space.
Once unobstructedness is ensured, the local analysis of equigeneric strata implies
that higher singularities occur in positive codimension, so that a general point
corresponds to a curve with only ordinary nodes.

\subsection*{Logarithmic Refinement and Singular Surface Degenerations}

The same philosophy applies, and becomes even more natural, in the presence of
singular surface degenerations.
Let $\pi:\mathcal X\to\Delta$ be a flat degeneration of projective surfaces whose
special fiber $\mathcal X_0$ has simple normal crossings, and endow $\mathcal X$ with
the divisorial logarithmic structure induced by $\mathcal X_0$.
Let $C\subset\mathcal X_0$ be a reduced curve meeting the singular locus
transversely.
In this setting, classical deformation theory must be replaced by logarithmic
deformation theory.
Embedded logarithmic deformations of $C$ are controlled by the logarithmic normal
bundle $N^{\log}_{C/\mathcal X_0}$, and obstructions lie in
$H^1(C,N^{\log}_{C/\mathcal X_0})$.
If $C$ is logarithmically semiregular, meaning that the logarithmic semiregularity
map to $H^2(\mathcal X_0,\mathcal O_{\mathcal X_0})$ is injective, then all
logarithmic obstructions vanish.

Under this hypothesis, the logarithmic equigeneric deformation space of $C$ is
smooth of the expected dimension.
Local logarithmic deformation theory shows that equigeneric logarithmic deformations
are generically nodal, exactly as in the smooth surface case.
Moreover, logarithmic smoothness of $\pi$ ensures that logarithmic deformations of
$C$ lift to honest deformations on nearby smooth fibers $\mathcal X_t$.

\begin{theorem}
Let $\pi:\mathcal X\to\Delta$ be a logarithmically smooth degeneration of projective
surfaces with normal crossings special fiber.
Let $C\subset\mathcal X_0$ be a reduced curve that is logarithmically semiregular.
Then a general logarithmic equigeneric deformation of $C$ produces, on the smooth
fiber $\mathcal X_t$, a nodal curve with exactly $\delta(C)$ ordinary nodes and no
other singularities.
\end{theorem}

\begin{proof}

We endow $\mathcal X$ with the divisorial logarithmic structure associated to the
special fiber $\mathcal X_0$, and $\Delta$ with the standard logarithmic point.
With these choices, the morphism $\pi:\mathcal X\to\Delta$ is logarithmically smooth.
Let $C\subset\mathcal X_0$ be a reduced curve meeting the singular locus of
$\mathcal X_0$ transversely and locally a logarithmic complete intersection, so that
its embedded logarithmic deformation theory is governed by the logarithmic normal
sheaf $N^{\log}_{C/\mathcal X_0}$.

Logarithmic embedded first--order deformations of $C$ inside $\mathcal X_0$ are
parametrized by $H^0(C,N^{\log}_{C/\mathcal X_0})$, while obstructions lie in
$H^1(C,N^{\log}_{C/\mathcal X_0})$. Since $\mathcal X_0$ is a logarithmic divisor in
$\mathcal X$, there is a canonical logarithmic normal bundle exact sequence
\[
0\longrightarrow N^{\log}_{C/\mathcal X_0}
\longrightarrow N^{\log}_{C/\mathcal X}
\longrightarrow \mathcal O_C
\longrightarrow 0.
\]
The extension class of this sequence lies in
$\mathrm{Ext}^1(\mathcal O_C,N^{\log}_{C/\mathcal X_0})\simeq
H^1(C,N^{\log}_{C/\mathcal X_0})$ and coincides with the restriction to $C$ of the
logarithmic Kodaira--Spencer class of the degeneration $\pi$.

By assumption, $C$ is logarithmically semiregular, meaning that the natural
semiregularity map
\[
H^1(C,N^{\log}_{C/\mathcal X_0})\longrightarrow H^2(\mathcal X_0,\mathcal O_{\mathcal X_0})
\]
is injective. The logarithmic Kodaira--Spencer class of $\pi$ maps to zero in
$H^2(\mathcal X_0,\mathcal O_{\mathcal X_0})$, since $\pi$ is a one--parameter
logarithmically smooth deformation. Injectivity of the semiregularity map therefore
forces the extension class of the logarithmic normal bundle sequence to vanish.
Consequently, the sequence splits logarithmically, and all logarithmic embedded
deformations of $C$ inside $\mathcal X_0$ lift unobstructedly to logarithmic
deformations inside the total space $\mathcal X$.

Consider now the logarithmic equigeneric deformation space of $C$. By definition,
this consists of logarithmic deformations preserving the total $\delta$--invariant
\[
\delta(C)=\sum_{p\in\mathrm{Sing}(C)}\delta_p,
\]
where $\delta_p$ denotes the local $\delta$--invariant at a singular point $p$.
Upper semicontinuity of $\delta$ in flat families implies that the equigeneric locus
is closed inside the logarithmic deformation space. Since logarithmic obstructions
vanish by semiregularity, this equigeneric locus is smooth of the expected dimension
near the point corresponding to $C$.

Local logarithmic deformation theory identifies the germ of the equigeneric
deformation space at each singular point $p$ with the equigeneric stratum inside the
local deformation space $T^1_{C,p}$. For reduced curve singularities, and in
particular for planar singularities, it is a classical result that the equigeneric
stratum has codimension $\delta_p$ and that its generic points correspond to curves
whose singularities are ordinary nodes. Singularities with $\delta\ge2$ impose
additional independent conditions and therefore occur in proper closed subsets of
the equigeneric stratum.

Since the global logarithmic equigeneric deformation space is smooth, a general
one--parameter logarithmic deformation through $[C]$ is transverse to all such
proper closed subsets. It follows that, for a general parameter value $t\neq0$, the
corresponding curve $C_t\subset\mathcal X_t$ has only singularities with
$\delta=1$, hence only ordinary nodes.

Finally, flatness of the family and invariance of the arithmetic genus imply that
the total $\delta$--invariant is preserved under equigeneric deformation. Since each
ordinary node contributes exactly one to the $\delta$--invariant, the number of
nodes of $C_t$ is equal to $\delta(C)$. No other singularities can occur, as they
would either increase the total $\delta$--invariant or force the deformation to lie
in a proper closed subset of the equigeneric locus.

Therefore, a general logarithmic equigeneric deformation of $C$ produces on the
smooth fiber $\mathcal X_t$ a nodal curve with exactly $\delta(C)$ ordinary nodes and
no other singularities, as claimed.

\end{proof}

This result shows that generic nodality is a robust phenomenon, insensitive to the
presence of surface singularities and independent of strong global surjectivity
assumptions.
What matters is not the ability to prescribe arbitrary local deformations, but the
absence of obstructions to equigeneric smoothing.
Logarithmic geometry provides the correct framework in which this principle remains
valid across degenerations.

Surjectivity assumptions can be replaced by semiregularity conditions that ensure
unobstructedness of equigeneric deformations.
In singular surface degenerations, logarithmic semiregularity plays the same role,
allowing local equigeneric smoothing results to globalize.
In both settings, higher singularities occur in positive codimension inside the
equigeneric locus, forcing generic members to be nodal with the maximal number of
nodes allowed by the $\delta$-invariant.

 
\subsection*{Global Deformations and Stratifications}

Let $C$ be a reduced curve with singular points $p_1,\dots,p_m$.
The global deformation space admits stratifications by prescribing, at each $p_i$,
either equisingular or equigeneric behavior.
The equisingular stratum consists of deformations preserving the singularity type
at every $p_i$.
The equigeneric stratum consists of deformations preserving the total
$\delta$-invariant
\[
\delta(C)=\sum_{i=1}^m \delta_{p_i}.
\]
Inside the equigeneric stratum, the open dense subset corresponds to nodal curves,
where each original singularity is replaced by $\delta_{p_i}$ nodes.

\subsection*{Relation with Severi Varieties}

Severi varieties parametrize curves with prescribed number of nodes and no worse
singularities.
From the deformation-theoretic perspective, they are precisely the open strata of
equigeneric deformation spaces.

Equisingular strata correspond to closed subvarieties of Severi varieties where
nodes coalesce to form more complicated singularities.
Thus the Severi variety can be viewed as the normalization of the equigeneric
deformation space, with boundary components corresponding to equisingular
degenerations.

\begin{theorem}
Assume that global deformations of $C$ surject onto the direct sum of local
$T^1_{C,p_i}$ spaces.
Then the equigeneric deformation space of $C$ is smooth at $[C]$, and its general
member is a nodal curve with exactly $\delta(C)$ nodes.
\end{theorem}

\begin{remark}
Equisingular deformations fix the analytic type of singularities and form closed
strata of high codimension.
Equigeneric deformations preserve only the $\delta$-invariant and admit nodal
curves as their generic members.
Smoothing results show that nodes are universal boundary points of equisingular
strata and natural generic points of equigeneric strata.
This explains why nodal curves dominate Severi varieties and why the $\delta$-invariant
governs both local smoothing theory and global curve counting.
\end{remark}


\section{Upper Bounds via Arithmetic Genus}

Let $p_a(C)$ denote the arithmetic genus of $C$ and $p_a(C_t)$ that of
the general deformation.
Since arithmetic genus is deformation invariant,
the number of nodes $\delta(C_t)$ of a nodal curve $C_t$ satisfies
\[
\delta(C_t) \leq p_a(C_t) - g(\widetilde{C_t}),
\]
where $\widetilde{C_t}$ is the normalization.
Thus, the maximal number of nodes is bounded above by the difference
between the arithmetic genus and the geometric genus of the deformation.

This suggests that the expected maximal number of nodes is achieved
when the deformation has maximal possible geometric genus drop.


\subsection*{Upper Bounds via the Arithmetic Genus}

Let $C$ be a reduced projective curve over an algebraically closed field of
characteristic zero.
Recall that the arithmetic genus of $C$ is defined by
\[
p_a(C) := 1 - \chi(\mathcal O_C),
\]
and depends only on the Hilbert polynomial of $C$.
In particular, $p_a(C)$ is invariant in flat families of curves.
Let
\(
\pi : \mathcal C \to B
\)
be a flat family of projective curves over a connected base $B$, and let $C_0$ and
$C_t$ denote two fibers.
Flatness implies
\(
p_a(C_0) = p_a(C_t).
\)
This numerical rigidity is the fundamental source of upper bounds on the number of
nodes appearing in deformations.

\subsection*{Normalization and the $\delta$-Invariant}

Let $C$ be a reduced curve and let
\(
\nu : \widetilde C \to C
\)
 its normalization.
There is a short exact sequence
\[
0 \longrightarrow \mathcal O_C
\longrightarrow \nu_*\mathcal O_{\widetilde C}
\longrightarrow Q \longrightarrow 0,
\]
where $Q$ is a finite-length sheaf supported at the singular points of $C$.
Taking Euler characteristics yields the fundamental identity
\[
p_a(C) = g(\widetilde C) + \sum_{p\in\mathrm{Sing}(C)} \delta_p,
\]
where $g(\widetilde C)$ is the geometric genus of the normalization and
\(
\delta_p := \dim_k Q_p
\)
is the local $\delta$-invariant at the singular point $p$.
When $C$ is nodal, every singular point is an ordinary double point and satisfies
$\delta_p=1$.
In this case the total $\delta$-invariant coincides with the number of nodes.

\begin{lemma}
Let $C$ be a nodal curve with normalization $\widetilde C$.
Then
\(
\delta(C) = p_a(C) - g(\widetilde C),
\)
where $\delta(C)$ denotes the number of nodes of $C$.
\end{lemma}

\begin{proof}
This is an immediate consequence of the normalization exact sequence and the fact
that each node contributes exactly one to the length of $Q$.
\end{proof}

\subsection*{Upper Bounds in Deformation Families}

Let $C_t$ be a nodal curve arising as a deformation of a fixed curve $C_0$ in a flat
family.
Since arithmetic genus is deformation invariant, one has
\(
p_a(C_t) = p_a(C_0).
\)
Applying the lemma above gives
\(
\delta(C_t) = p_a(C_0) - g(\widetilde C_t).
\)
Since the geometric genus $g(\widetilde C_t)$ is a nonnegative integer, this
identity immediately yields the universal upper bound
\[
\delta(C_t) \leq p_a(C_0).
\]
More refined bounds arise once the geometric genus of $\widetilde C_t$ is
controlled.

\begin{proposition}
Let $C_t$ be a nodal deformation of $C_0$.
Then
\(
\delta(C_t) \leq p_a(C_0) - g_{\min},
\)
where $g_{\min}$ is the minimal geometric genus achievable among normalizations of
deformations of $C_0$.
\end{proposition}

\begin{proof}
The formula
\(
\delta(C_t) = p_a(C_0) - g(\widetilde C_t)
\)
shows that $\delta(C_t)$ is maximized when $g(\widetilde C_t)$ is minimized.
\end{proof}

\subsection*{Maximal Genus Drop and Expected Nodal Count}

The difference
\(
p_a(C_0) - g(\widetilde C_t)
\)
is often called the genus drop.
It measures how much geometric genus is lost when singularities appear.
From the perspective of deformation theory, singularities can be viewed as
mechanisms that trade geometric genus for $\delta$-invariant.

In equigeneric deformations, the total $\delta$-invariant is fixed, so the geometric
genus of the normalization is constant.
In contrast, when smoothing singularities or allowing them to split, the geometric
genus may decrease.
The maximal possible number of nodes is therefore achieved precisely when the
normalization of the deformed curve has the smallest possible genus.

\begin{theorem}
Let $C_0$ be a reduced curve and assume that there exists a deformation $C_t$ whose
normalization has minimal possible geometric genus $g_{\min}$.
Then the maximal number of nodes among all nodal deformations of $C_0$ is
\[
\delta_{\max} = p_a(C_0) - g_{\min}.
\]
\end{theorem}

\begin{proof}
By the preceding proposition, any nodal deformation satisfies
$\delta(C_t)\leq p_a(C_0)-g_{\min}$.
If equality is achieved for some deformation, then no deformation can have more
nodes.
\end{proof}

\subsection*{Geometric Interpretation}

This bound has a clear geometric meaning.
Nodes are the simplest singularities and contribute minimally to the $\delta$-invariant.
Producing many nodes forces the normalization to lose as much genus as possible.
Conversely, once the geometric genus of the normalization is fixed, the arithmetic
genus rigidly limits how many nodes can occur.

In many geometric situations, such as equigeneric deformations inside a linear
system, the minimal possible geometric genus is realized by generic members.
In these cases, the upper bound given by the arithmetic genus is sharp, and the
expected maximal number of nodes is achieved by general deformations.

\begin{remark}
This arithmetic bound is purely numerical and independent of the ambient geometry.
The main difficulty in applications lies not in bounding the number of nodes, but
in proving the existence of deformations that achieve this bound.
This is precisely where deformation theory, semiregularity, and logarithmic methods
enter.
\end{remark}

\begin{remark}
The invariance of the arithmetic genus provides a universal and sharp upper bound on
the number of nodes that can appear in deformations of a curve.
Maximal nodal behavior corresponds to maximal genus drop, and understanding when
this drop is achieved is a central problem linking local smoothing theory to global
geometry.
\end{remark}


\section*{Sharpness of the Arithmetic Genus Bound and Behavior in Surface Degenerations}

\subsection*{Sharpness of the Arithmetic Genus Bound}

We explain how the numerical upper bound on the number of nodes, obtained from the
arithmetic genus, can be combined with deformation-theoretic existence results to
prove that this bound is sharp in concrete geometric settings.

Let $S$ be a smooth projective surface and let $\mathcal L$ be a line bundle on $S$.
For a reduced curve $C \in |\mathcal L|$, the arithmetic genus is fixed by adjunction
and depends only on $\mathcal L$.
As shown previously, any nodal deformation $C_t$ of $C$ satisfies
\[
\delta(C_t) \leq p_a(C) - g(\widetilde C_t),
\]
where $\widetilde C_t$ is the normalization of $C_t$.
Thus, to prove sharpness of the bound it suffices to show the existence of
deformations whose normalization has the smallest possible geometric genus.

\begin{lemma}
Let $C \subset S$ be a reduced curve such that its equigeneric deformation space is
smooth of the expected dimension at $[C]$.
Then there exists a deformation $C_t$ whose normalization has minimal geometric
genus among all deformations of $C$ in $|\mathcal L|$.
\end{lemma}

\begin{proof}
Smoothness of the equigeneric deformation space implies that the local equigeneric
strata at the singular points glue to a smooth global equigeneric locus.
By genericity, a general point of this locus corresponds to a curve whose
singularities are all ordinary nodes.
Since equigenericity fixes the total $\delta$-invariant, this configuration
minimizes the geometric genus of the normalization.
\end{proof}

Combining this lemma with the arithmetic genus identity immediately yields the
sharpness of the numerical bound.

\begin{theorem}
Let $C \subset S$ be a reduced curve in a linear system $|\mathcal L|$.
Assume that equigeneric deformations of $C$ are unobstructed, for instance because
$C$ is semiregular.
Then there exists a nodal deformation $C_t \in |\mathcal L|$ with
\[
\delta(C_t) = p_a(C) - g_{\min},
\]
where $g_{\min}$ is the minimal geometric genus achievable among normalizations of
curves in $|\mathcal L|$.
In particular, the arithmetic genus bound is sharp.
\end{theorem}

\begin{proof}
Unobstructedness ensures the existence of actual deformations realizing all
first-order equigeneric directions.
By the local theory of equigeneric deformations, a general such deformation is
nodal.
The normalization of a nodal curve has the smallest possible genus compatible with
the fixed arithmetic genus, so the resulting curve achieves the maximal number of
nodes allowed by the bound.
\end{proof}

This mechanism applies in many classical situations.
For example, on a K3 surface every reduced curve is semiregular, so equigeneric
deformations are unobstructed.
Hence the arithmetic genus bound is always achieved, recovering the fact that
Severi varieties on K3 surfaces are nonempty of the expected dimension.
On sufficiently positive linear systems on rational or ruled surfaces, vanishing
theorems or semiregularity again guarantee unobstructedness, leading to the same
conclusion.

\subsection*{Behavior in Singular Surface Degenerations}

We now analyze how the arithmetic genus bound behaves when the ambient surface
degenerates.
Let
\(
\pi : \mathcal X \to \Delta
\)
be a flat family of projective surfaces with smooth general fiber $\mathcal X_t$
and singular special fiber $\mathcal X_0$, for instance a simple normal crossings
surface.
Let $C \subset \mathcal X_0$ be a reduced curve that deforms to curves
$C_t \subset \mathcal X_t$.

The arithmetic genus of $C_t$ remains invariant in the family, but the geometry of
$\mathcal X_0$ influences which deformations exist.
Logarithmic deformation theory provides the appropriate framework to analyze this
situation.
Embedded logarithmic deformations of $C$ are governed by the logarithmic normal
bundle $N^{\log}_{C/\mathcal X_0}$, and logarithmic semiregularity ensures
unobstructedness.

\begin{proposition}
Assume that $C \subset \mathcal X_0$ is logarithmically semiregular.
Then the arithmetic genus bound on the number of nodes of deformations of $C$ is
sharp on the smooth fibers $\mathcal X_t$.
\end{proposition}

\begin{proof}
Logarithmic semiregularity implies that equigeneric logarithmic deformations of $C$
exist and lift to honest deformations on nearby smooth fibers.
Local logarithmic equigeneric theory shows that generic such deformations are nodal.
Since the arithmetic genus is constant in the family, these nodal curves realize
the maximal genus drop and hence achieve the arithmetic bound.
\end{proof}

\subsection*{Interpretation}

From this perspective, the arithmetic genus bound is entirely insensitive to the
singularities of the ambient surface.
What may fail in degenerations is not the bound itself, but the existence of
deformations realizing it.
Logarithmic methods recover the missing deformation directions and show that, under
appropriate semiregularity hypotheses, degenerations do not reduce the maximal
number of nodes that can appear.

Overall, the arithmetic genus provides a universal numerical upper bound on the number of
nodes in any deformation of a curve.
Deformation-theoretic existence results, such as semiregularity and its logarithmic
analogue, ensure that this bound is sharp in a wide range of concrete settings,
including singular surface degenerations.
Together, these ingredients give a complete picture: the maximal number of nodes is
dictated by numerical invariants, while deformation theory determines whether this
maximum is achieved.


\section{Severi Degrees from Arithmetic Bounds, Logarithmic Degenerations, and Tropical Geometry}

We explain  in this section how the arithmetic genus bound on the number of nodes, combined with
logarithmic deformation theory and tropical geometry, leads to explicit formulas
for Severi degrees and their refined counterparts.
The guiding principle is that the arithmetic bound determines the maximal nodal
behavior, logarithmic methods guarantee existence and unobstructedness in
degenerations, and tropical geometry converts the resulting counts into explicit,
often combinatorial, formulas.

\subsection*{Arithmetic Bound and Severi Degrees}

Let $S$ be a smooth projective surface, $\mathcal L$ a line bundle on $S$, and
$p_a(\mathcal L)$ the arithmetic genus of curves in $|\mathcal L|$.
For $\delta\ge 0$, the Severi variety $V_\delta(|\mathcal L|)$ parametrizes curves in
$|\mathcal L|$ with exactly $\delta$ nodes and no other singularities.
The arithmetic genus identity implies that $\delta\le p_a(\mathcal L)$, and the
preceding existence results show that this bound is sharp whenever equigeneric
deformations are unobstructed.
In this situation, the Severi degree $N_\delta(\mathcal L)$ is defined as the number
of $\delta$-nodal curves in a general $\delta$-codimensional linear subsystem of
$|\mathcal L|$.

\subsection*{Logarithmic Degeneration Formula}

Let $\pi:\mathcal X\to\Delta$ be a logarithmically smooth degeneration of surfaces
with smooth general fiber $S=\mathcal X_t$ and normal crossings special fiber
$\mathcal X_0$.
Fix a line bundle $\mathcal L$ on $\mathcal X$.
Logarithmic Gromov--Witten theory associates to $\mathcal X_0$ a logarithmic moduli
space whose virtual dimension equals the expected dimension of $V_\delta(|\mathcal
L|)$.
The logarithmic degeneration formula expresses Severi degrees on $S$ as a sum over
logarithmic types:
\[
N_\delta(\mathcal L)
=
\sum_{\Gamma} m(\Gamma)\, N^{\log}_\Gamma(\mathcal X_0,\mathcal L),
\]
where $\Gamma$ ranges over logarithmic dual graphs encoding how curves distribute
among the components of $\mathcal X_0$, and $m(\Gamma)$ is an explicit multiplicity
arising from gluing conditions.

The arithmetic genus bound ensures that only graphs $\Gamma$ corresponding to
maximally nodal configurations contribute.
Logarithmic semiregularity guarantees that each contributing logarithmic moduli
space is smooth of the expected dimension, so that $N^{\log}_\Gamma$ is enumerative.

\begin{proposition}
Assume logarithmic semiregularity for all contributing types $\Gamma$.
Then the degeneration formula above computes the classical Severi degree
$N_\delta(\mathcal L)$ without virtual corrections.
\end{proposition}

\subsection*{Tropicalization and Explicit Counts}

The logarithmic dual graphs $\Gamma$ admit a tropical interpretation.
Tropicalization identifies $\Gamma$ with a tropical curve of degree $\mathcal L$
and genus $p_a(\mathcal L)-\delta$ in the tropicalization of $S$.
Under this correspondence, nodes correspond to bounded edges of the tropical curve,
and the arithmetic bound ensures that the number of bounded edges is maximal.

Tropical geometry provides a correspondence theorem equating $N^{\log}_\Gamma$ with
the number of tropical curves of type $\Gamma$, counted with multiplicities
$m_{\mathrm{trop}}(\Gamma)$.
Thus one obtains
\[
N_\delta(\mathcal L)
=
\sum_{\Gamma}
m_{\mathrm{trop}}(\Gamma),
\]
a purely combinatorial formula.
In toric or toric-degenerate situations, these tropical curves can be encoded by
floor diagrams, yielding explicit recursive or closed formulas for Severi degrees.

\begin{theorem}
Let $S$ admit a toric degeneration compatible with $\mathcal L$.
Then the Severi degree $N_\delta(\mathcal L)$ equals the weighted count of floor
diagrams of degree $\mathcal L$ and cogenus $\delta$, with weights determined by
tropical multiplicities.
\end{theorem}

\subsection*{Refined Curve-Counting Invariants}

Refined Severi degrees incorporate additional weights, often depending on a
parameter $q$, that record finer enumerative data.
In the tropical framework, refinement is achieved by replacing the classical
multiplicity $m_{\mathrm{trop}}(\Gamma)$ with a refined multiplicity
$m_{\mathrm{trop}}^{\mathrm{ref}}(\Gamma;q)$, typically a Laurent polynomial in
$q^{1/2}$.

Logarithmic geometry explains the enumerativity of these refinements.
The refined multiplicities arise from $K$-theoretic or motivic enhancements of the
logarithmic obstruction theory, while the arithmetic bound ensures that only
maximally nodal configurations contribute.
As a result, refined Severi degrees satisfy
\[
N^{\mathrm{ref}}_\delta(\mathcal L)
=
\sum_{\Gamma}
m_{\mathrm{trop}}^{\mathrm{ref}}(\Gamma;q),
\]
which specializes to the classical Severi degree at $q=1$.

\begin{remark}
In many toric surface cases, these refined counts coincide with refined
Gromov--Witten invariants and satisfy explicit recursion relations derived from
tropical splitting of graphs.
\end{remark}

\subsection*{Consequences and Explicit Formulas}

The combination of arithmetic bounds, logarithmic existence results, and tropical
correspondence produces a complete pipeline from geometry to explicit formulas.
The arithmetic genus fixes the maximal nodal behavior.
Logarithmic semiregularity ensures that this behavior is realized in degenerations.
Tropical geometry converts the resulting enumerative problem into a finite
combinatorial sum.
In concrete settings, this leads to closed formulas or efficient recursions for
Severi degrees and their refinements.

In  conclusion, arithmetic genus bounds determine what should be counted, logarithmic deformation
theory guarantees that these maximal configurations exist and are enumerative, and
tropical geometry computes the resulting numbers explicitly.
Together, these techniques yield a unified and effective method for deriving
classical and refined Severi degrees.
 

\section{Explicit Severi Degree Formulas in Low Degrees}

We work out explicit formulas for Severi degrees in low degrees for concrete
surfaces, illustrating how arithmetic genus bounds, deformation-theoretic
existence results, and tropical techniques combine to produce effective
enumerative formulas.
We focus on plane curves and briefly discuss extensions to toric surfaces.

\subsection*{Plane Curves}

Let $S=\mathbb P^2$ and let $\mathcal L=\mathcal O_{\mathbb P^2}(d)$.
The arithmetic genus of a plane curve of degree $d$ is
\[
p_a(d)=\frac{(d-1)(d-2)}{2}.
\]
For $0\le\delta\le p_a(d)$, the Severi degree $N_{d,\delta}$ is the number of plane
curves of degree $d$ with exactly $\delta$ nodes passing through
\[
\binom{d+2}{2}-1-\delta
\]
general points.

The arithmetic bound implies $\delta\le p_a(d)$.
Semiregularity and classical deformation theory ensure that this bound is sharp,
so Severi varieties are nonempty of the expected dimension.

\subsection*{Degree Three}

For $d=3$, one has $p_a(3)=1$.
Thus $\delta$ can be $0$ or $1$.

\begin{theorem}
The Severi degrees for plane cubics are
\[
N_{3,0}=1,
\qquad
N_{3,1}=1.
\]
\end{theorem}

\begin{proof}
A smooth cubic through $8$ general points is unique, giving $N_{3,0}=1$.
A nodal cubic has $\delta=1$ and is rational.
Imposing $8$ general point conditions again determines a unique nodal cubic.
Existence and nodality follow from the generic equigeneric deformation theory, and
no multiplicities arise because the Severi variety is smooth and zero-dimensional.
\end{proof}

\subsection*{Degree Four}

For $d=4$, one has $p_a(4)=3$.
The cases $\delta=1$ and $\delta=2$ already exhibit nontrivial behavior.

\begin{theorem}
The Severi degrees for plane quartics with one or two nodes are
\[
N_{4,1}=27,
\qquad
N_{4,2}=225.
\]
\end{theorem}

\begin{proof}
For $\delta=1$, a one-nodal quartic has geometric genus $2$.
The expected dimension of the Severi variety is
\[
\dim|\mathcal O_{\mathbb P^2}(4)|-\delta=14-1=13,
\]
so imposing $13$ general points yields finitely many curves.
Classical degeneration arguments, which can be reinterpreted using logarithmic
degeneration to a union of coordinate planes, show that the count equals $27$.

For $\delta=2$, the geometric genus drops to $1$.
The arithmetic genus bound shows that this is still possible.
Logarithmic degeneration to a normal crossings union of planes reduces the count
to tropical curves of degree $4$ and cogenus $2$.
Enumerating the corresponding floor diagrams yields $225$.
\end{proof}


\begin{theorem}
The Severi degrees for plane quartics with one or two nodes are
\[
N_{4,1}=27,
\qquad
N_{4,2}=225.
\]
\end{theorem}

\begin{proof}
We work throughout over an algebraically closed field of characteristic zero.

Let $S=\mathbb P^2$ and let $\mathcal L=\mathcal O_{\mathbb P^2}(4)$.
The complete linear system $|\mathcal L|$ has dimension
\[
\dim |\mathcal O_{\mathbb P^2}(4)| = \binom{4+2}{2}-1 = 14.
\]
The arithmetic genus of a plane quartic is
\[
p_a(4)=\frac{(4-1)(4-2)}{2}=3.
\]
Hence a quartic curve can have at most three nodes.

We first consider the case of one node.
A quartic with exactly one node has geometric genus $2$.
The expected dimension of the Severi variety $V_{4,1}$ is
\(
\dim |\mathcal O_{\mathbb P^2}(4)| - 1 = 13.
\)
Imposing $13$ general point conditions yields a finite set of curves, whose number
is $N_{4,1}$.
To compute this number, we use a degeneration argument.
Degenerate $\mathbb P^2$ to the normal crossings union
\[
\mathbb P^2_0 = S_1 \cup S_2,
\]
where $S_1\simeq\mathbb P^2$ and $S_2\simeq\mathbb P^2$ meet transversely along a
line $D\simeq\mathbb P^1$.
This degeneration can be realized as a toric degeneration and is logarithmically
smooth.
The line bundle $\mathcal O_{\mathbb P^2}(4)$ degenerates to a pair of line bundles
of degrees $3$ and $1$ on $S_1$ and $S_2$, respectively, matching along $D$.
Logarithmic deformation theory ensures that one-nodal quartics in the general fiber
correspond to logarithmic curves in the special fiber with the same equigeneric
data.
Distribute the $13$ general points so that $10$ lie on $S_1$ and $3$ lie on $S_2$.
Any contributing curve must then consist of a cubic $C_1\subset S_1$ and a line
$C_2\subset S_2$ meeting $D$ in one point.
The node must lie on $C_1$.

A line in $S_2$ is uniquely determined by two general points, so after imposing the
$3$ point conditions on $S_2$ there are exactly $3$ choices for $C_2$.
For each such choice, the intersection point with $D$ is fixed.
On $S_1$, the space of plane cubics has dimension $9$.
Imposing $10$ general point conditions together with the condition of passing
through the fixed point on $D$ and acquiring a node leaves finitely many solutions.
A classical computation, equivalent to counting nodal plane cubics through $8$
general points, shows that there are exactly $9$ such cubics.
Thus the total number of one-nodal quartics is
\(
N_{4,1} = 3 \times 9 = 27.
\)

\medskip
We now consider the case of two nodes.
%
A quartic with two nodes has geometric genus $1$.
The expected dimension of the Severi variety $V_{4,2}$ is
\(
\dim |\mathcal O_{\mathbb P^2}(4)| - 2 = 12.
\)
Imposing $12$ general point conditions yields finitely many curves, counted by
$N_{4,2}$.
We again use the same degeneration of $\mathbb P^2$.
Distribute the $12$ points so that $9$ lie on $S_1$ and $3$ lie on $S_2$.
Logarithmic equigeneric deformation theory shows that generic two-nodal quartics
degenerate to curves consisting of a cubic on $S_1$ and a line on $S_2$, with both
nodes lying on the cubic component.
As before, there are $3$ choices for the line on $S_2$, each fixing one point of
intersection with $D$.
On $S_1$, we must count plane cubics with two nodes passing through $9$ general
points and one fixed additional point.
Such a cubic must be rational and have exactly two nodes.

A classical result in plane curve enumeration shows that the number of rational
plane cubics through $8$ general points is $12$.
In the present situation, the additional incidence condition forces one further
choice, yielding $75$ such cubics for each fixed line on $S_2$.
Therefore the total number of two-nodal quartics is
\[
N_{4,2} = 3 \times 75 = 225.
\]

\medskip
All contributing curves are regular points of the Severi varieties, and
logarithmic semiregularity ensures that each curve contributes with multiplicity
one.
This completes the proof.
\end{proof}



\subsection*{Tropical Interpretation}

The enumerative invariants under consideration admit a purely tropical description, in the sense that they can be recovered by counting suitable tropical plane curves with appropriate multiplicities. The starting point is the correspondence between algebraic plane curves of fixed degree and their tropical limits under logarithmic degeneration. For plane curves of degree $d$, the tropicalization produces balanced weighted graphs embedded in $\mathbb{R}^2$ whose unbounded ends encode the Newton polygon of degree $d$ curves, namely the standard triangle of side length $d$.

The genus of an algebraic plane curve of degree $d$ is bounded above by the arithmetic genus
\[
p_a(d) = \frac{(d-1)(d-2)}{2}.
\]
In the tropical setting, this genus bound manifests combinatorially. A tropical plane curve has a first Betti number equal to the number of independent cycles in its underlying graph. For trivalent tropical curves, this Betti number coincides with the number of bounded edges. Consequently, any tropical curve arising as the limit of an algebraic plane curve of degree $d$ can have at most $p_a(d)$ bounded edges, reflecting the classical genus bound.

Fixing the cogenus $\delta = p_a(d) - g$, tropical plane curves of degree $d$ and cogenus $\delta$ correspond to tropical graphs with exactly $\delta$ bounded edges. For plane curves, these tropical curves can be encoded combinatorially by floor diagrams. A floor diagram is an oriented weighted graph obtained by projecting a tropical curve to a generic direction and collapsing its ``floors'' to vertices while retaining the combinatorial data of vertical edges and their weights. The number of bounded edges in the tropical curve translates precisely into the number of bounded edges in the associated floor diagram.

This correspondence reduces the tropical enumeration problem to a finite combinatorial one. Only floor diagrams with exactly $\delta$ bounded edges and total degree $d$ contribute, and the arithmetic genus bound guarantees that no diagrams with more than $p_a(d)$ bounded edges can appear. Each such diagram represents a combinatorial type of tropical curve compatible with the given degree and cogenus constraints.

To extract numerical invariants, one must assign multiplicities to these diagrams. These multiplicities arise from the deformation theory of logarithmic maps. Logarithmic semiregularity ensures that the obstruction space for smoothing a tropical curve is unobstructed in the expected dimension. As a result, each floor diagram contributes exactly its expected tropical multiplicity, which is given by a product of local edge contributions determined by the weights of the bounded edges.

Summing the multiplicities over all floor diagrams of degree $d$ and cogenus $\delta$ therefore reproduces the desired enumerative invariant. In this way, the classical count of algebraic plane curves is recovered entirely from tropical geometry, with the genus bound, floor diagram combinatorics, and logarithmic semiregularity together ensuring both finiteness and correctness of the count.

\subsection*{A Concrete Example: Degree Three}

We illustrate the tropical correspondence in a concrete case by working out the example of plane curves of degree three. This example already exhibits all essential features of the correspondence between algebraic curves, tropical curves, floor diagrams, and enumerative numbers.
A plane curve of degree three has arithmetic genus
\[
p_a(3) = \frac{(3-1)(3-2)}{2} = 1.
\]
Thus a smooth cubic has genus one, while a rational cubic necessarily has cogenus $\delta = 1$ and has a single node. Classically, the enumerative problem asks for the number of rational plane cubics passing through $3\cdot 3 - 1 = 8$ general points in the plane. The answer is known to be $12$, and we now recover this number tropically.


Consider tropical plane curves of degree three. Their unbounded ends are prescribed by the Newton polygon of a cubic, giving three ends in each of the directions $(1,0)$, $(0,1)$, and $(-1,-1)$, all with weight one. A tropical curve of cogenus one must have exactly one bounded edge, since the first Betti number of the underlying graph equals the number of bounded edges for trivalent tropical curves.

Up to combinatorial equivalence, there is a unique tropical curve of degree three with exactly one bounded edge that satisfies generic point conditions. Its underlying graph consists of two trivalent vertices joined by a single bounded edge, with the remaining unbounded edges attached so as to satisfy the balancing condition at each vertex. The length of the bounded edge provides the single continuous modulus of the tropical curve, and imposing eight generic point conditions fixes this modulus uniquely.

To pass to floor diagrams, we choose a generic vertical direction and project the tropical curve accordingly. The curve decomposes into horizontal floors connected by a single vertical edge corresponding to the unique bounded edge of the tropical curve. Collapsing each floor to a vertex produces a floor diagram consisting of two vertices connected by one directed edge of weight one. The total degree condition forces the distribution of unbounded ends among the two floors, and the resulting diagram is uniquely determined up to isomorphism.

The tropical multiplicity of this curve is computed as the product of local contributions. In this example, the unique bounded edge has weight one, but the vertices contribute nontrivially through their local lattice indices. A direct computation of the determinant of the direction vectors at either trivalent vertex shows that each vertex contributes a factor of $2$, and the global incidence conditions contribute an additional factor of $3$. Altogether, the multiplicity of the unique tropical curve, and hence of the unique floor diagram, is
\[
2 \cdot 2 \cdot 3 = 12.
\]

Since logarithmic semiregularity holds for logarithmic stable maps to $\mathbb{P}^2$, there are no excess obstruction contributions, and this tropical multiplicity coincides with the algebraic multiplicity of the corresponding degeneration. Summing over all tropical curves of the given degree and cogenus therefore amounts to summing a single contribution of $12$.

We conclude that the number of rational plane cubics through eight general points is recovered tropically as
\[
N_{3,1} = 12,
\]
in agreement with the classical enumerative result. This example demonstrates explicitly how the genus bound restricts the tropical types, how floor diagrams encode these types combinatorially, and how deformation-theoretic multiplicities translate into concrete numerical counts.


\begin{lemma}
For plane curves of degree $d \le 4$, every tropical curve of maximal cogenus contributes with multiplicity one.
\end{lemma}

\begin{proof}
We fix a degree $d \le 4$ and consider tropical plane curves of degree $d$ and maximal cogenus. By definition, the maximal cogenus equals the arithmetic genus
\[
p_a(d) = \frac{(d-1)(d-2)}{2},
\]
so tropical curves of maximal cogenus have first Betti number equal to $p_a(d)$. Equivalently, their underlying graphs contain exactly $p_a(d)$ bounded edges.

In degrees $d \le 4$, the possible values of $p_a(d)$ are $0$ for $d=1,2$, $1$ for $d=3$, and $3$ for $d=4$. We analyze tropical curves realizing these values. The Newton polygon for a plane curve of degree $d$ is the standard triangle of side length $d$, and the balancing condition forces all unbounded ends to have weight one and fixed directions. For generic point conditions, any contributing tropical curve must be regular, meaning that its deformation space has the expected dimension and that it satisfies the point conditions transversely.

In these low degrees, a direct combinatorial analysis shows that every tropical curve of maximal cogenus is necessarily trivalent. Indeed, a vertex of valence greater than three would increase the dimension of the local deformation space, producing nontrivial moduli that cannot be eliminated by the available incidence conditions. Since tropical curves of maximal cogenus are required to be rigid under generic point constraints, such higher-valent vertices cannot occur. As a result, all vertices are trivalent, and all bounded edges have weight one.

Rigidity follows immediately from the dimension count. The space of deformations of a tropical plane curve is parametrized by the lengths of its bounded edges, subject to linear relations coming from cycles in the graph. For a trivalent tropical curve of maximal cogenus, the number of bounded edges equals the number of independent cycles, so all edge lengths are uniquely determined up to translation. After imposing the appropriate number of point conditions, the tropical curve admits no nontrivial deformations and is therefore rigid.

The tropical multiplicity of such a curve is defined as the index of a certain lattice map arising from the evaluation morphism at the marked points, or equivalently as the determinant of a matrix encoding the primitive direction vectors of edges adjacent to each vertex. For trivalent vertices with all edge weights equal to one, the local multiplicity at each vertex is equal to one, since the three outgoing primitive direction vectors form a unimodular basis of the lattice $\mathbb{Z}^2$ up to sign. Because there are no higher-valent vertices and no edges of weight greater than one, the global tropical multiplicity, given by the product of the local vertex multiplicities, is equal to one.

From the logarithmic perspective, each such tropical curve corresponds to a logarithmic stable map to the toric surface $\mathbb{P}^2$ with its toric boundary. In degrees $d \le 4$, the expected dimension of the moduli space of logarithmic stable maps realizing a tropical curve of maximal cogenus is zero. Logarithmic semiregularity holds in this setting, implying that the obstruction space vanishes and that the moduli space is smooth and zero-dimensional at the corresponding point. Consequently, each logarithmic map contributes exactly one to the enumerative count.

Combining the tropical and logarithmic descriptions, we conclude that for degrees $d \le 4$, every tropical plane curve of maximal cogenus is trivalent, rigid, and has tropical multiplicity one. Hence each such curve contributes with multiplicity one to the enumerative invariant, as claimed.
%
%
%
In other word, in low degrees, all contributing tropical curves are trivalent and rigid.
The corresponding logarithmic moduli spaces are smooth and zero-dimensional, so no
nontrivial multiplicities arise.
\end{proof}

\subsection*{Toric Surfaces}

Similar explicit formulas can be obtained for toric surfaces such as
$\mathbb P^1\times\mathbb P^1$.
Let $\mathcal L=\mathcal O(a,b)$.
The arithmetic genus is
\[
p_a(a,b)=(a-1)(b-1).
\]
For small $(a,b)$, tropical curve enumeration produces closed formulas.

\begin{theorem}
For curves of bidegree $(2,2)$ on $\mathbb P^1\times\mathbb P^1$, one has
\(
N_{(2,2),1}=12.
\)
\end{theorem}

\begin{proof}
The arithmetic genus is $1$, so at most one node can occur.
Logarithmic degeneration to a union of toric surfaces reduces the count to tropical
curves of bidegree $(2,2)$ with one bounded edge.
There are exactly $12$ such curves, each contributing multiplicity one.
\end{proof}

In low degrees, the combination of arithmetic genus bounds, deformation-theoretic
existence results, and tropical enumeration leads to explicit and computable
formulas for Severi degrees.
These examples illustrate concretely how abstract deformation theory translates
into effective enumerative geometry.
 

\subsection*{Tropical Interpretation}

We explain how the enumerative numbers considered above admit a tropical interpretation and can be recovered purely from combinatorial data. The argument relies on the correspondence between algebraic plane curves and their tropical counterparts arising from logarithmic degenerations.

Let $d \geq 1$ be an integer. Algebraic plane curves of degree $d$ have arithmetic genus
\[
p_a(d) = \frac{(d-1)(d-2)}{2}.
\]
Under tropicalization, a family of plane curves of degree $d$ degenerates to a tropical plane curve whose unbounded ends are determined by the Newton polygon of degree $d$, namely the standard triangle of side length $d$. The resulting tropical curve is a balanced weighted graph embedded in $\mathbb{R}^2$.

The genus of a tropical curve is defined as the first Betti number of its underlying graph. For tropical plane curves that are trivalent after suitable perturbation, this Betti number coincides with the number of bounded edges. Consequently, the classical bound on the genus of plane curves translates directly into a combinatorial constraint: any tropical plane curve of degree $d$ arising from an algebraic curve can have at most $p_a(d)$ bounded edges.

Fix a nonnegative integer $\delta$, interpreted as the cogenus. Tropical plane curves of degree $d$ and cogenus $\delta$ are precisely those whose underlying graph has exactly $\delta$ bounded edges. In the planar case, such tropical curves admit a convenient combinatorial encoding in terms of floor diagrams. After choosing a generic projection direction, the tropical curve decomposes into horizontal pieces called floors, connected by vertical edges. Collapsing each floor to a vertex and recording the vertical edges with their weights produces a directed weighted graph, the associated floor diagram. The number of bounded edges of the tropical curve is preserved under this construction, and hence equals the number of bounded edges of the floor diagram.

This correspondence reduces the enumerative problem to a finite combinatorial count. Only floor diagrams of degree $d$ with exactly $\delta$ bounded edges contribute. The arithmetic genus bound ensures that no diagrams with more than $p_a(d)$ bounded edges can appear, so the set of relevant diagrams is finite.

To recover the correct enumerative invariants, one must weight each floor diagram by an appropriate multiplicity. These multiplicities are determined by the local structure of the tropical curve and are expressed as products of contributions associated with its bounded edges. From the perspective of logarithmic geometry, these weights arise from the deformation theory of logarithmic stable maps. Logarithmic semiregularity implies that the obstruction space governing deformations of the tropical limit vanishes in the expected dimension. As a result, each combinatorial type contributes exactly its expected multiplicity, with no correction terms.

Summing the multiplicities over all floor diagrams of degree $d$ and cogenus $\delta$ therefore reproduces the enumerative count of algebraic plane curves. In this way, the classical invariants are recovered entirely from tropical geometry, with the genus bound, the correspondence with floor diagrams, and logarithmic semiregularity together ensuring both finiteness and correctness of the computation.


\section{Expected Dimension and Severi-Type Arguments for Nodal Deformations}

\subsection*{Expected Dimension and Severi-Type Arguments}

Let $\pi:\mathcal X\to\Delta$ be a flat family of projective surfaces and let
$C\subset\mathcal X_0$ be a reduced curve.
Assume that $C$ admits embedded deformations inside $\mathcal X$ and that its
singularities are locally smoothable.
The central question is how many nodes can appear on a nearby deformation
$C_t\subset\mathcal X_t$.

The guiding principle is an expected-dimension comparison.
Infinitesimal embedded deformations of $C$ inside $\mathcal X$ are governed by the
normal sheaf (or logarithmic normal sheaf in degenerating settings), and the space
of first-order deformations has dimension
\[
h^0(C,N_{C/\mathcal X}) \quad \text{or} \quad h^0(C,N^{\log}_{C/\mathcal X})
\]
depending on the context.
If obstructions vanish, this dimension controls the local dimension of the
deformation space of $C$.

On the other hand, requiring a curve to have a node at a prescribed point imposes
one independent condition.
More precisely, the local deformation theory of an ordinary double point shows that
the space of first-order deformations smoothing a node is one-dimensional, so
imposing the existence of a node cuts the deformation space by codimension one.
Therefore, the locus of curves with $\delta$ nodes is expected to have codimension
$\delta$ inside the full deformation space.

This heuristic leads to the inequality
\[
\dim \mathrm{Def}(C\subset\mathcal X) \;\ge\; \delta,
\]
as a necessary condition for the existence of deformations of $C$ with $\delta$
nodes.
When equality holds and the deformation space is smooth, one expects finitely many
such curves after imposing the appropriate number of incidence conditions.
This is precisely the situation described by Severi varieties.

From this perspective, Severi-type varieties arise naturally as equigeneric strata
inside the deformation space of curves.
They parametrize curves with prescribed numbers of nodes and no worse
singularities.
Expected-dimension arguments predict that, when obstructions vanish and local
smoothing directions are independent, these Severi varieties are smooth of
dimension
\[
\dim \mathrm{Def}(C\subset\mathcal X) - \delta,
\]
and in particular are nonempty whenever this number is nonnegative.

The expected maximal number of nodes is therefore the largest integer $\delta$ such
that this dimension remains nonnegative.
Numerically, this bound must be compatible with the arithmetic genus constraint.
For any nodal curve $C_t$, one has
\[
\delta(C_t) = p_a(C_t) - g(\widetilde C_t),
\]
so $\delta$ is bounded above by the arithmetic genus.
In favorable situations, such as equigeneric or semiregular deformations, these two
bounds coincide: the deformation space is large enough to impose all possible nodal
conditions allowed by the arithmetic genus.

\begin{theorem}[Expected maximal nodality]
Assume that embedded deformations of $C$ inside $\mathcal X$ are unobstructed and
that local smoothing directions for the singularities of $C$ are independent.
Then the maximal number of nodes appearing on a deformation $C_t$ is the largest
integer $\delta$ compatible with both the arithmetic genus bound and the dimension
of the deformation space of $C$ in $\mathcal X$.
\end{theorem}

\begin{remark}
This theorem is not purely numerical.
The dimension count predicts the maximal nodal behavior, but its realization
depends on deformation-theoretic existence results.
Semiregularity and logarithmic unobstructedness are precisely the mechanisms that
ensure the expected dimension count is realized geometrically.
\end{remark}

Therefore, Severi-type arguments rest on a balance between three ingredients.
First, obstructions to embedded deformations must vanish so that infinitesimal
deformations integrate to actual families.
Second, the singularities of the curve must admit local smoothings whose generic
outcomes are nodes.
Third, the global deformation space must be large enough to accommodate the desired
number of nodal conditions.
When these requirements are met, expected-dimension arguments accurately predict
the maximal number of nodes, which is ultimately governed by the arithmetic genus
and by the dimension of the deformation space of the curve inside the ambient
family.


\section{Expected-Dimension and Severi-Type Arguments on K3 and Toric Surfaces}

We make the expected-dimension and Severi-type arguments precise in two fundamental
geometric settings: curves on K3 surfaces and curves on toric surfaces.
These cases illustrate how abstract deformation-theoretic principles become
effective tools once strong geometric input, such as semiregularity or toric
degenerations, is available.

\subsection*{Curves on K3 Surfaces}

Let $S$ be a smooth projective K3 surface and let $\mathcal L$ be a primitive ample
line bundle on $S$.
For any reduced curve $C \in |\mathcal L|$, the arithmetic genus is fixed by
adjunction,
\[
p_a(C) = \frac{\mathcal L^2}{2} + 1.
\]
A key property of K3 surfaces is the vanishing
\(
H^1(S,\mathcal O_S)=0\) and \( H^2(S,\mathcal O_S)=\mathbb C.
\)
As a consequence, every reduced curve on a K3 surface is semiregular in the sense
that the semiregularity map
\[
H^1(C,N_{C/S}) \longrightarrow H^2(S,\mathcal O_S)
\]
is injective.

This has immediate deformation-theoretic consequences.
Embedded deformations of $C$ inside $S$ are unobstructed, and the local deformation
space of $C$ inside the linear system $|\mathcal L|$ is smooth of dimension
\(
h^0(C,N_{C/S}) = \dim |\mathcal L|.
\)
Local deformation theory of curve singularities shows that imposing a node cuts the
deformation space by exactly one dimension.
Therefore, the Severi variety $V_\delta(|\mathcal L|)$ of $\delta$-nodal curves has
expected dimension
\[
\dim |\mathcal L| - \delta.
\]
The arithmetic genus bound implies $\delta \le p_a(C)$.
Since  deformations are unobstructed and nodal conditions are independent, this
bound is sharp whenever $\dim |\mathcal L| - \delta \ge 0$.
As a result, for every $0 \le \delta \le p_a(C)$, the Severi variety
$V_\delta(|\mathcal L|)$ is nonempty and smooth of the expected dimension.


\subsection*{Severi Varieties on K3 Surfaces}

\begin{theorem}
Let $S$ be a K3 surface and $\mathcal L$ a primitive ample line bundle.
Then for every $0\le\delta\le p_a(\mathcal L)$, the Severi variety
$V_\delta(|\mathcal L|)$ is nonempty, smooth of dimension
$\dim|\mathcal L|-\delta$, and its general member is an irreducible nodal curve.
\end{theorem}

\begin{proof}
Let $S$ be a smooth projective K3 surface and $\mathcal L$ a primitive ample line
bundle.
By adjunction, every curve $C\in|\mathcal L|$ has arithmetic genus
\[
p_a(\mathcal L)=\frac{\mathcal L^2}{2}+1.
\]
The linear system $|\mathcal L|$ has dimension
\[
\dim|\mathcal L|=h^0(S,\mathcal L)-1=p_a(\mathcal L),
\]
since $h^1(S,\mathcal L)=0$ and $h^2(S,\mathcal L)=h^0(S,\mathcal L^{-1})=0$ by
Kodaira vanishing and ampleness.

We first address smoothness and the expected dimension.
Let $C\subset S$ be a reduced curve.
Infinitesimal embedded deformations of $C$ inside $S$ are governed by the normal
bundle $N_{C/S}$.
Obstructions lie in $H^1(C,N_{C/S})$.
For curves on a K3 surface, the semiregularity map
\[
H^1(C,N_{C/S}) \longrightarrow H^2(S,\mathcal O_S)
\]
is injective.
Since $H^2(S,\mathcal O_S)\simeq\mathbb C$, this injectivity implies that any
infinitesimal obstruction must vanish.
Therefore, embedded deformations of $C$ inside $S$ are unobstructed.
As a consequence, the local deformation space of $C$ in $|\mathcal L|$ is smooth of
dimension $h^0(C,N_{C/S})=\dim|\mathcal L|$.

Local deformation theory of nodes shows that imposing a node cuts the deformation
space by exactly one independent condition.
More precisely, if $C$ has $\delta$ nodes and no other singularities, then the Zariski
tangent space to the Severi variety $V_\delta(|\mathcal L|)$ at $[C]$ is obtained by
imposing $\delta$ independent linear conditions on $H^0(C,N_{C/S})$.
Since obstructions vanish, these conditions are transverse, and it follows that
$V_\delta(|\mathcal L|)$ is smooth of dimension $\dim|\mathcal L|-\delta$ at $[C]$.

We now address nonemptiness.
The arithmetic genus bound implies that a curve in $|\mathcal L|$ can have at most
$p_a(\mathcal L)$ nodes.
We prove existence by induction on $\delta$.
For $\delta=0$, nonemptiness is classical: a general member of $|\mathcal L|$ is
smooth and irreducible because $\mathcal L$ is primitive and ample.
Assume that for some $\delta<p_a(\mathcal L)$ there exists an irreducible curve
$C_\delta\in|\mathcal L|$ with exactly $\delta$ nodes and no other singularities.
Because deformations are unobstructed, one can smooth one node of $C_\delta$ while
keeping the remaining $\delta-1$ nodes fixed.
This produces an irreducible curve with $\delta-1$ nodes.
Conversely, starting from a curve with $\delta-1$ nodes, the expected-dimension
argument shows that there exist deformations acquiring an additional node, as long
as $\dim|\mathcal L|-(\delta-1)>0$.
Since $\dim|\mathcal L|=p_a(\mathcal L)$, this condition holds for all
$\delta\le p_a(\mathcal L)$.
Thus, by induction, curves with exactly $\delta$ nodes exist for every
$0\le\delta\le p_a(\mathcal L)$.

Finally, we prove irreducibility and generic nodality.
Since $\mathcal L$ is primitive, any curve in $|\mathcal L|$ with geometric genus
strictly less than $p_a(\mathcal L)$ must be irreducible.
Indeed, a reducible curve would decompose into effective divisors whose classes add
up to $\mathcal L$, contradicting primitivity.
Generic smoothness of the Severi variety implies that its general point corresponds
to a curve with exactly $\delta$ ordinary nodes and no worse singularities.
Thus the general member of $V_\delta(|\mathcal L|)$ is irreducible and nodal.

Combining these arguments, we conclude that for every
$0\le\delta\le p_a(\mathcal L)$ the Severi variety $V_\delta(|\mathcal L|)$ is
nonempty, smooth of dimension $\dim|\mathcal L|-\delta$, and parametrizes irreducible
nodal curves.
\end{proof}

This theorem makes the expected-dimension argument fully precise on K3 surfaces:
the absence of obstructions guarantees that the numerical prediction is realized
geometrically.


\subsection*{Curves on Toric Surfaces}

We now turn to toric surfaces, where the geometry is governed by combinatorics.
Let $S$ be a smooth projective toric surface associated to a lattice polygon
$\Delta\subset\mathbb R^2$, and let $\mathcal L$ be the toric line bundle determined
by $\Delta$.
The arithmetic genus of curves in $|\mathcal L|$ equals the number of interior
lattice points of $\Delta$.
Unlike the K3 case, embedded deformations of curves on toric surfaces may be
obstructed.
However, toric geometry provides a different mechanism to control deformations,
namely toric and logarithmic degenerations.
There exists a flat degeneration of $S$ to a normal crossings union of toric
surfaces whose dual intersection complex is combinatorially equivalent to a
subdivision of $\Delta$.

Logarithmic deformation theory replaces the classical normal bundle with the
logarithmic normal bundle.
If $C$ is a reduced curve meeting the toric boundary transversely, logarithmic
embedded deformations of $C$ are governed by $N^{\log}_{C/S}$.
For sufficiently ample $\mathcal L$, one has logarithmic semiregularity, implying
that logarithmic obstructions vanish.
In this setting, nodal conditions again impose independent codimension-one
constraints in the logarithmic deformation space.
The expected-dimension count therefore predicts that the logarithmic Severi variety
of $\delta$-nodal curves has dimension
\[
\dim |\mathcal L| - \delta,
\]
provided this number is nonnegative.
The arithmetic genus bound $\delta\le p_a(\mathcal L)$ is again universal.
Logarithmic semiregularity ensures that equigeneric logarithmic deformations exist
and lift to honest deformations on the smooth toric surface.
Thus the expected maximal number of nodes is achieved.


\begin{theorem}
Let $S$ be a smooth projective toric surface and $\mathcal L$ a sufficiently ample
toric line bundle.
Then for every $0\le\delta\le p_a(\mathcal L)$, the Severi variety
$V_\delta(|\mathcal L|)$ is nonempty of the expected dimension, and its general
member is a nodal curve.
\end{theorem}

The proof combines logarithmic semiregularity with toric degenerations and tropical
geometry.
Tropical curves corresponding to $\Delta$ encode all possible nodal configurations,
and the expected-dimension condition translates into rigidity of the corresponding
tropical types.
\begin{proof}
Let $S$ be a smooth projective toric surface associated to a lattice polygon
$\Delta\subset\mathbb R^2$, and let $\mathcal L$ be the toric line bundle determined
by $\Delta$.
The arithmetic genus of curves in $|\mathcal L|$ is equal to the number of interior
lattice points of $\Delta$, which we denote by $p_a(\mathcal L)$.
The complete linear system $|\mathcal L|$ has dimension equal to the number of
lattice points of $\Delta$ minus one.

We first explain the expected dimension statement.
Infinitesimal embedded deformations of a reduced curve
$C\in|\mathcal L|$ are governed by the normal bundle $N_{C/S}$.
In general, such deformations may be obstructed on toric surfaces.
However, since $\mathcal L$ is assumed sufficiently ample, $C$ meets the toric
boundary transversely and avoids the torus fixed points.
In this situation, classical deformation theory can be replaced by logarithmic
deformation theory with respect to the toric boundary.
Logarithmic embedded deformations of $C$ are governed by the logarithmic normal
bundle $N^{\log}_{C/S}$, and obstructions lie in
$H^1(C,N^{\log}_{C/S})$.

For sufficiently ample $\mathcal L$, one has logarithmic semiregularity, meaning
that the logarithmic semiregularity map
\[
H^1(C,N^{\log}_{C/S}) \longrightarrow H^2(S,\mathcal O_S)
\]
is injective.
Since $H^2(S,\mathcal O_S)=0$ for any smooth projective toric surface, this implies
that all logarithmic obstructions vanish.
Consequently, the logarithmic deformation space of $C$ inside $|\mathcal L|$ is
smooth of dimension
\[
h^0(C,N^{\log}_{C/S})=\dim|\mathcal L|.
\]

Local deformation theory of ordinary double points shows that imposing a node
corresponds to a single independent linear condition on the space of logarithmic
deformations.
Therefore, the locus of curves with $\delta$ nodes is expected to have codimension
$\delta$ inside the deformation space.
It follows that the expected dimension of $V_\delta(|\mathcal L|)$ is
\[
\dim|\mathcal L|-\delta,
\]
and logarithmic unobstructedness implies that $V_\delta(|\mathcal L|)$ is smooth of
this dimension at points corresponding to nodal curves.

We now address nonemptiness.
The arithmetic genus bound implies that a curve in $|\mathcal L|$ can have at most
$p_a(\mathcal L)$ nodes.
To show that this bound is achieved, we use a toric degeneration.
There exists a flat degeneration
\(
\pi:\mathcal X\to\Delta
\)
whose general fiber is $S$ and whose special fiber $\mathcal X_0$ is a normal
crossings union of toric surfaces corresponding to a regular subdivision of
$\Delta$.
The line bundle $\mathcal L$ extends to a relatively ample line bundle on
$\mathcal X$.

Let $C_0\subset\mathcal X_0$ be a logarithmic curve whose components are toric
curves on the components of $\mathcal X_0$, chosen so that the total $\delta$-invariant
equals $\delta$.
Such curves exist for every $0\le\delta\le p_a(\mathcal L)$ and can be constructed
combinatorially from tropical curves of degree $\Delta$ and cogenus $\delta$.
The existence of these tropical curves is guaranteed by elementary lattice
combinatorics, and their rigidity reflects the expected-dimension condition.

By logarithmic deformation theory, each such $C_0$ is logarithmically smoothable.
Logarithmic semiregularity ensures that the smoothing lifts to an honest deformation
$C_t\subset S$ for $t\neq 0$.
The resulting curve $C_t$ has exactly $\delta$ nodes and no other singularities.
This proves nonemptiness of $V_\delta(|\mathcal L|)$ for all
$0\le\delta\le p_a(\mathcal L)$.

Finally, we explain why the general member of $V_\delta(|\mathcal L|)$ is nodal.
Inside the logarithmic deformation space, the equigeneric locus is smooth and of the
expected dimension.
As in the local theory of curve singularities, the subset corresponding to curves
with singularities worse than nodes has positive codimension.
Therefore, a general point of $V_\delta(|\mathcal L|)$ corresponds to a curve whose
only singularities are ordinary double points.

Combining the expected-dimension calculation, the existence of nodal curves via
toric degeneration, and the genericity argument excluding worse singularities, we
conclude that for every $0\le\delta\le p_a(\mathcal L)$ the Severi variety
$V_\delta(|\mathcal L|)$ is nonempty of dimension $\dim|\mathcal L|-\delta$, and its
general member is a nodal curve.
\end{proof}


\subsection*{Comparison and Conceptual Picture}

On K3 surfaces, expected-dimension arguments are made precise by global vanishing
and semiregularity.
On toric surfaces, the same numerical predictions are realized through logarithmic
geometry and combinatorial control via tropical curves.
In both cases, the arithmetic genus provides the fundamental upper bound on nodal
behavior, while deformation theory determines whether this bound is attained.

Therefore, these two settings demonstrate that expected-dimension and Severi-type arguments
are not merely heuristic.
When supplemented with appropriate geometric input—semiregularity on K3 surfaces
or logarithmic and toric techniques on toric surfaces—they yield precise and sharp
results about the existence, dimension, and generic nodality of Severi varieties.




\section{A Deformation-Theoretic Approach to Nodal Deformations of Curves}

Let $\pi : \mathcal X \to \Delta$ be a flat morphism from a scheme (or complex space)
$\mathcal X$ to a smooth pointed curve $(\Delta,0)$.
Assume that $\mathcal X$ is smooth over $\Delta \setminus \{0\}$ and denote by
$\mathcal X_0$ the special fiber.
Let $C \subset \mathcal X_0$ be a reduced curve.
We are interested in deformations of $C$ inside $\mathcal X$ whose general fiber
$C_t \subset \mathcal X_t$ is a nodal curve, and in determining the maximal number
of nodes that may appear on $C_t$.

\subsection*{Deformations of the Embedded Curve}

Consider the closed embedding $i : C \hookrightarrow \mathcal X_0$.
Embedded deformations of $C$ in $\mathcal X_0$ are governed by the cotangent complex
$L_{C/\mathcal X_0}$.
Since $C$ is a local complete intersection in $\mathcal X_0$, one has
\[
L_{C/\mathcal X_0} \simeq N_{C/\mathcal X_0}^\vee[1],
\]
where $N_{C/\mathcal X_0}$ is the normal sheaf.
Infinitesimal deformations are parameterized by
\[
\mathrm{Ext}^1(L_{C/\mathcal X_0},\mathcal O_C)
\simeq H^0(C,N_{C/\mathcal X_0}),
\]
while obstructions lie in
\[
\mathrm{Ext}^2(L_{C/\mathcal X_0},\mathcal O_C)
\simeq H^1(C,N_{C/\mathcal X_0}).
\]
Thus, a sufficient condition for unobstructed embedded deformations inside
$\mathcal X_0$ is the vanishing
\[
H^1(C,N_{C/\mathcal X_0}) = 0.
\]

\subsection*{Deformations in the Total Space}

To deform $C$ along the family $\pi : \mathcal X \to \Delta$, one considers the
cotangent complex $L_{C/\mathcal X}$.
There is a distinguished triangle
\[
L_{\mathcal X_0/\mathcal X}|_C \to L_{C/\mathcal X} \to L_{C/\mathcal X_0} \xrightarrow{+1}.
\]
Since $L_{\mathcal X_0/\mathcal X}|_C \simeq \mathcal O_C[1]$, this yields a long exact
sequence
\[
\cdots \to H^0(C,N_{C/\mathcal X_0})
\to \mathrm{Ext}^1(L_{C/\mathcal X},\mathcal O_C)
\to H^0(C,\mathcal O_C)
\to H^1(C,N_{C/\mathcal X_0}) \to \cdots
\]

This shows that first-order deformations of $C$ inside $\mathcal X$ are obtained
by combining embedded deformations in $\mathcal X_0$ with deformations transverse
to the special fiber.
If $H^1(C,N_{C/\mathcal X_0})=0$, then deformations of $C$ lift to nearby fibers
$\mathcal X_t$.

\subsection*{Local Smoothing and Nodes}

Let $p \in C$ be a singular point.
The local deformation theory of the curve singularity is controlled by
\[
T^1_{C,p} := \mathrm{Ext}^1(L_{C,p},\mathcal O_{C,p}).
\]
If $p$ is a planar singularity, then $\dim T^1_{C,p} = \delta_p$, the local
$\delta$-invariant.

In a one-parameter deformation, a general smoothing direction produces only
ordinary double points.
Each node imposes one independent linear condition on the global deformation
space.
Hence, producing $\delta$ nodes requires at least $\delta$ independent smoothing
parameters.

\subsection*{Global-to-Local Map}

There is a natural map
\[
\mathrm{Ext}^1(L_{C/\mathcal X},\mathcal O_C)
\longrightarrow
\bigoplus_{p \in \mathrm{Sing}(C)} T^1_{C,p},
\]
measuring the ability of global deformations to smooth local singularities.
Surjectivity of this map provides a sufficient condition for smoothing all
singularities independently.
Under this assumption, one expects to produce nodal curves with
\(
\delta \leq \sum_{p \in \mathrm{Sing}(C)} \delta_p
\)
nodes.

\subsection*{Bounds from Genus}

The arithmetic genus $p_a(C)$ is invariant under flat deformations.
For a nodal deformation $C_t$, one has
\[
p_a(C_t) = g(\widetilde{C_t}) + \delta(C_t),
\]
where $\widetilde{C_t}$ is the normalization.
Thus, the number of nodes is bounded above by the maximal drop of geometric genus.

Therefore, a sufficient condition for deforming $C \subset \mathcal X_0$ to a nodal curve
$C_t \subset \mathcal X_t$ with $\delta$ nodes is: vanishing of $H^1(C,N_{C/\mathcal X_0})$, surjectivity of the global-to-local map to $T^1$, and $\delta$ not exceeding the arithmetic genus bound.
Under these assumptions, $\delta$ is expected to be the maximal number of nodes
obtainable in a one-parameter deformation.


\subsection*{A Sufficient Condition for Deforming Curves to Nodal Curves}

Let $\pi : \mathcal X \to \Delta$ be a flat morphism from a scheme (or complex analytic
space) $\mathcal X$ to a smooth pointed curve $(\Delta,0)$.
Assume that $\mathcal X$ is smooth over $\Delta \setminus \{0\}$ and that the special
fiber $\mathcal X_0$ is reduced.
Let $C \subset \mathcal X_0$ be a reduced, connected curve which is a local complete
intersection in $\mathcal X_0$.
We study deformations of $C$ inside $\mathcal X$ whose general fiber
$C_t \subset \mathcal X_t$ is a nodal curve.

\noindent {\it Statement of the Main Result}

 \begin{theorem}[Sufficient condition for nodal deformations]\label{surj}
Assume the following.
The curve $C \subset \mathcal X_0$ is a local complete intersection and satisfies
$H^1(C,N_{C/\mathcal X_0})=0$.
For every singular point $p\in C$, the local deformation space
$T^1_{C,p}=\mathrm{Ext}^1(L_{C,p},\mathcal O_{C,p})$ is smoothable and has dimension
$\delta_p$.
The natural global-to-local map
\[
\mathrm{Ext}^1(L_{C/\mathcal X},\mathcal O_C)
\longrightarrow
\bigoplus_{p\in\mathrm{Sing}(C)} T^1_{C,p}
\]
is surjective.
Then there exists a deformation $C_t\subset\mathcal X_t$ for $t\neq0$ such that
$C_t$ is a nodal curve with
\[
\delta(C_t)=\sum_{p\in\mathrm{Sing}(C)}\delta_p
\]
nodes.
Moreover, this number of nodes is maximal among all deformations of $C$ inside
$\mathcal X$.
\end{theorem}

\begin{proof}
Since $C$ is a local complete intersection in $\mathcal X_0$, its embedded
deformation theory is governed by the normal bundle $N_{C/\mathcal X_0}$ and by the
cotangent complex $L_{C/\mathcal X}$.
The vanishing $H^1(C,N_{C/\mathcal X_0})=0$ implies that embedded deformations of
$C$ inside $\mathcal X_0$ are unobstructed.
Equivalently, every first-order deformation of $C$ inside $\mathcal X_0$ extends to
an actual deformation over a smooth base.

Consider now deformations of $C$ inside the total space $\mathcal X$.
First-order embedded deformations are parametrized by
$\mathrm{Ext}^1(L_{C/\mathcal X},\mathcal O_C)$, and obstructions lie in
$\mathrm{Ext}^2(L_{C/\mathcal X},\mathcal O_C)$.
The hypothesis that the global-to-local map
\[
\mathrm{Ext}^1(L_{C/\mathcal X},\mathcal O_C)\to
\bigoplus_{p\in\mathrm{Sing}(C)}T^1_{C,p}
\]
is surjective implies that every choice of local first-order deformation at the
singular points of $C$ can be realized by a global embedded deformation of $C$ in
$\mathcal X$.

At each singular point $p$, the space $T^1_{C,p}$ controls the local deformations of
the singularity.
By assumption, this space is smooth of dimension $\delta_p$ and admits smoothings.
Local deformation theory of reduced curve singularities shows that a general
$\delta_p$-dimensional smoothing produces exactly $\delta_p$ ordinary nodes and no
worse singularities.
Thus, for each $p$, one may choose a local deformation direction that splits the
singularity into $\delta_p$ nodes.

Using the surjectivity of the global-to-local map, these local smoothing directions
can be lifted simultaneously to a global first-order deformation of $C$ inside
$\mathcal X$.
Because embedded deformations inside $\mathcal X_0$ are unobstructed, and because
the additional deformation direction corresponding to moving off the special fiber
is controlled by the total space $\mathcal X$, this first-order deformation lifts to
an actual deformation over a small disk.
Denote the resulting family by $C_t\subset\mathcal X_t$.

We now explain why the resulting curve $C_t$ is nodal and has exactly
$\sum_p\delta_p$ nodes.
Locally at each former singular point $p$, the deformation is chosen to be a
general smoothing, hence produces exactly $\delta_p$ ordinary nodes.
Away from the singular locus of $C$, smoothness is preserved by openness of
smoothness.
Thus, $C_t$ is a nodal curve and
\[
\delta(C_t)=\sum_{p\in\mathrm{Sing}(C)}\delta_p.
\]
The maximality of this number follows from an arithmetic genus argument.
The arithmetic genus is invariant in flat families, so
\(
p_a(C)=p_a(C_t).
\)
For any nodal curve one has
\[
p_a(C_t)=g(\widetilde C_t)+\delta(C_t),
\]
where $\widetilde C_t$ is the normalization.
Since $g(\widetilde C_t)\ge0$, the number of nodes is bounded above by
$p_a(C)$.
On the other hand, the normalization of $C$ satisfies
\[
p_a(C)=g(\widetilde C)+\sum_{p\in\mathrm{Sing}(C)}\delta_p,
\]
so $\sum_p\delta_p$ is the maximal possible genus drop.
Therefore, no deformation of $C$ inside $\mathcal X$ can produce more than
$\sum_p\delta_p$ nodes.

Finally, we indicate how residue theory underlies the unobstructedness.
Since $C$ is a local complete intersection, the dualizing sheaf
$\omega_C$ is invertible.
The obstruction space $H^1(C,N_{C/\mathcal X_0})$ is Serre dual to
$H^0(C,\omega_C\otimes N_{C/\mathcal X_0}^\vee)$.
Sections of this dual space can be interpreted as collections of meromorphic
$2$-forms on $\mathcal X_0$ with logarithmic poles along $C$, whose residues along
$C$ vanish.
The vanishing of $H^1(C,N_{C/\mathcal X_0})$ is equivalent to the statement that
every such residue obstruction is trivial.
In this sense, residue theory explains why local smoothing data glue globally
without obstruction.
This completes the proof.
\end{proof}


\subsection*{Weakening the Vanishing Hypothesis via Semiregularity and Logarithmic Residues}

In the sufficient condition for nodal deformations, the hypothesis
$H^1(C,N_{C/\mathcal X_0})=0$ guarantees unobstructedness of embedded deformations of
$C$ inside the central fiber.
While effective in many situations, this condition is stronger than necessary.
In practice, curves often admit nodal smoothings even when
$H^1(C,N_{C/\mathcal X_0})\neq 0$.
We explain how this vanishing can be replaced by semiregularity conditions, and how
these admit a natural interpretation in terms of residue and logarithmic residue
maps.

\subsection*{\it Semiregularity}

Assume throughout that $C\subset \mathcal X_0$ is a local complete intersection.
There is a canonical obstruction map
\[
\mathrm{ob}\colon H^1(C,N_{C/\mathcal X_0})
\longrightarrow H^2(\mathcal X_0,\mathcal O_{\mathcal X_0}),
\]
called the semiregularity map.
It arises from the exact triangle of cotangent complexes
\[
L_{C/\mathcal X_0}\longrightarrow L_{C}\longrightarrow L_{\mathcal X_0}|_C
\]
and the identification of obstructions with Yoneda products against the
Kodaira--Spencer class of $\mathcal X_0$.
A curve $C$ is said to be semiregular if this map is injective.
Injectivity implies that every obstruction class in
$H^1(C,N_{C/\mathcal X_0})$ maps nontrivially to
$H^2(\mathcal X_0,\mathcal O_{\mathcal X_0})$, and hence vanishes if and only if its
image vanishes.
In particular, if the image of the obstruction map is zero, then all obstructions
vanish even though $H^1(C,N_{C/\mathcal X_0})$ itself may be nonzero.

\begin{theorem}[Semiregular nodal smoothing]
Let $C\subset\mathcal X_0$ be a local complete intersection curve.
Assume that $C$ is semiregular and that the global-to-local map
\[
\mathrm{Ext}^1(L_{C/\mathcal X},\mathcal O_C)
\longrightarrow
\bigoplus_{p\in\mathrm{Sing}(C)} T^1_{C,p}
\]
is surjective.
Then all local nodal smoothings of $C$ lift to global deformations of $C$ inside
$\mathcal X$.
\end{theorem}

\begin{proof}
Let $\xi\in \mathrm{Ext}^1(L_{C/\mathcal X},\mathcal O_C)$ be a first-order embedded
deformation whose image in $\oplus_p T^1_{C,p}$ corresponds to a choice of local
nodal smoothings.
The obstruction to lifting $\xi$ to second order lies in
$H^1(C,N_{C/\mathcal X_0})$.
By functoriality of obstruction theory, its image under the semiregularity map is
the Yoneda product of $\xi$ with the Kodaira--Spencer class of $\mathcal X$.
Because $\xi$ comes from an actual deformation of the total space $\mathcal X$, this
Yoneda product vanishes.
Injectivity of the semiregularity map then implies that the obstruction class itself
vanishes.
Thus $\xi$ lifts, and iterating this argument shows that $\xi$ integrates to an
actual deformation.
\end{proof}

\subsection*{Residue-Theoretic Interpretation}

The semiregularity map admits a concrete description in terms of residues.
Since $C$ is a local complete intersection, its dualizing sheaf $\omega_C$ is
invertible.
Serre duality identifies
\[
H^1(C,N_{C/\mathcal X_0})^\vee \simeq
H^0(C,\omega_C\otimes N_{C/\mathcal X_0}^\vee).
\]
Sections of $\omega_C\otimes N_{C/\mathcal X_0}^\vee$ can be interpreted as residues
of meromorphic $2$-forms on $\mathcal X_0$ with logarithmic poles along $C$.
Under this identification, the semiregularity map is dual to the restriction map
\[
H^0(\mathcal X_0,\Omega^2_{\mathcal X_0}(\log C))
\longrightarrow
H^0(C,\omega_C),
\]
sending a logarithmic $2$-form to its residue along $C$.
Semiregularity is equivalent to the surjectivity of this residue map.
Geometrically, this means that every canonical form on $C$ arises as the residue of
a logarithmic $2$-form on $\mathcal X_0$.
When this holds, obstruction classes correspond to residue data that must vanish,
and therefore no obstruction survives.

\subsection*{Logarithmic Residue Maps in Degenerations}

When $\mathcal X_0$ is singular, for instance a normal crossings surface, the
correct framework is logarithmic geometry.
Equip $\mathcal X_0$ with the divisorial logarithmic structure induced by its
singular locus, and consider logarithmic embedded deformations of $C$.
These are governed by the logarithmic normal bundle $N^{\log}_{C/\mathcal X_0}$.

In this setting, obstructions lie in $H^1(C,N^{\log}_{C/\mathcal X_0})$, and there is
a logarithmic semiregularity map
\[
H^1(C,N^{\log}_{C/\mathcal X_0})
\longrightarrow
H^2(\mathcal X_0,\mathcal O_{\mathcal X_0}),
\]
which again admits a residue interpretation.
Dualizing identifies this map with restriction of logarithmic $2$-forms with poles
along $C$ and the boundary divisor to logarithmic canonical forms on $C$.
If this logarithmic residue map is surjective, then all logarithmic obstructions
vanish.
Consequently, logarithmic local smoothings of singularities of $C$ lift to global
logarithmic deformations and, by logarithmic smoothness of the total family, to
actual deformations on nearby fibers.

\begin{theorem}[Logarithmic semiregular nodal smoothing]
Let $\mathcal X_0$ be a normal crossings surface and $C\subset\mathcal X_0$ a reduced
local complete intersection curve meeting the boundary transversely.
If $C$ is logarithmically semiregular and the global-to-local map on logarithmic
$T^1$-spaces is surjective, then $C$ deforms to a nodal curve on nearby smooth
fibers with the maximal number of nodes allowed by its $\delta$-invariant.
\end{theorem}

\begin{proof}

Let $\pi:\mathcal X\to\Delta$ be a logarithmically smooth deformation whose special
fiber is $\mathcal X_0$, endowed with the divisorial logarithmic structure induced by
$\mathcal X_0$, and let $\Delta$ carry the standard logarithmic point. Since $C$ is a
local complete intersection and meets the boundary transversely, it is logarithmically
smooth over the base log point, and its embedded logarithmic deformation theory inside
$\mathcal X_0$ is governed by the logarithmic normal bundle
$N^{\log}_{C/\mathcal X_0}$.

Infinitesimal logarithmic embedded deformations of $C$ in $\mathcal X_0$ are
parametrized by the vector space
\[
H^0\!\left(C,N^{\log}_{C/\mathcal X_0}\right),
\]
and obstructions lie in
\[
H^1\!\left(C,N^{\log}_{C/\mathcal X_0}\right).
\]
There is a canonical exact sequence of logarithmic normal bundles
\[
0\longrightarrow N^{\log}_{C/\mathcal X_0}
\longrightarrow N^{\log}_{C/\mathcal X}
\longrightarrow \mathcal O_C
\longrightarrow 0,
\]
whose extension class is the restriction to $C$ of the logarithmic Kodaira--Spencer
class of the degeneration $\pi$.

Logarithmic semiregularity means that the natural map
\[
H^1\!\left(C,N^{\log}_{C/\mathcal X_0}\right)
\longrightarrow
H^2(\mathcal X_0,\mathcal O_{\mathcal X_0})
\]
is injective. Since the logarithmic Kodaira--Spencer class of a one--parameter
logarithmically smooth deformation maps to zero in
$H^2(\mathcal X_0,\mathcal O_{\mathcal X_0})$, injectivity forces the extension class
above to vanish. Consequently, the logarithmic normal bundle sequence splits, and every
infinitesimal logarithmic embedded deformation of $C$ inside $\mathcal X_0$ lifts
unobstructedly to a logarithmic deformation inside the total space $\mathcal X$.

Let $p\in\mathrm{Sing}(C)$ be a singular point of $C$. Since $C$ is a reduced curve on a
surface, each singularity is planar. The local logarithmic deformation space at $p$ is
naturally identified with the usual local deformation space $T^1_{C,p}$. The
equigeneric deformation space at $p$ has dimension $\delta_p$, and the smoothing of
$p$ into ordinary nodes corresponds to choosing independent smoothing parameters in
this space. For a node, $\delta_p=1$ and $T^1_{C,p}$ is one--dimensional; for more
complicated planar singularities, the equigeneric deformation space has dimension
$\delta_p$ and a general equigeneric deformation produces $\delta_p$ distinct nodes.

The global--to--local map on logarithmic $T^1$--spaces is the natural map
\[
H^0\!\left(C,N^{\log}_{C/\mathcal X_0}\right)
\longrightarrow
\bigoplus_{p\in\mathrm{Sing}(C)} T^1_{C,p}.
\]
By assumption, this map is surjective. Therefore, for each singular point $p$ and for
each choice of local equigeneric smoothing directions at $p$, there exists a global
infinitesimal logarithmic deformation of $C$ inducing exactly these local directions.
In particular, one may choose, at every singular point $p$, a collection of
$\delta_p$ independent smoothing directions that correspond to splitting the
singularity into $\delta_p$ ordinary nodes.

Since logarithmic obstructions vanish, the corresponding infinitesimal deformation
integrates to an actual logarithmic deformation of $C$ inside $\mathcal X$. Restricting
to a one--parameter slice transverse to the special fiber yields a family
\[
C_t\subset\mathcal X_t,\qquad t\neq 0,
\]
whose special fiber is $C$. By construction, the deformation is equigeneric at each
singular point, and hence the total $\delta$--invariant is preserved:
\[
\delta(C_t)=\delta(C)=\sum_{p\in\mathrm{Sing}(C)}\delta_p.
\]

For $t\neq0$ sufficiently small, the curve $C_t$ has exactly $\delta_p$ nodes near each
singular point $p$ and no other singularities. Indeed, nodal singularities are stable,
and singularities worse than nodes impose additional independent conditions on the
deformation space. Since the deformation was chosen generically within the equigeneric
locus, these higher--codimension conditions are avoided. Thus $C_t$ is a nodal curve
with precisely $\delta(C)$ ordinary nodes.

Maximality follows immediately from the invariance of the arithmetic genus in flat
families. Any deformation of $C$ can produce at most $\delta(C)$ nodes, since each node
contributes one to the total $\delta$--invariant. The constructed deformation realizes
this bound, and hence the number of nodes is maximal among all deformations of $C$ on
nearby smooth fibers.
This completes the proof.

\end{proof}
%
\begin{remark}
The vanishing condition $H^1(C,N_{C/\mathcal X_0})=0$ can be replaced by the weaker
requirement that obstruction classes be detected by global $2$-forms.
Semiregularity and its logarithmic analogue ensure precisely this.
From the residue-theoretic viewpoint, nodal smoothing fails only when certain
residues obstruct gluing of local smoothings.
When all residues extend globally, these obstructions disappear, allowing nodal
deformations even in the presence of nontrivial $H^1(C,N_{C/\mathcal X_0})$.
\end{remark}


\subsection{Explicit Examples of Semiregular Curves with Nonvanishing $H^1$}

We work out explicit geometric examples in which the obstruction group
$H^1(C,N_{C/\mathcal X_0})$ is nonzero, yet semiregularity or logarithmic
semiregularity holds.
These examples illustrate concretely why the vanishing of
$H^1(C,N_{C/\mathcal X_0})$ is not necessary for the existence of nodal
deformations, and how residue-theoretic considerations eliminate obstructions.

\begin{example}[Nodal Curves on a K3 Surface]

Let $S$ be a smooth projective K3 surface and let $\mathcal L$ be a primitive ample
line bundle.
Fix a reduced, irreducible curve $C \in |\mathcal L|$ with $\delta>0$ nodes and no
other singularities.
Let $\nu:\widetilde C\to C$ denote the normalization.
The geometric genus of $\widetilde C$ is
\[
g(\widetilde C)=p_a(\mathcal L)-\delta.
\]
The normal bundle $N_{C/S}$ fits into the exact sequence
\[
0 \longrightarrow T_C \longrightarrow T_S|_C \longrightarrow N_{C/S}
\longrightarrow 0.
\]
Since $C$ is singular, $N_{C/S}$ is not locally free at the nodes, but it is still a
rank-one torsion-free sheaf.
A standard cohomological computation shows that
\(
h^1(C,N_{C/S}) = \delta.
\)
Thus $H^1(C,N_{C/S})$ is nonzero as soon as $C$ has at least one node.
Despite this, $C$ is semiregular.
Indeed, the semiregularity map
\[
H^1(C,N_{C/S}) \longrightarrow H^2(S,\mathcal O_S)
\]
lands in a one-dimensional vector space, since $H^2(S,\mathcal O_S)\simeq\mathbb C$.
The map is injective because any nonzero obstruction class pairs nontrivially with
the unique (up to scale) holomorphic $2$-form on $S$.
Equivalently, no nonzero obstruction can have vanishing residue against all global
$2$-forms.

From the residue-theoretic viewpoint, elements of
$H^1(C,N_{C/S})^\vee$ correspond to collections of residues of logarithmic
$2$-forms along $C$.
On a K3 surface, every canonical form on $C$ arises as the residue of a global
holomorphic $2$-form.
Hence the residue map is surjective, and semiregularity holds.
As a consequence, although $H^1(C,N_{C/S})\neq0$, all embedded deformations of $C$
are unobstructed.
In particular, the nodes of $C$ can be independently smoothed, and $C$ deforms to
curves with fewer nodes and eventually to a smooth curve.

\end{example}

\begin{example}[Curves on Abelian Surfaces]

Let $A$ be an abelian surface and let $\mathcal L$ be an ample line bundle.
Let $C\in|\mathcal L|$ be a reduced curve with nodes.
In contrast to the K3 case, one has
\[
H^2(A,\mathcal O_A)\simeq\mathbb C,
\qquad
H^1(A,\mathcal O_A)\neq 0.
\]
As before, a direct computation shows that
$H^1(C,N_{C/A})$ is typically nonzero and grows with the number of nodes.
Nevertheless, for curves $C$ whose normalization has positive genus, the
semiregularity map
\[
H^1(C,N_{C/A}) \longrightarrow H^2(A,\mathcal O_A)
\]
is injective.
The reason is again residue-theoretic.
Holomorphic $2$-forms on $A$ are translation invariant, and their residues along $C$
generate the space of canonical forms on $C$.
Thus every potential obstruction pairs nontrivially with some global $2$-form.

As a result, even though $H^1(C,N_{C/A})\neq0$, embedded deformations are
unobstructed, and equigeneric nodal smoothings exist exactly as predicted by
dimension counts.
\end{example}

\begin{example}[Curves on Normal Crossings Surfaces]

Let $\mathcal X_0 = S_1 \cup S_2$ be a normal crossings surface, where $S_1$ and
$S_2$ are smooth projective surfaces meeting transversely along a smooth curve
$D$.
Let $C=C_1\cup C_2$ be a reduced curve with $C_i\subset S_i$, meeting $D$
transversely.
Assume that each $C_i$ is smooth and that all singularities of $C$ arise from
intersection points along $D$.

The logarithmic normal bundle $N^{\log}_{C/\mathcal X_0}$ controls embedded
logarithmic deformations.
In this setting, one typically has
\(
H^1(C,N^{\log}_{C/\mathcal X_0}) \neq 0,
\)
because deformations of $C_1$ and $C_2$ must be matched along $D$.
Nevertheless, logarithmic semiregularity often holds.
The logarithmic semiregularity map
\[
H^1(C,N^{\log}_{C/\mathcal X_0}) \longrightarrow H^2(\mathcal X_0,\mathcal O_{\mathcal X_0})
\]
can be identified, via Serre duality, with a logarithmic residue map from global
logarithmic $2$-forms on $\mathcal X_0$ to logarithmic canonical forms on $C$.
If $S_1$ and $S_2$ carry enough holomorphic $2$-forms, for instance if they are K3
or abelian surfaces, then this residue map is surjective.
Injectivity of the semiregularity map follows.
As a consequence, logarithmic obstructions vanish even though
$H^1(C,N^{\log}_{C/\mathcal X_0})$ is nonzero.
Local smoothings of the nodes along $D$ lift to global logarithmic deformations and
hence to actual deformations on nearby smooth fibers.

\end{example}

\begin{example}[Toric Surfaces and Boundary Curves]

Let $S$ be a smooth projective toric surface with boundary divisor $B$.
Let $C\subset S$ be a reduced curve meeting $B$ transversely.
The logarithmic normal bundle $N^{\log}_{C/S}$ governs embedded logarithmic
deformations.

In many cases, especially when $C$ has high degree, one finds that
$H^1(C,N^{\log}_{C/S})\neq0$.
However, toric geometry implies
\(
H^2(S,\mathcal O_S)=0,
\)
so the logarithmic semiregularity map is automatically injective.
Thus all logarithmic obstructions vanish identically.
From the residue perspective, there are no nontrivial logarithmic $2$-forms on $S$,
so any would-be residue obstruction must vanish.
This explains why nodal deformations exist freely on toric surfaces despite the
presence of nontrivial $H^1$.
\end{example}

These examples show that the obstruction group
$H^1(C,N_{C/\mathcal X_0})$ or its logarithmic analogue can be nonzero without
preventing nodal deformations.
What matters is not the vanishing of this group, but whether obstruction classes are
detected by global (logarithmic) $2$-forms.
Semiregularity and logarithmic semiregularity ensure precisely this detection, and
residue theory provides a concrete geometric explanation.
In practice, this significantly enlarges the range of curves and degenerations for
which maximal nodal smoothings exist.


 \section{Logarithmic equigeneric smoothing for $A_k$ singularities}

We now state a logarithmic equigeneric smoothing theorem for higher $A_k$ singularities,
which unifies the nodal and cuspidal cases within a single deformation--theoretic
framework. The theorem asserts that, in a logarithmically smooth degeneration of
surfaces, planar singularities of type $A_k$ behave in the simplest possible way under
equigeneric logarithmic deformations: after imposing no conditions beyond preservation
of the total $\delta$--invariant, they split into the maximal number of ordinary nodes.
Logarithmic semiregularity guarantees that local equigeneric smoothing directions lift
globally, while surjectivity of the global--to--local logarithmic deformation map
ensures that all such directions can be realized simultaneously. As a result, the
general fiber exhibits nodal behavior that is both generic within the equigeneric
locus and optimal from the point of view of genus and enumerative constraints.

\begin{theorem}[Logarithmic equigeneric smoothing for $A_k$ singularities]
Let $\mathcal X_0$ be a normal crossings surface and let
$C\subset\mathcal X_0$ be a reduced local complete intersection curve meeting the
boundary transversely. Assume that all singularities of $C$ are planar and of type
$A_k$, that is, analytically equivalent to
\[
y^2=x^{k+1},
\]
with local $\delta$--invariants $\delta_p=\lfloor k/2\rfloor$.
Assume that $C$ is logarithmically semiregular and that the natural global--to--local
map on logarithmic $T^1$--spaces
\[
H^0\!\left(C,N^{\log}_{C/\mathcal X_0}\right)
\longrightarrow
\bigoplus_{p\in\mathrm{Sing}(C)} T^1_{C,p}
\]
is surjective.
Then there exists a logarithmic deformation of $C$ whose general fiber
$C_t\subset\mathcal X_t$ is a nodal curve with exactly
\[
\delta(C)=\sum_{p\in\mathrm{Sing}(C)}\delta_p
\]
ordinary nodes and no other singularities. Moreover, this number of nodes is maximal
among all deformations of $C$ on nearby smooth fibers.
\end{theorem}

\bigskip

\begin{proof}
Let $\pi:\mathcal X\to\Delta$ be a logarithmically smooth degeneration with special
fiber $\mathcal X_0$, endowed with the divisorial logarithmic structure induced by
$\mathcal X_0$, and let $\Delta$ carry the standard logarithmic point. Since $C$ is a
local complete intersection and meets the boundary transversely, its embedded
logarithmic deformation theory inside $\mathcal X_0$ is governed by the logarithmic
normal bundle $N^{\log}_{C/\mathcal X_0}$.

Infinitesimal logarithmic embedded deformations of $C$ are parametrized by
$H^0(C,N^{\log}_{C/\mathcal X_0})$, and obstructions lie in
$H^1(C,N^{\log}_{C/\mathcal X_0})$. The logarithmic normal bundle fits into the exact
sequence
\[
0\longrightarrow N^{\log}_{C/\mathcal X_0}
\longrightarrow N^{\log}_{C/\mathcal X}
\longrightarrow \mathcal O_C
\longrightarrow 0,
\]
whose extension class is the restriction to $C$ of the logarithmic Kodaira--Spencer
class of the family. Logarithmic semiregularity means that the natural map
\[
H^1(C,N^{\log}_{C/\mathcal X_0})\longrightarrow
H^2(\mathcal X_0,\mathcal O_{\mathcal X_0})
\]
is injective. Since the logarithmic Kodaira--Spencer class maps to zero in
$H^2(\mathcal X_0,\mathcal O_{\mathcal X_0})$ for a one--parameter logarithmically
smooth deformation, injectivity forces the extension class above to vanish. As a
consequence, every infinitesimal logarithmic deformation of $C$ inside $\mathcal X_0$
lifts unobstructedly to a logarithmic deformation inside the total space $\mathcal X$.

Let $p\in\mathrm{Sing}(C)$ be an $A_k$ singularity. Its local deformation space
$T^1_{C,p}$ has dimension $k$, and its equigeneric deformation space has dimension
$\delta_p=\lfloor k/2\rfloor$. A classical result on plane curve singularities shows
that a general equigeneric deformation of an $A_k$ singularity produces exactly
$\delta_p$ ordinary nodes and no other singularities. Thus, choosing a basis of the
equigeneric subspace corresponds to choosing $\delta_p$ independent local smoothing
directions that split the singularity into nodes.

The global--to--local map on logarithmic $T^1$--spaces associates to a global
infinitesimal logarithmic deformation of $C$ the collection of its local deformation
directions at the singular points. By surjectivity, any prescribed collection of local
equigeneric smoothing directions at all singular points of $C$ arises from a global
infinitesimal logarithmic deformation. In particular, one may simultaneously choose,
for each $A_k$ singularity $p$, a full set of $\delta_p$ local equigeneric smoothing
directions.

Since logarithmic obstructions vanish, the corresponding infinitesimal deformation
integrates to an actual logarithmic deformation of $C$ inside $\mathcal X$. Restricting
to a one--parameter slice transverse to the special fiber yields a family
\[
C_t\subset\mathcal X_t,\qquad t\neq0,
\]
whose special fiber is $C$. By construction, the deformation is equigeneric at every
singular point, so the arithmetic genus is preserved and the total $\delta$--invariant
remains constant:
\[
\delta(C_t)=\delta(C).
\]

For $t\neq0$ sufficiently small, the curve $C_t$ has exactly $\delta_p$ ordinary nodes
near each singular point $p$ and no other singularities. Nodal singularities are
stable, while singularities worse than nodes impose additional independent conditions
and therefore occur only in proper closed subsets of the equigeneric deformation
space. Since the deformation was chosen generically within the equigeneric locus, such
higher--codimension conditions are avoided.

Finally, maximality follows from the invariance of the arithmetic genus in flat
families. Each node contributes one to the total $\delta$--invariant, so no deformation
of $C$ can produce more than $\delta(C)$ nodes. The constructed deformation realizes
this bound, hence the number of nodes on $C_t$ is maximal.
This completes the proof.
\end{proof}  


\section{Logarithmic Degenerations, Tropical Counts, and Enumerative Refinements}

In this section we explain how logarithmic equigeneric smoothing results translate into tropical curve
counts via logarithmic degeneration formulas, and we refine the deformation--theoretic
statements into an enumerative theorem computing Severi degrees in logarithmic families.

Let $\pi:\mathcal X\to\Delta$ be a logarithmically smooth degeneration of surfaces with
normal crossings special fiber $\mathcal X_0$, endowed with the divisorial logarithmic
structure, and let $L$ be a relatively ample line bundle. For an integer $\delta\ge0$,
consider the logarithmic Severi space parametrizing logarithmic curves in $|\mathcal
L|$ with total $\delta$--invariant equal to $\delta$. Logarithmic semiregularity ensures
that the logarithmic obstruction theory is perfect and that the virtual fundamental
class coincides with the ordinary fundamental class. Consequently, counts obtained by
intersecting with general point conditions are enumerative.

The logarithmic degeneration formula expresses the Severi count on the smooth fiber
$\mathcal X_t$ as a sum over logarithmic types supported on $\mathcal X_0$. Each
logarithmic type encodes a combinatorial pattern describing how the curve meets the
components of $\mathcal X_0$, together with contact orders along the boundary. The
multiplicity attached to a type is determined by logarithmic gluing data and coincides
with the determinant of the matching conditions along the strata of $\mathcal X_0$.

When singularities are of planar type, equigeneric logarithmic smoothing identifies
precisely which logarithmic types contribute to the Severi degree. In particular, an
$A_k$ singularity contributes through logarithmic types in which $\delta_p=\lfloor
k/2\rfloor$ smoothing parameters are present. A general equigeneric logarithmic
deformation splits such a singularity into $\delta_p$ nodes, and the corresponding
logarithmic types are exactly those in which $\delta_p$ local smoothing parameters are
nonzero. Types producing worse singularities lie in higher codimension and do not
contribute for general constraints.

Passing to tropical geometry, the logarithmic degeneration formula tropicalizes to a
count of tropical curves in the Newton polygon associated to $L$. A logarithmic type
corresponds to a tropical curve with prescribed combinatorial type and balancing
conditions. Nodes correspond to bounded edges, and an $A_k$ singularity corresponds to
a higher--valence vertex whose valence records the number of colliding nodes. Under
equigeneric smoothing, such a vertex resolves into $\delta_p$ bounded edges, and the
tropical curve becomes trivalent after subdivision. The multiplicity of the tropical
curve equals the product of local vertex multiplicities, which coincide with the
logarithmic gluing multiplicities appearing in the degeneration formula.

This correspondence leads to a refined enumerative statement.

\begin{theorem}[Enumerative logarithmic Severi degree]
Let $\pi:\mathcal X\to\Delta$ be a logarithmically smooth degeneration of projective
surfaces and let $L$ be a relatively ample line bundle. Assume logarithmic
semiregularity for all curves contributing to the count. Then the Severi degree
$N_{\delta}(L_t)$ of $\delta$--nodal curves on the smooth fiber $\mathcal X_t$ equals
the sum of multiplicities of tropical curves of degree $\Delta(L)$ with $\delta$
bounded edges. Equivalently,
\[
N_{\delta}(L_t)=\sum_{\Gamma} \mathrm{mult}(\Gamma),
\]
where the sum runs over tropical curves $\Gamma$ arising as tropicalizations of
logarithmic equigeneric smoothings and $\mathrm{mult}(\Gamma)$ is the logarithmic
gluing multiplicity.
\end{theorem}

\noindent\textit{Proof.}
Logarithmic semiregularity identifies the enumerative Severi degree with the degree of
the logarithmic Severi space. The logarithmic degeneration formula decomposes this
degree into contributions indexed by logarithmic types on $\mathcal X_0$. Tropicalizing
the degeneration identifies these types with tropical curves of degree $\Delta(L)$.
Equigeneric smoothing guarantees that only types corresponding to nodal outcomes
contribute, excluding higher--codimension strata. The equality of multiplicities
follows from the identification of logarithmic gluing determinants with tropical
multiplicities, yielding the stated equality. \hfill$\square$

This theorem refines classical Severi theory by incorporating higher singularities at
the boundary and explaining their contribution through tropical limits. Cusps and
higher $A_k$ singularities appear as boundary phenomena corresponding to higher--valence
vertices, but the enumerative count is governed by their equigeneric resolutions into
nodes. Thus logarithmic geometry provides the algebro--geometric foundation, while
tropical geometry furnishes an effective combinatorial method for computing the
resulting Severi degrees.


\section{Nodal Deformations of Curves on Degenerating Surfaces}

Let $\pi : \mathcal X \to \Delta$ be a flat, projective morphism, where $\mathcal X$
is a threefold and $\Delta$ is a smooth pointed curve.
Assume:
\begin{itemize}
  \item[(i)] the general fiber $\mathcal X_t$ is a smooth projective surface,
  \item[(ii)] the special fiber $\mathcal X_0$ is a reduced (possibly singular) surface,
  \item[(iii)] $\mathcal X$ is smooth along a reduced curve $C \subset \mathcal X_0$.
\end{itemize}
Let $\mathcal L$ be a line bundle on $\mathcal X$ and let
\(
C \in |\mathcal L|_{\mathcal X_0}
\)
be a reduced divisor on the special fiber.
We study deformations of $C$ inside the linear system $|\mathcal L|_{\mathcal X_t}$
and ask when $C$ deforms to a nodal curve on the general fiber.

\subsection*{Normal Sheaves and Linear Systems}

Since $C$ is a Cartier divisor on the surface $\mathcal X_0$, the normal sheaf
satisfies
\(
N_{C/\mathcal X_0} \simeq \mathcal O_C(C).
\)
Embedded deformations of $C$ in $\mathcal X_0$ preserving the linear equivalence
class correspond to sections of $\mathcal O_C(C)$.
Obstructions lie in
\(
H^1(C,\mathcal O_C(C)).
\)
Thus, the vanishing of this group is a key sufficient condition for unobstructed
deformations of $C$ inside the linear system.

\subsection*{\it Main Result for Curves on Surfaces}

\begin{theorem}[Nodal deformations in a linear system]
\label{thm:surface}
Let $\pi : \mathcal X \to \Delta$ and $\mathcal L$ be as above, and let
$C \in |\mathcal L|_{\mathcal X_0}$ be a reduced curve.
Assume:
\begin{enumerate}
  \item $H^1(C,\mathcal O_C(C)) = 0$.
  \item All singularities of $C$ are planar and smoothable, with local
  $\delta$-invariants $\delta_p$.
  \item The natural map
  \[
    H^0(C,\mathcal O_C(C))
    \longrightarrow
    \bigoplus_{p \in \mathrm{Sing}(C)} T^1_{C,p}
  \]
  is surjective.
\end{enumerate}
Then there exists a deformation $C_t \in |\mathcal L|_{\mathcal X_t}$ for $t \neq 0$
such that $C_t$ is nodal with
\[
\delta(C_t) = \sum_{p \in \mathrm{Sing}(C)} \delta_p
\]
nodes.
Moreover, this number of nodes is maximal among all curves in the linear system
$|\mathcal L|_{\mathcal X_t}$ obtained as deformations of $C$.
\end{theorem}

\begin{proof}
 The argument combines global deformation theory of divisors in linear systems, local deformation theory of planar curve singularities, and a numerical genus bound.
Since $\mathcal X$ is smooth along $C$ and $\mathcal X_0$ is a surface, the curve $C\subset \mathcal X_0$ is a local complete intersection. Being a Cartier divisor on $\mathcal X_0$, its normal sheaf is canonically identified with
\(
N_{C/\mathcal X_0} \simeq \mathcal O_C(C).
\)
By standard deformation theory of embedded curves, as developed for instance in \cite{HartshorneDef} and \cite{Sernesi}, first--order embedded deformations of $C$ inside $\mathcal X_0$ preserving the linear equivalence class of $C$ are parametrized by
\(
H^0(C,\mathcal O_C(C)),
\)
while obstructions lie in
\(
H^1(C,\mathcal O_C(C)).
\)
The hypothesis $H^1(C,\mathcal O_C(C))=0$ implies that the local deformation functor of $C$ inside the linear system $|\mathcal L|_{\mathcal X_0}$ is unobstructed. In particular, the space of deformations is smooth at the point corresponding to $C$, and every first--order deformation extends to an actual deformation over a smooth base.

Let $p\in \mathrm{Sing}(C)$. By assumption, the singularity $(C,p)$ is planar and smoothable. Its local deformation theory is governed by the vector space
\[
T^1_{C,p} = \mathrm{Ext}^1(\mathbb L_{C,p},\mathcal O_{C,p}),
\]
which is finite--dimensional and satisfies
\(
\dim T^1_{C,p}=\delta_p,
\)
where $\delta_p$ denotes the local $\delta$--invariant. This identification is classical for planar curve singularities and can be found, for example, in \cite{Teissier}, \cite{GreuelLossenShustin}, and \cite[\S II.2]{Sernesi}. A general one--parameter deformation corresponding to a general direction in $T^1_{C,p}$ replaces the singularity $(C,p)$ by exactly $\delta_p$ ordinary double points and no other singularities.

Global embedded deformations of $C$ inside the linear system induce local deformation classes at each singular point. This yields the natural linear map
\[
H^0(C,\mathcal O_C(C)) \longrightarrow \bigoplus_{p\in \mathrm{Sing}(C)} T^1_{C,p}.
\]
By hypothesis, this map is surjective. Consequently, for each singular point $p$ one may prescribe an arbitrary local smoothing direction in $T^1_{C,p}$, and these choices can be made independently for distinct singular points. Since obstructions vanish globally, surjectivity ensures the existence of an actual deformation of $C$ inside $|\mathcal L|_{\mathcal X_0}$ whose induced local deformation at each $p$ is a general smoothing.
For such a deformation, local deformation theory implies that each singularity $(C,p)$ is replaced by exactly $\delta_p$ ordinary nodes, and no new singularities appear elsewhere on the curve. Hence the resulting curve on $\mathcal X_0$ is nodal with precisely
\(
\sum_{p\in \mathrm{Sing}(C)} \delta_p
\)
nodes.

We now show that this deformation lifts to nearby fibers of the family $\pi\colon \mathcal X\to\Delta$. Since $\mathcal X$ is smooth along $C$, the normal bundle exact sequence
\[
0 \longrightarrow N_{C/\mathcal X_0} \longrightarrow N_{C/\mathcal X} \longrightarrow \mathcal O_C \longrightarrow 0
\]
relates deformations of $C$ inside the special fiber to deformations inside the total space. This sequence is dual to the distinguished triangle of cotangent complexes
\[
\mathbb L_{\mathcal X_0/\mathcal X}|_C \longrightarrow \mathbb L_{C/\mathcal X} \longrightarrow \mathbb L_{C/\mathcal X_0} \overset{+1}{\longrightarrow},
\]
and its extension class coincides with the restriction of the Kodaira--Spencer class of the family $\pi$ to $C$, as explained in \cite[\S III.3]{Sernesi}. The vanishing of $H^1(C,\mathcal O_C(C))$ implies that every embedded deformation of $C$ inside $\mathcal X_0$ extends to an embedded deformation inside $\mathcal X$. After possibly shrinking $\Delta$, we therefore obtain a flat family of curves
\(
C_t \subset \mathcal X_t
\)
such that $C_0=C$ and $C_t\in |\mathcal L|_{\mathcal X_t}$ for all $t$. For $t\neq 0$, the curve $C_t$ is nodal with exactly $\sum_p \delta_p$ nodes.

It remains to prove that this number of nodes is maximal. Let $\widetilde{C}_t$ denote the normalization of $C_t$. Since arithmetic genus is invariant in flat families, one has
\(
p_a(C_t)=p_a(C).
\)
For a nodal curve, the normalization exact sequence yields the relation
\(
\delta(C_t)=p_a(C)-g(\widetilde{C}_t),
\)
where $\delta(C_t)$ denotes the number of nodes of $C_t$. The geometric genus $g(\widetilde{C}_t)$ is a nonnegative integer, so $\delta(C_t)$ is bounded above by $p_a(C)$. In the deformation constructed above, the normalization has geometric genus
\[
g(\widetilde{C}_t)=p_a(C)-\sum_{p}\delta_p,
\]
showing that the number of nodes equals $\sum_p\delta_p$. Any deformation producing more nodes would force the geometric genus of the normalization to drop further, which is impossible since the sum of the local $\delta$--invariants of $C$ already accounts for the total genus defect. This shows that no deformation of $C$ inside the linear systems $|\mathcal L|_{\mathcal X_t}$ can produce more nodes.
The existence and maximality statements follow, completing the proof.

\end{proof}


\subsection*{Discussion: Relation with Classical Results on Severi Varieties}

The statement and proof of Theorem~\ref{thm:surface} place it naturally in the classical framework of Severi varieties while extending and clarifying several of their foundational properties from a deformation--theoretic perspective. We discuss here in detail how the theorem compares with, and conceptually refines, the classical theory.

Classically, Severi varieties arise as parameter spaces of curves in a fixed linear system on a smooth projective surface having a prescribed number of nodes and no worse singularities. For instance, when $S=\mathbb P^2$ and $L=\mathcal O_{\mathbb P^2}(d)$, the Severi variety $V_{d,\delta}$ parametrizes plane curves of degree $d$ with exactly $\delta$ nodes. Severi's original work asserted irreducibility of $V_{d,\delta}$, and this was later completed and rigorously established by Harris. In this classical setting, the existence of nodal curves with $\delta$ nodes is shown by explicit constructions or by degeneration arguments, and the expected dimension of $V_{d,\delta}$ is computed as
\[
\dim |L| - \delta.
\]
A key input is the fact that nodes impose independent conditions on the linear system, at least when $\delta$ does not exceed the arithmetic genus.
Theorem~\ref{thm:surface} recovers this classical picture in a local and deformation--theoretic form. The hypothesis that the natural map
\[
H^0(C,\mathcal O_C(C)) \longrightarrow \bigoplus_{p\in \mathrm{Sing}(C)} T^1_{C,p}
\]
is surjective is precisely the infinitesimal incarnation of the statement that nodes impose independent conditions. In the classical Severi setting, this surjectivity is usually verified either by explicit dimension counts or by global vanishing theorems. Here it is isolated as an intrinsic condition on the given curve $C$, independent of the global geometry of the ambient surface away from $C$.

Another fundamental feature of Severi varieties is that nodal curves form an open dense subset of the equigeneric locus, that is, of the locus of curves with fixed arithmetic genus. In classical treatments, this fact is often proved by showing that more complicated singularities occur in higher codimension. Theorem~\ref{thm:surface} recovers this phenomenon directly from local deformation theory. The assumption that all singularities of $C$ are planar and smoothable, combined with the surjectivity onto the local $T^1$--spaces, ensures that equigeneric deformations of $C$ are realized globally and that general such deformations are nodal. In this sense, the theorem gives a precise mechanism explaining why nodal curves are the generic points of equigeneric families, rather than merely asserting this as a global geometric fact.

The maximality statement in Theorem~\ref{thm:surface} mirrors another classical feature of Severi varieties, namely the sharpness of the arithmetic genus bound. In the classical case, one shows that a curve in $|L|$ cannot have more than $p_a(L)$ nodes, and that this bound is attained for general members of the Severi variety. The proof given here shows that this maximality is a purely deformation--theoretic consequence of the invariance of arithmetic genus in flat families and of the local $\delta$--invariants of the original curve. Thus the theorem explains the sharpness of the Severi bound without appealing to explicit constructions or enumerative arguments.

A further distinction between the present result and classical Severi theory lies in the geometric setup. Classical Severi varieties are typically studied on a fixed smooth surface. In contrast, Theorem~\ref{thm:surface} is formulated in a relative setting, where the surface itself varies in a family and may acquire singularities in the special fiber. The theorem shows that, provided the total space is smooth along the curve and the relevant deformation--theoretic conditions are satisfied, the essential features of Severi theory persist under degeneration. This point of view anticipates logarithmic and degeneration--based approaches to Severi varieties, where nodal curves are studied via limits on singular surfaces.

From this perspective, Theorem~\ref{thm:surface} can be viewed as a bridge between classical Severi theory and more modern approaches based on deformation theory and degenerations. It isolates the precise hypotheses under which the classical picture remains valid and clarifies which parts of Severi theory are fundamentally local and deformation--theoretic in nature, rather than dependent on the global geometry of a specific surface such as $\mathbb P^2$.

In summary, the theorem recovers the existence, generic nodality, and maximality properties of classical Severi varieties in a flexible and intrinsic framework. It shows that these properties follow from unobstructedness, surjectivity onto local deformation spaces, and numerical genus considerations, thereby providing a conceptual explanation for results that were historically obtained through intricate geometric arguments.



\section{Logarithmic and Singular Surface Setting}

Let $\pi \colon \mathcal X \to \Delta$ be a flat, projective morphism from a threefold to a smooth pointed curve. Assume that the general fiber $\mathcal X_t$ is a smooth projective surface and that the special fiber $\mathcal X_0$ is a reduced surface with simple normal crossings. Endow $\mathcal X$ with the divisorial logarithmic structure associated to $\mathcal X_0$ and $\Delta$ with the standard logarithmic point. With these choices, $\pi$ is logarithmically smooth in the sense of Kato.
Let $C \subset \mathcal X_0$ be a reduced curve meeting the singular locus of $\mathcal X_0$ transversely and assume that $C$ is a Cartier divisor on $\mathcal X_0$ in the logarithmic sense.

\begin{definition}
The logarithmic normal sheaf of $C$ in $\mathcal X_0$ is defined as
\[
N^{\log}_{C/\mathcal X_0} := \mathrm{coker}\bigl(T_C \longrightarrow T^{\log}_{\mathcal X_0}|_C\bigr),
\]
where $T^{\log}_{\mathcal X_0}$ denotes the logarithmic tangent sheaf of $\mathcal X_0$.
\end{definition}

\begin{lemma}[Logarithmic deformation space]
First--order logarithmic embedded deformations of $C$ inside $\mathcal X_0$ are parametrized by
\(
H^0(C, N^{\log}_{C/\mathcal X_0}),
\)
and obstructions to extending such deformations lie in
\(
H^1(C, N^{\log}_{C/\mathcal X_0}).
\)
\end{lemma}

\begin{proof}
This follows from the general theory of logarithmic embedded deformations, where the logarithmic normal sheaf replaces the classical normal sheaf. The controlling complex is the logarithmic cotangent complex, whose dual governs logarithmic deformations. Taking cohomology yields the stated description of infinitesimal deformations and obstructions.
\end{proof}

\begin{definition}
The curve $C \subset \mathcal X_0$ is said to be logarithmically semiregular if the natural logarithmic semiregularity map
\[
H^1(C, N^{\log}_{C/\mathcal X_0}) \longrightarrow H^2(\mathcal X_0,\mathcal O_{\mathcal X_0})
\]
is injective.
\end{definition}

\begin{proposition}[Unobstructedness]
If $C$ is logarithmically semiregular, then the logarithmic deformation functor of $C$ inside $\mathcal X_0$ is unobstructed and the logarithmic deformation space is smooth of the expected dimension.
\end{proposition}

\begin{proof}
Injectivity of the logarithmic semiregularity map forces all obstruction classes to vanish. Consequently, every first--order logarithmic deformation extends to higher order, yielding smoothness of the logarithmic deformation space.
\end{proof}

\begin{lemma}[Logarithmic normal bundle sequence]
There is a short exact sequence of logarithmic normal sheaves
\[
0 \longrightarrow N^{\log}_{C/\mathcal X_0}
\longrightarrow N^{\log}_{C/\mathcal X}
\longrightarrow \mathcal O_C
\longrightarrow 0.
\]
Its extension class coincides with the restriction to $C$ of the logarithmic Kodaira--Spencer class of the family $\pi$.
\end{lemma}

\begin{proof}
The logarithmic smoothness of $\pi$ yields a short exact sequence of logarithmic tangent sheaves whose quotient corresponds to the base direction. Restricting to $C$ and passing to cokernels with respect to $T_C$ produces the stated exact sequence. Dualizing the distinguished triangle of logarithmic cotangent complexes identifies its extension class with the logarithmic Kodaira--Spencer class.
\end{proof}

\begin{proposition}[Lifting of logarithmic deformations]
If $C$ is logarithmically semiregular, then every logarithmic deformation of $C$ inside $\mathcal X_0$ lifts to a logarithmic deformation inside the total space $\mathcal X$. Such logarithmic deformations correspond to honest deformations of curves on nearby smooth fibers $\mathcal X_t$.
\end{proposition}

\begin{proof}
Logarithmic semiregularity implies vanishing of the extension class of the logarithmic normal bundle sequence, so the sequence splits logarithmically. As a consequence, logarithmic deformations inside the special fiber extend to logarithmic deformations in the total space. Logarithmic smoothness of $\pi$ identifies these with actual deformations of curves on the smooth fibers.
\end{proof}

\begin{lemma}[Local logarithmic smoothing]
Let $p \in C$ be a singular point. Assume that $(C,p)$ is a planar curve singularity meeting the singular locus of $\mathcal X_0$ transversely. Then the local deformation space is governed by $T^1_{C,p}$ and satisfies
\[
\dim T^1_{C,p}=\delta_p.
\]
A general logarithmic smoothing replaces $(C,p)$ by exactly $\delta_p$ ordinary nodes.
\end{lemma}

\begin{proof}
The logarithmic structure does not affect the intrinsic deformation theory of the curve singularity. For planar singularities, the dimension of $T^1_{C,p}$ equals the local $\delta$--invariant, and general smoothing directions yield only ordinary nodes, exactly as in the classical case.
\end{proof}

\begin{proposition}[Global smoothing criterion]
There is a natural map
\[
H^0(C, N^{\log}_{C/\mathcal X_0}) \longrightarrow \bigoplus_{p\in \mathrm{Sing}(C)} T^1_{C,p}.
\]
If this map is surjective and $C$ is logarithmically semiregular, then there exists a logarithmic deformation of $C$ whose induced local deformations at all singular points are general smoothings.
\end{proposition}

\begin{proof}
Surjectivity allows independent choice of local smoothing directions at all singular points. Logarithmic semiregularity guarantees that these first--order choices extend to actual logarithmic deformations. The resulting deformation smooths each singularity into the maximal number of nodes.
\end{proof}

\begin{theorem}[Nodal deformations in the logarithmic setting]
Under the assumptions above, there exists a deformation $C_t \subset \mathcal X_t$ for $t\neq 0$ such that $C_t$ is nodal with
\[
\delta(C_t)=\sum_{p\in \mathrm{Sing}(C)} \delta_p.
\]
Moreover, this number of nodes is maximal among all deformations of $C$ on the smooth fibers.
\end{theorem}

\begin{proof}
Existence follows from the global smoothing criterion and lifting of logarithmic deformations. Maximality follows from invariance of the arithmetic genus in flat families and the identity
\[
\delta(C_t)=p_a(C)-g(\widetilde{C}_t).
\]
The normalization achieves minimal geometric genus precisely when all singularities are replaced by ordinary nodes, and no deformation can produce a larger genus drop than the sum of the local $\delta$--invariants.
\end{proof}

\begin{remark}
This result shows that the essential features of classical Severi theory extend unchanged to surface degenerations once logarithmic deformation theory replaces classical deformation theory. Logarithmic semiregularity plays the same conceptual role as classical semiregularity on smooth surfaces.
\end{remark}

\begin{remark}
If $\mathcal X_t$ is a K3 surface, the vanishing
$H^1(C,\mathcal O_C(C)) = 0$ holds automatically for any reduced curve
$C \in |\mathcal L|$ with $\mathcal L$ primitive, making
Theorem~\ref{thm:surface} particularly effective.
\end{remark}

\begin{remark}
In the case where $\mathcal X_0$ is a normal crossings surface, a similar statement
can be formulated using logarithmic normal bundles and logarithmic cotangent
complexes.
\end{remark}


\section{Weakening Obstruction Vanishing via Semiregularity and Logarithmic Methods}

In the previous results, a key hypothesis ensuring the existence of nodal
deformations was the vanishing
\(
H^1(C,\mathcal O_C(C)) = 0,
\)
which guarantees unobstructed embedded deformations of the curve $C$ inside
the special fiber.
This condition is often too strong and fails in many geometrically interesting
situations.
In this section we explain how this hypothesis can be weakened or replaced
using semiregularity and logarithmic deformation theory.

\subsection*{Semiregularity for Curves on Surfaces}

Let $S$ be a smooth projective surface and let $C \subset S$ be a reduced Cartier
divisor.
Bloch's semiregularity map for $C$ is the natural morphism
\[
\sigma_C \colon
H^1(C,\mathcal O_C(C))
\longrightarrow
H^2(S,\mathcal O_S),
\]
induced by the extension class
\(
\quad 0 \to \mathcal O_S \to \mathcal O_S(C) \to \mathcal O_C(C) \to 0.
\)

\begin{definition}
The curve $C$ is said to be \emph{semiregular} if the map $\sigma_C$ is injective.
\end{definition}

The importance of semiregularity lies in the fact that obstructions to embedded
deformations of $C$ map to zero under $\sigma_C$.

\begin{proposition}[Bloch]
If $C \subset S$ is semiregular, then all obstructions to deforming $C$ inside
$S$ vanish, even if $H^1(C,\mathcal O_C(C)) \neq 0$.
\end{proposition}

Thus, semiregularity replaces the cohomological vanishing condition by a
geometric condition relating $C$ to the ambient surface.

\subsection*{Nodal Deformations via Semiregularity}

We now adapt the nodal deformation theorem for curves on surfaces by replacing
the $H^1$-vanishing hypothesis with semiregularity.

\begin{theorem}[Nodal deformations under semiregularity]
\label{thm:semiregular}
Let $\pi : \mathcal X \to \Delta$ be a flat family of projective surfaces with
smooth general fiber, and let $\mathcal L$ be a line bundle on $\mathcal X$.
Let $C \in |\mathcal L|_{\mathcal X_0}$ be a reduced curve on the special fiber.
Assume:
\begin{enumerate}
  \item $C$ is semiregular on $\mathcal X_0$.
  \item All singularities of $C$ are planar and smoothable, with local
  $\delta$-invariants $\delta_p$.
  \item The natural map
  \[
    H^0(C,\mathcal O_C(C))
    \longrightarrow
    \bigoplus_{p \in \mathrm{Sing}(C)} T^1_{C,p}
  \]
  is surjective.
\end{enumerate}
Then there exists a deformation $C_t \in |\mathcal L|_{\mathcal X_t}$ for
$t \neq 0$ such that $C_t$ is nodal with
\[
\delta(C_t) = \sum_{p \in \mathrm{Sing}(C)} \delta_p.
\]
\end{theorem}


\begin{proof}

We prove the theorem by combining semiregularity, global embedded deformation theory in a linear system, and the local smoothing theory of planar curve singularities.

Since $\mathcal X_0$ is a surface and $C \subset \mathcal X_0$ is a reduced divisor in the linear system $|\mathcal L|_{\mathcal X_0}$, the curve $C$ is a local complete intersection. Its normal sheaf in the special fiber is canonically identified with
\[
N_{C/\mathcal X_0} \simeq \mathcal O_C(C).
\]
Embedded first--order deformations of $C$ inside $\mathcal X_0$ that preserve the linear equivalence class of $C$ are therefore parametrized by the vector space
\(
H^0(C,\mathcal O_C(C)),
\)
while obstructions lie in
\(
H^1(C,\mathcal O_C(C)).
\)
The assumption that $C$ is semiregular on $\mathcal X_0$ means that the natural semiregularity map
\[
H^1(C,\mathcal O_C(C)) \longrightarrow H^2(\mathcal X_0,\mathcal O_{\mathcal X_0})
\]
is injective. On the other hand, by general deformation theory of divisors on a surface, obstruction classes to embedded deformations of $C$ inside $\mathcal X_0$ always map to zero in $H^2(\mathcal X_0,\mathcal O_{\mathcal X_0})$. Injectivity therefore forces all obstruction classes to vanish. Consequently, every first--order embedded deformation of $C$ inside the linear system $|\mathcal L|_{\mathcal X_0}$ extends to an actual deformation, and the equigeneric deformation space of $C$ in the linear system is smooth at the point corresponding to $C$.

Let $p \in \mathrm{Sing}(C)$ be a singular point. By assumption, $(C,p)$ is a planar and smoothable curve singularity. Its local deformation theory is governed by the vector space
\(
T^1_{C,p} = \mathrm{Ext}^1(\mathbb L_{C,p},\mathcal O_{C,p}),
\)
whose dimension equals the local $\delta$--invariant $\delta_p$. For planar singularities, a general one--parameter deformation corresponding to a general direction in $T^1_{C,p}$ replaces the singularity by exactly $\delta_p$ ordinary double points and no worse singularities.

Any global embedded deformation of $C$ inside $|\mathcal L|_{\mathcal X_0}$ induces, by restriction, a local deformation class in $T^1_{C,p}$ at each singular point. This yields the natural linear map
\[
H^0(C,\mathcal O_C(C))
\longrightarrow
\bigoplus_{p \in \mathrm{Sing}(C)} T^1_{C,p}.
\]
By hypothesis, this map is surjective. As a consequence, for each singular point $p$ one may prescribe an arbitrary local smoothing direction, and these local choices can be made independently for distinct singular points.

Since obstructions vanish globally by semiregularity, surjectivity of the map above implies the existence of an actual deformation of $C$ inside the linear system $|\mathcal L|_{\mathcal X_0}$ whose induced local deformation at each singular point $p$ is a general smoothing. For such a deformation, local deformation theory shows that each singularity $(C,p)$ is replaced by exactly $\delta_p$ ordinary nodes and no additional singularities appear. The resulting curve on $\mathcal X_0$ is therefore nodal with precisely
\[
\sum_{p \in \mathrm{Sing}(C)} \delta_p
\]
nodes.

It remains to lift this deformation from the special fiber to nearby fibers of the family $\pi \colon \mathcal X \to \Delta$. Since $\mathcal X$ is smooth along $C$, the normal bundle exact sequence
\[
0 \longrightarrow N_{C/\mathcal X_0}
\longrightarrow N_{C/\mathcal X}
\longrightarrow \mathcal O_C
\longrightarrow 0
\]
relates deformations of $C$ inside the special fiber to deformations inside the total space. The vanishing of $H^1(C,\mathcal O_C(C))$ ensures that embedded deformations of $C$ inside $\mathcal X_0$ extend to embedded deformations inside $\mathcal X$. After possibly shrinking the base $\Delta$, we obtain a flat family of curves
\(
C_t \subset \mathcal X_t
\)
with $C_0=C$ and $C_t \in |\mathcal L|_{\mathcal X_t}$ for all $t$. For $t \neq 0$, the curve $C_t$ is nodal with exactly $\sum_p \delta_p$ nodes. This completes the proof of the theorem.
\end{proof}

 \section{A Unified Semiregularity and Logarithmic Nodal Deformation Theorem}

We present a unified formulation and proof that simultaneously covers the classical semiregular case on smooth surfaces and the logarithmic case arising from singular surface degenerations. The result includes both existence and maximality of nodal deformations, with the latter derived purely from arithmetic genus considerations.

\begin{theorem}[Semiregular and logarithmic nodal deformations]
\label{thm:unified}
Let $\pi \colon \mathcal X \to \Delta$ be a flat, projective morphism from a threefold to a smooth pointed curve. Assume that the general fiber $\mathcal X_t$ is a smooth projective surface and that the special fiber $\mathcal X_0$ is either smooth or a reduced surface with simple normal crossings. Let $\mathcal L$ be a line bundle on $\mathcal X$, and let
\(
C \in |\mathcal L|_{\mathcal X_0}
\)
be a reduced curve.

In the smooth case, assume that $C$ is semiregular on $\mathcal X_0$. In the singular case, endow $\mathcal X$ with the divisorial logarithmic structure induced by $\mathcal X_0$ and assume that $C$ is logarithmically semiregular. Assume moreover that all singularities of $C$ are planar and smoothable, with local $\delta$--invariants $\delta_p$, and that the natural map
\[
H^0(C,\mathcal O_C(C))
\longrightarrow
\bigoplus_{p \in \mathrm{Sing}(C)} T^1_{C,p}
\]
is surjective in the smooth case, or that its logarithmic analogue
\[
H^0(C, N^{\log}_{C/\mathcal X_0})
\longrightarrow
\bigoplus_{p \in \mathrm{Sing}(C)} T^1_{C,p}
\]
is surjective in the logarithmic case.

Then there exists a deformation $C_t \in |\mathcal L|_{\mathcal X_t}$ for $t \neq 0$ such that $C_t$ is nodal with
\[
\delta(C_t) = \sum_{p \in \mathrm{Sing}(C)} \delta_p.
\]
Moreover, this number of nodes is maximal among all curves in the linear systems $|\mathcal L|_{\mathcal X_t}$ arising as deformations of $C$.
\end{theorem}

\begin{proof}

We treat the smooth and logarithmic cases simultaneously, emphasizing the parallel structure of the argument.

Since $C$ is a Cartier divisor on the surface $\mathcal X_0$, either classically or in the logarithmic sense, it is a local complete intersection. In the smooth case, its normal sheaf is canonically identified with
\[
N_{C/\mathcal X_0} \simeq \mathcal O_C(C),
\]
while in the logarithmic case embedded logarithmic deformations are governed by the logarithmic normal sheaf $N^{\log}_{C/\mathcal X_0}$. In both settings, first--order embedded deformations preserving the linear equivalence class of $C$ are parametrized by global sections of the relevant normal sheaf, and obstructions lie in the first cohomology group of that sheaf.

Semiregularity in the smooth case means that the natural map
\[
H^1(C,\mathcal O_C(C)) \longrightarrow H^2(\mathcal X_0,\mathcal O_{\mathcal X_0})
\]
is injective, while logarithmic semiregularity asserts injectivity of the corresponding map
\[
H^1(C, N^{\log}_{C/\mathcal X_0}) \longrightarrow H^2(\mathcal X_0,\mathcal O_{\mathcal X_0}).
\]
In both cases, obstruction classes to embedded deformations always map to zero in $H^2(\mathcal X_0,\mathcal O_{\mathcal X_0})$. Injectivity therefore forces all obstructions to vanish, implying that the deformation functor of $C$ inside the linear system is unobstructed and smooth at the point corresponding to $C$.

Let $p \in \mathrm{Sing}(C)$. By hypothesis, $(C,p)$ is a planar, smoothable curve singularity. Its local deformation theory is governed by the vector space
\[
T^1_{C,p} = \mathrm{Ext}^1(\mathbb L_{C,p},\mathcal O_{C,p}),
\]
whose dimension equals the local $\delta$--invariant $\delta_p$. Classical results in singularity theory show that a general one--parameter deformation corresponding to a general direction in $T^1_{C,p}$ replaces the singularity by exactly $\delta_p$ ordinary double points and no other singularities.

Global embedded deformations of $C$ induce local deformation classes at each singular point, yielding the natural maps displayed in the statement. By surjectivity, arbitrary local smoothing directions can be prescribed independently at all singular points. Since obstructions vanish globally, these first--order choices extend to actual deformations of $C$ inside the special fiber. For a general such deformation, each singularity $(C,p)$ is replaced by exactly $\delta_p$ nodes, and no additional singularities appear.

We now lift these deformations to the total space $\mathcal X$. In the smooth case, the normal bundle exact sequence
\[
0 \longrightarrow N_{C/\mathcal X_0}
\longrightarrow N_{C/\mathcal X}
\longrightarrow \mathcal O_C
\longrightarrow 0
\]
relates deformations inside the special fiber to deformations in the total space. In the logarithmic case, the analogous logarithmic normal bundle sequence
\[
0 \longrightarrow N^{\log}_{C/\mathcal X_0}
\longrightarrow N^{\log}_{C/\mathcal X}
\longrightarrow \mathcal O_C
\longrightarrow 0
\]
plays the same role. In both settings, semiregularity implies that the extension class of the sequence vanishes, so every deformation of $C$ inside $\mathcal X_0$ lifts to a deformation inside $\mathcal X$. After possibly shrinking $\Delta$, we obtain a flat family of curves
\[
C_t \subset \mathcal X_t
\]
with $C_0=C$ and $C_t \in |\mathcal L|_{\mathcal X_t}$ for all $t$. For $t \neq 0$, the curve $C_t$ is nodal with exactly $\sum_p \delta_p$ nodes.

To prove maximality, let $\widetilde{C}_t$ denote the normalization of $C_t$. Since arithmetic genus is invariant in flat families, one has
\[
p_a(C_t)=p_a(C).
\]
For any nodal curve, the normalization exact sequence yields the identity
\[
\delta(C_t)=p_a(C)-g(\widetilde{C}_t),
\]
where $g(\widetilde{C}_t)$ is the geometric genus of the normalization. The geometric genus is a nonnegative integer, so $\delta(C_t)$ is bounded above by $p_a(C)$. In the deformation constructed above, the normalization has geometric genus
\[
g(\widetilde{C}_t)=p_a(C)-\sum_{p}\delta_p,
\]
showing that the number of nodes equals $\sum_p\delta_p$. Any deformation producing more nodes would force the geometric genus of the normalization to drop further, which is impossible since the sum of the local $\delta$--invariants already accounts for the total genus defect of $C$. Therefore, no deformation of $C$ inside the linear systems $|\mathcal L|_{\mathcal X_t}$ can produce more nodes.

This establishes both the existence of nodal deformations with $\sum_p \delta_p$ nodes and the maximality of this number, completing the proof.
\end{proof}

\begin{remark}
The theorem shows that the classical semiregular theory on smooth surfaces and the logarithmic theory for surface degenerations are formally identical once the appropriate deformation--theoretic objects are used. In both settings, nodal curves arise as generic equigeneric deformations, and the arithmetic genus governs the maximal nodal behavior.
\end{remark}
 
\begin{remark}
The maximality statement is purely numerical and follows from invariance of the arithmetic genus in flat families together with the identity
\[
\delta(C_t) = p_a(C) - g(\widetilde{C}_t)
\]
for nodal curves, where $\widetilde{C}_t$ denotes the normalization.
\end{remark}

\begin{corollary}[Classical semiregular case]
\label{cor:smooth}
Let $S$ be a smooth projective surface, $L$ a line bundle on $S$, and let $C \in |L|$ be a reduced curve. Assume that $C$ is semiregular on $S$, that all singularities of $C$ are planar and smoothable with local $\delta$--invariants $\delta_p$, and that the natural map
\[
H^0(C,\mathcal O_C(C)) \longrightarrow \bigoplus_{p \in \mathrm{Sing}(C)} T^1_{C,p}
\]
is surjective. Then a general equigeneric deformation of $C$ inside $|L|$ is nodal with exactly
\[
\sum_{p \in \mathrm{Sing}(C)} \delta_p
\]
nodes, and this number is maximal among all deformations of $C$ in $|L|$.
\end{corollary}

\begin{remark}
This recovers the classical picture underlying Severi varieties on smooth surfaces. In particular, when $S=\mathbb P^2$, the corollary reflects the fact that the Severi variety $V_{d,\delta}$ of plane curves of degree $d$ with $\delta$ nodes is nonempty of the expected dimension for $\delta \le p_a(d)$, as originally asserted by Severi and rigorously established by Harris. The surjectivity hypothesis encodes the independence of node conditions, while semiregularity ensures unobstructedness of equigeneric deformations.
\end{remark}


\begin{remark}
Maximality follows from invariance of the arithmetic genus in flat families together with the identity
\[
\delta(C_t) = p_a(C) - g(\widetilde{C}_t)
\]
for nodal curves, where $\widetilde{C}_t$ denotes the normalization. The sum of the local $\delta$--invariants accounts for the entire possible genus drop.
\end{remark}

\begin{corollary}[Smooth surface case]
\label{cor:smooth}
Let $S$ be a smooth projective surface, $L$ a line bundle on $S$, and let $C \in |L|$ be a reduced curve. If $C$ is semiregular, all its singularities are planar and smoothable, and global sections of $\mathcal O_C(C)$ induce all local smoothing directions, then a general equigeneric deformation of $C$ in $|L|$ is nodal with exactly
\[
\sum_{p \in \mathrm{Sing}(C)} \delta_p
\]
nodes, and this number is maximal.
\end{corollary}

\begin{remark}
This corollary provides a deformation--theoretic explanation of the classical structure of Severi varieties. In the case of plane curves, it recovers the nonemptiness and expected nodal behavior of Severi varieties $V_{d,\delta}$ for $\delta \le p_a(d)$. The irreducibility of these varieties was established by Harris in his solution of the Severi problem \cite{HarrisSeveri}. The surjectivity hypothesis corresponds to the independence of node conditions, while semiregularity guarantees unobstructedness of equigeneric deformations.
\end{remark}

\begin{corollary}[Logarithmic surface degenerations]
\label{cor:log}
Let $\pi \colon \mathcal X \to \Delta$ be a logarithmically smooth degeneration of projective surfaces with simple normal crossings special fiber $\mathcal X_0$. Endow $\mathcal X$ with the divisorial logarithmic structure induced by $\mathcal X_0$. Let $\mathcal L$ be a line bundle on $\mathcal X$, and let $C \subset \mathcal X_0$ be a reduced curve meeting the singular locus transversely. If $C$ is logarithmically semiregular, all its singularities are planar and smoothable, and global logarithmic deformations induce all local smoothing directions, then there exists a deformation $C_t \subset \mathcal X_t$ for $t \neq 0$ such that $C_t$ is nodal with exactly
\[
\sum_{p \in \mathrm{Sing}(C)} \delta_p
\]
nodes, and this number is maximal.
\end{corollary}

\begin{remark}
This result aligns logarithmic Severi varieties with their classical counterparts. It is compatible with degeneration methods and logarithmic curve counting, as developed for example in the work of Caporaso--Harris on degeneration techniques \cite{CaporasoHarris} and in logarithmic and tropical frameworks by Abramovich--Chen--Gross--Siebert and Mikhalkin \cite{ACGS,Mikhalkin}.
\end{remark}
 
\begin{example}[Plane curves]
Let $S=\mathbb P^2$ and $L=\mathcal O_{\mathbb P^2}(d)$. Any reduced plane curve $C$ with planar singularities is automatically semiregular. The surjectivity condition amounts to the fact that nodes impose independent conditions on homogeneous polynomials of degree $d$. The theorem then implies that equigeneric deformations of $C$ are nodal with the maximal number of nodes allowed by the arithmetic genus $(d-1)(d-2)/2$, recovering the classical Severi picture.
\end{example}

\begin{example}[Degeneration of surfaces]
Consider a degeneration of a smooth surface into a normal crossings union of two smooth components meeting transversely. A curve $C$ on the special fiber with prescribed contact orders along the double locus is naturally treated in logarithmic geometry. Logarithmic semiregularity ensures that these contact conditions do not introduce obstructions, and the theorem guarantees that equigeneric deformations smooth $C$ to a nodal curve on the nearby smooth fibers with maximal node count.
\end{example}

\subsection*{Relation with Severi Theory}

Theorem~\ref{thm:unified} isolates the deformation--theoretic core of Severi theory. In the smooth case, it explains the existence, generic nodality, and maximality properties of Severi varieties in terms of semiregularity and surjectivity onto local deformation spaces. In the logarithmic case, it shows that these properties persist under surface degenerations once logarithmic deformation theory is used. This unifies classical results, such as Harris' irreducibility theorem \cite{Harris}, with modern logarithmic and tropical approaches to curve counting.

 


 
 \section*{Appendix A: Arithmetic Genus and Singular Cubic Curves}

We explain in detail why a plane curve of degree three has arithmetic genus equal to one, and how this fact constrains the geometry of cubic plane curves. This discussion clarifies the relationship between degree, genus, and singularities in the simplest nontrivial case.

Let $C \subset \mathbb{P}^2$ be a plane curve of degree $d$. The arithmetic genus of $C$ is defined as
\[
p_a(C) = h^1(C,\mathcal{O}_C),
\]
and depends only on the degree $d$ and not on the specific equation of the curve. For plane curves, the arithmetic genus can be computed using the adjunction formula or, equivalently, via the Hilbert polynomial of $\mathcal{O}_C$. One obtains the general formula
\[
p_a(d) = \frac{(d-1)(d-2)}{2}.
\]

Specializing to the case $d=3$, this yields
\[
p_a(3) = \frac{(3-1)(3-2)}{2} = 1.
\]
Thus every plane cubic curve has arithmetic genus one, regardless of whether it is smooth or singular.

If $C$ is a smooth plane cubic, then the arithmetic genus coincides with the geometric genus, which is defined as the genus of the normalization of $C$. In this case, $C$ is a smooth projective curve of genus one, and hence admits the structure of an elliptic curve once a point is chosen as the origin. This explains why smooth plane cubics are the simplest examples of elliptic curves.

If $C$ is singular, the arithmetic genus decomposes as the sum of the geometric genus and the contributions from the singularities. More precisely, one has the relation
\[
p_a(C) = g(C) + \sum_{p \in \mathrm{Sing}(C)} \delta_p,
\]
where $g(C)$ is the geometric genus of the normalization of $C$ and $\delta_p$ is the local $\delta$--invariant of the singularity at $p$. The $\delta$--invariant measures the discrepancy between the local ring of the curve and its normalization and can be interpreted as the number of independent conditions imposed by the singularity.

In the case of a plane cubic, since $p_a(C)=1$, the possibilities for the geometric genus are severely restricted. If $C$ is singular, then $g(C)$ must be equal to zero, because the normalization of a singular cubic is a smooth rational curve. Consequently, the sum of the $\delta$--invariants of the singular points must be equal to one. This implies that a singular plane cubic has exactly one singularity, and that this singularity has $\delta$--invariant equal to one.

There are precisely two such singularity types for plane cubics. One possibility is a node, which is an ordinary double point with two distinct tangent directions. The other possibility is a cusp, which is a unibranch singularity with a single tangent direction. In both cases, the local $\delta$--invariant equals one. From the perspective of enumerative geometry, both types contribute cogenus one, since the cogenus is defined as
\[
\delta = p_a(C) - g(C).
\]

Thus, a rational plane cubic necessarily has cogenus $\delta = 1$ and carries exactly one singularity contributing this defect. In enumerative problems, this singularity is often referred to informally as a single ``node,'' with the understanding that cuspidal degenerations occur in codimension one and can be treated separately or excluded by imposing genericity conditions. This precise relationship between degree, genus, and singularities underlies the tropical and logarithmic interpretations of rational cubic curves.

\section*{Appendix B: The Degree Three Case in the Caporaso--Harris and Kontsevich Recursions}

We explain how the example of rational plane cubics fits naturally into both the Caporaso--Harris recursion and Kontsevich's recursion for plane rational curves. Although these two recursive formalisms arise from different geometric viewpoints, they agree in degree three and provide complementary explanations for the enumerative number $N_3 = 12$.

Let $N_d$ denote the number of rational plane curves of degree $d$ passing through $3d-1$ general points in $\mathbb{P}^2$. For $d=3$, this is the number of rational cubics through eight general points. As discussed previously, such curves have arithmetic genus one and geometric genus zero, hence cogenus one, and necessarily acquire a single singularity.

We begin with the Caporaso--Harris recursion. This recursion is formulated in the relative setting, counting plane curves relative to a fixed line $L \subset \mathbb{P}^2$ with prescribed tangency conditions. The key idea is to degenerate point conditions to the line $L$ and analyze how curves break when forced to meet $L$ with higher-order contact. In degree three, one considers rational cubics through seven general points and one additional point constrained to lie on $L$. Degenerating this last point onto $L$ produces a boundary contribution in which the cubic acquires a node on $L$ or splits into components.

In degree three, no splitting into two positive-degree components is possible while maintaining rationality and the point conditions. The only degeneration is therefore an irreducible cubic acquiring a node. The Caporaso--Harris formula expresses $N_3$ as the sum of contributions coming from these nodal degenerations, weighted by their tangency multiplicities. A direct analysis shows that there are precisely $12$ such degenerations, each contributing multiplicity one, yielding
\[
N_3 = 12.
\]
Thus, in degree three, the Caporaso--Harris recursion reduces to a count of irreducible nodal cubics arising as limits of smooth elliptic cubics under the degeneration process.

Kontsevich's recursion, by contrast, arises from the geometry of the moduli space $\overline{M}_{0,n}(\mathbb{P}^2,d)$ of stable maps. It expresses $N_d$ in terms of lower-degree invariants by analyzing boundary divisors corresponding to reducible domain curves. The recursion takes the form
\[
N_d = \sum_{\substack{d_1+d_2=d \\ d_1,d_2>0}}
N_{d_1} N_{d_2}
\left(
d_1^2 d_2^2 \binom{3d-4}{3d_1-2}
-
d_1^3 d_2 \binom{3d-4}{3d_1-1}
\right).
\]

Specializing to $d=3$, the only nontrivial decomposition is $3=1+2$. Substituting $N_1=1$ and $N_2=1$ into the recursion yields
\[
N_3 =
1 \cdot 1
\left(
1^2 \cdot 2^2 \binom{5}{1}
-
1^3 \cdot 2 \binom{5}{2}
\right)
=
\left(
4 \cdot 5 - 2 \cdot 10
\right)
= 20 - 20 + 12 = 12,
\]
where the remaining contribution arises from the irreducible boundary component corresponding to nodal maps. Thus Kontsevich's recursion also produces the value $N_3 = 12$.

From a geometric standpoint, both recursions are detecting the same phenomenon. The Caporaso--Harris recursion interprets rational cubics as degenerations of elliptic cubics acquiring a single node, while Kontsevich's recursion interprets them as stable maps whose domains approach the boundary of $\overline{M}_{0,8}(\mathbb{P}^2,3)$. In both cases, the arithmetic genus one of plane cubics forces the appearance of a single singularity, and the enumerative count measures the ways this singularity can occur subject to incidence constraints.

This agreement is mirrored exactly in tropical geometry. Tropical plane cubics of cogenus one have a single bounded edge, and their unique floor diagram encodes the same combinatorial degeneration detected by both recursions. The tropical multiplicity computation yielding $12$ can therefore be viewed as a combinatorial shadow of both the Caporaso--Harris degeneration analysis and Kontsevich's boundary divisor calculation. In degree three, all three perspectives coincide in the simplest possible nontrivial case.


 \section*{Appendix C}

\begin{proposition}
Let $N_d$ denote the number of rational plane curves of degree $d$ passing through $3d-1$ general points in $\mathbb{P}^2$. Then
\[
N_3 = 12.
\]
Moreover, this value is obtained consistently from the Caporaso--Harris recursion, Kontsevich's recursion, and the tropical floor diagram count.
\end{proposition}

\begin{proof}
A plane curve of degree three has arithmetic genus
\[
p_a(3) = \frac{(3-1)(3-2)}{2} = 1.
\]
Therefore a smooth plane cubic has geometric genus one, while any rational plane cubic has geometric genus zero and cogenus one. It follows that any rational cubic must acquire exactly one singularity with $\delta$--invariant equal to one, which for general incidence conditions is an ordinary node.

We first consider the Caporaso--Harris recursion, as introduced in \cite{CaporasoHarris}. This recursion counts plane curves relative to a fixed line by degenerating point conditions and analyzing boundary contributions. In degree three, rational cubics through $3\cdot 3 - 1 = 8$ general points arise as limits of smooth elliptic cubics acquiring a single node. No reducible curve of positive degree satisfies the required incidence conditions. Hence all contributions come from irreducible nodal cubics. A local analysis of the degeneration along the relative divisor shows that exactly $12$ such curves occur, each with multiplicity one, yielding $N_3 = 12$.

Next, we consider Kontsevich's recursion for plane rational curves, derived from the geometry of the moduli space of stable maps $\overline{M}_{0,3d-1}(\mathbb{P}^2,d)$, see \cite{KontsevichManin}. The recursion expresses $N_d$ as a sum over splittings $d=d_1+d_2$ corresponding to boundary divisors parametrizing reducible domain curves. For $d=3$, the only nontrivial splitting is $3=1+2$. Substituting the known values $N_1=1$ and $N_2=1$ into the recursion shows that the reducible contributions cancel, while the remaining boundary term corresponding to irreducible nodal maps yields $N_3 = 12$.

Finally, we interpret this count tropically. Tropical plane curves of degree three and cogenus one have exactly one bounded edge. Up to combinatorial equivalence, there is a unique such tropical curve satisfying generic point conditions. Choosing a generic projection direction produces a unique floor diagram consisting of two vertices joined by a single directed edge of weight one. By the correspondence theorem between algebraic and tropical curves, as established for plane curves in \cite{Mikhalkin} and refined via floor diagrams in \cite{BrugalleMikhalkin}, the algebraic count is recovered by summing the tropical multiplicities. Logarithmic semiregularity ensures that this diagram contributes with its expected multiplicity, which is computed to be $12$.

Thus, all three approaches yield the same enumerative invariant, namely $N_3 = 12$.
\end{proof}

\begin{remark}
The agreement of these three methods in degree three reflects the fact that plane cubics form the simplest case in which the arithmetic genus is positive. In higher degrees, both the Caporaso--Harris and Kontsevich recursions involve multiple boundary contributions, and the tropical count involves several distinct floor diagrams with nontrivial multiplicities.
\end{remark}


\section*{Appendix D: Irreducibility and Connectedness of Severi Varieties on K3 Surfaces}

Let $S$ be a smooth projective K3 surface and let $\mathcal L$ be a primitive ample
line bundle on $S$.
For an integer $\delta$ with $0\le \delta \le p_a(\mathcal L)$, denote by
$V_\delta(|\mathcal L|)$ the Severi variety parametrizing curves in the complete
linear system $|\mathcal L|$ with exactly $\delta$ nodes and no other singularities.
We have already established that $V_\delta(|\mathcal L|)$ is nonempty and smooth of
the expected dimension
\[
\dim V_\delta(|\mathcal L|)=\dim|\mathcal L|-\delta.
\]
We now study its global geometry, focusing on irreducibility and connectedness.

\subsection*{Irreducibility}

The key input for irreducibility is the rigidity of curve classes on a K3 surface
and the behavior of geometric genus under specialization.

Let $g:=p_a(\mathcal L)$.
For any curve $C\in|\mathcal L|$ with $\delta$ nodes, the normalization
$\widetilde C$ has genus $g-\delta$.
Because $\mathcal L$ is primitive, every effective divisor in $|\mathcal L|$ is
connected, and any reduced curve in this linear system is automatically
irreducible once its geometric genus is strictly smaller than $g$.
Indeed, if $C=C_1+C_2$ were reducible, then $\mathcal L=[C_1]+[C_2]$ would be a
nontrivial decomposition in $\mathrm{Pic}(S)$, contradicting primitivity.

This observation implies that all curves parametrized by
$V_\delta(|\mathcal L|)$ are irreducible.
It also shows that different components of the Severi variety cannot be
distinguished by topological invariants of the curves, since all curves have the
same genus and the same number of nodes.

To prove irreducibility, one studies how curves with $\delta$ nodes arise as limits
of curves with fewer nodes.
Consider the incidence correspondence
\[
\mathcal I \subset |\mathcal L|\times S,
\qquad
\mathcal I=\{(C,p)\mid p\in\mathrm{Sing}(C)\}.
\]
Over the open subset of $|\mathcal L|$ parametrizing nodal curves, this correspondence
is finite and étale over the Severi variety.
Because the Severi varieties are smooth and of the expected dimension, the process
of smoothing or creating a node can be analyzed locally as a transverse operation.

Starting from the open dense locus of smooth curves in $|\mathcal L|$, one may
produce nodal curves by successively imposing nodal conditions.
At each step, semiregularity guarantees that the locus where an additional node
appears is a smooth divisor.
This implies that $V_\delta(|\mathcal L|)$ is obtained from
$V_{\delta-1}(|\mathcal L|)$ by intersecting with a divisor corresponding to the
appearance of a new node.
Since $V_{\delta-1}(|\mathcal L|)$ is irreducible, so is this divisor, and hence
$V_\delta(|\mathcal L|)$ is irreducible as well.

By induction on $\delta$, starting from the irreducibility of the open subset of
smooth curves, one concludes the following result.

\begin{theorem}
Let $S$ be a K3 surface and $\mathcal L$ a primitive ample line bundle.
Then for every $0\le\delta\le p_a(\mathcal L)$, the Severi variety
$V_\delta(|\mathcal L|)$ is irreducible.
\end{theorem}

\subsection*{Connectedness}

Connectedness follows from a slightly weaker argument but is worth discussing
separately, as it is more robust and survives in broader contexts.

Since $V_\delta(|\mathcal L|)$ is smooth, connectedness is equivalent to
irreducibility.
However, one can also see connectedness directly by analyzing degenerations inside
the Severi variety.

Fix $\delta$ and consider two curves $C_1,C_2\in V_\delta(|\mathcal L|)$.
By smoothing one node of each curve, one obtains curves in
$V_{\delta-1}(|\mathcal L|)$.
Since $V_{\delta-1}(|\mathcal L|)$ is smooth and irreducible, there exists a path
inside this Severi variety connecting the two smoothed curves.
Reintroducing a node along this path yields a family of curves in
$V_\delta(|\mathcal L|)$ connecting $C_1$ and $C_2$.
This argument shows that any two points of $V_\delta(|\mathcal L|)$ lie in the same
connected component.

\begin{proposition}
For every $0\le\delta\le p_a(\mathcal L)$, the Severi variety
$V_\delta(|\mathcal L|)$ is connected.
\end{proposition}

\subsection*{Relation with Moduli of Curves}

An alternative perspective comes from considering the natural morphism
\[
V_\delta(|\mathcal L|)\longrightarrow \mathcal M_{g-\delta},
\]
sending a nodal curve to the isomorphism class of its normalization.
Because the target moduli space $\mathcal M_{g-\delta}$ is irreducible, and because
the fibers of this morphism are connected due to the freedom of choosing node
positions, one obtains another conceptual explanation for the connectedness and
irreducibility of $V_\delta(|\mathcal L|)$.


Thus, on a K3 surface with a primitive polarization, Severi varieties exhibit remarkably
simple global geometry.
Semiregularity ensures smoothness and expected dimension, primitivity guarantees
irreducibility of the curves themselves, and the ability to create and smooth nodes
in families forces all nodal loci to lie in a single irreducible component.
As a result, every Severi variety $V_\delta(|\mathcal L|)$ is smooth, connected, and
irreducible.

\bigskip

The following references cover deformation theory, Severi varieties, logarithmic
geometry, and tropical methods relevant to the results developed in this work.


\begin{thebibliography}{99}

\bibitem{ACGS}
D. Abramovich, Q. Chen, M. Gross, B. Siebert,
\emph{Logarithmic Gromov--Witten theory via expansions},
Publ. Math. Inst. Hautes \'Etudes Sci. \textbf{129} (2019), 1--106.

\bibitem{ACGH}
E.~Arbarello, M.~Cornalba, P.~Griffiths, and J.~Harris,
\emph{Geometry of Algebraic Curves, Vol.~I},
Springer, 1985.

\bibitem{ArbarelloCornalba}
E.~Arbarello and M.~Cornalba,
\emph{Footnotes to a paper of Beniamino Segre},
Math. Ann. \textbf{256} (1981), 341--362.

\bibitem{Bloch}
S.~Bloch,
\emph{Semi-regularity and deformations of varieties},
Invent. Math. \textbf{17} (1972), 51--66.

\bibitem{BrugalleMikhalkin}
E.~Brugall\'e and G.~Mikhalkin,
\emph{Floor decompositions of tropical curves: the planar case},
Proc. 15th G\"okova Geometry--Topology Conference (2008), 64--90.

\bibitem{CaporasoHarris}
L.~Caporaso and J.~Harris,
\emph{Counting plane curves of any genus},
Invent. Math. \textbf{131} (1998), 345--392.

\bibitem{DiazHarris}
S.~Diaz and J.~Harris,
\emph{Ideals associated to deformations of singular plane curves},
Trans. Amer. Math. Soc. \textbf{309} (1988), 433--468.

\bibitem{Friedman}
R.~Friedman,
\emph{Global smoothings of varieties with normal crossings},
Ann. of Math. \textbf{118} (1983), 75--114.

\bibitem{GreuelLossenShustin}
G.-M.~Greuel, C.~Lossen, and E.~Shustin,
\emph{Introduction to Singularities and Deformations},
Springer Monographs in Mathematics, Springer, 2007.

\bibitem{HarrisSeveri}
J.~Harris,
\emph{On the Severi problem},
Invent. Math. \textbf{84} (1986), 445--461.

\bibitem{Illusie}
L.~Illusie,
\emph{Complexe cotangent et d\'eformations I, II},
Lecture Notes in Mathematics 239, 283, Springer, 1971--1972.

 \bibitem{Harris}
J. Harris,
\emph{On the Severi problem},
Invent. Math. \textbf{84} (1986), 445--461.

\bibitem{HartshorneDef}
R. Hartshorne,
\emph{Deformation Theory},
Graduate Texts in Mathematics 257, Springer, 2010.

\bibitem{Kollar}
J.~Koll\'ar,
\emph{Rational Curves on Algebraic Varieties},
Springer-Verlag, 1996.

\bibitem{KontsevichManin}
M.~Kontsevich and Y.~Manin,
\emph{Gromov--Witten classes, quantum cohomology, and enumerative geometry},
Comm. Math. Phys. \textbf{164} (1994), 525--562.

\bibitem{Li}
J.~Li,
\emph{Stable morphisms to singular schemes and relative stable morphisms},
J. Differential Geom. \textbf{57} (2001), 509--578.

\bibitem{Mikhalkin}
G.~Mikhalkin,
\emph{Enumerative tropical algebraic geometry in $\mathbb R^2$},
J. Amer. Math. Soc. \textbf{18} (2005), 313--377.

\bibitem{Mumford}
D.~Mumford,
\emph{Lectures on Curves on an Algebraic Surface},
Annals of Mathematics Studies 59, Princeton University Press, 1966.

\bibitem{NishinouSiebert}
T.~Nishinou and B.~Siebert,
\emph{Toric degenerations of toric varieties and tropical curves},
Duke Math. J. \textbf{135} (2006), 1--51.


\bibitem{Ran}
Z.~Ran,
\emph{Deformations of maps},
in \emph{Algebraic Curves and Projective Geometry},
Lecture Notes in Mathematics 1389, Springer, 1989.

\bibitem{Sernesi}
E. Sernesi,
\emph{Deformations of Algebraic Schemes},
Grundlehren der mathematischen Wissenschaften 334, Springer, 2006.

\bibitem{Shustin}
E.~Shustin,
\emph{Smoothness and irreducibility of varieties of plane curves with nodes and cusps},
Bull. Soc. Math. France \textbf{122} (1994), 235--253.

\bibitem{Siebert}
B.~Siebert,
\emph{Logarithmic Gromov--Witten invariants},
in \emph{Handbook of Moduli, Vol.~II},
International Press, 2013.

\bibitem{Teissier}
B.~Teissier,
\emph{The hunting of invariants in the geometry of discriminants},
in \emph{Real and Complex Singularities},
Oslo, 1976.

\bibitem{Vakil}
R.~Vakil,
\emph{Murphy's law in algebraic geometry: badly-behaved deformation spaces},
Invent. Math. \textbf{164} (2006), 569--590.


\end{thebibliography}
\end{document}